\documentclass[12pt,a4paper]{amsart}
\usepackage[latin1]{inputenc}
\usepackage[english]{babel}
\usepackage[T1]{fontenc}
\usepackage{lmodern} 
\usepackage{amssymb,amsmath,amsthm,amsfonts,enumerate}
\usepackage[autostyle]{csquotes}
\usepackage{alphalph,cite}
\usepackage{nameref}
\usepackage{verbatim}
\usepackage{mathrsfs}
\usepackage{manfnt}
\usepackage{pifont}
\usepackage{yfonts}
\usepackage[perpage]{footmisc}
\usepackage{enumitem}
\usepackage[all]{xy} 
\usepackage{wasysym}

\usepackage{stmaryrd}

\usepackage{apptools}
\usepackage{stackengine}
\usepackage{scalerel}
\usepackage{array}
\usepackage{datetime}
\usepackage{xparse} 
\usepackage{tikz-cd}
\usepackage{rotating}
\usepackage{esvect}
\usepackage[colorlinks=true, linkcolor={blue}, citecolor={magenta}, urlcolor={green}]{hyperref}
\usepackage{mathtools}
\usepackage{stackengine}
\usepackage{chessfss}
\usepackage{caption}
\usepackage{makecell}
\usepackage{float}
\usepackage{diagbox}
\usepackage{orcidlink}

\makeatletter
\@namedef{subjclassname@2010}{ \textup{2010} Mathematics Subject Classification}
\makeatother
\theoremstyle{definition}
\newtheorem{theorem}{Theorem}[subsection]

\newtheorem{theoremletters}{Theorem}

\renewcommand{\thefootnote}{\arabic{footnote}}

\theoremstyle{definition}
\newtheorem{lemma}[theorem]{Lemma}
\theoremstyle{definition}
\newtheorem{corollary}[theorem]{Corollary}
\theoremstyle{definition}
\newtheorem{proposition}[theorem]{Proposition}
\theoremstyle{definition}

\theoremstyle{definition}
\newtheorem{remark}[theorem]{Remark}
\theoremstyle{definition}
\newtheorem{example}[theorem]{Example}
\theoremstyle{definition}
\newtheorem{question}[theorem]{Question}
\theoremstyle{definition}
\newtheorem{definition}[theorem]{Definition}
\numberwithin{equation}{section}

\numberwithin{equation}{section}
\theoremstyle{definition}

\theoremstyle{definition}
\newtheorem{notation}[theorem]{Notation}
\theoremstyle{definition}

\theoremstyle{definition}

\theoremstyle{definition}

\theoremstyle{definition}

\theoremstyle{definition}

\theoremstyle{definition}

\theoremstyle{definition}

\theoremstyle{definition}
\newtheorem{construction}[theorem]{Construction}
\theoremstyle{definition}
\newtheorem{digression}[theorem]{Digression}

\newcommand{\norm}[1]{\left\lVert#1\right\rVert}
\newcommand{\vertiii}[1]{{\left\vert\kern-0.25ex\left\vert\kern-0.25ex\left\vert #1 
		\right\vert\kern-0.25ex\right\vert\kern-0.25ex\right\vert}}
\newcommand{\abs}[1]{\left\lvert#1\right\rvert}

\newcommand\xrowht[2][0]{\addstackgap[.5\dimexpr#2\relax]{\vphantom{#1}}}

\renewcommand{\thefootnote}{\alph{footnote}}

\makeatletter
\def\author@andify{
	\nxandlist {\unskip ,\penalty-1 \space\ignorespaces}
	{\unskip {} \@@and~}
	{\unskip \penalty-2 \space \@@and~}
}
\makeatother

\newcommand\xqed[1]{%
	\leavevmode\unskip\penalty9999 \hbox{}\nobreak\hfill
	\quad\hbox{#1}}
\newcommand\demo{\xqed{{\Large $\blacktriangle$}}}

\makeatletter
\def\l@subsection{\@tocline{2}{0pt}{1pc}{5pc}{}} \def\l@subsection{\@tocline{2}{0pt}{2pc}{6pc}{}}
\makeatother

\begin{document}
	\title{Direct sums and abstract Kadets--Klee properties}
	
	\author{Tomasz Kiwerski}
	\address[Tomasz Kiwerski \orcidlink{0000-0002-8349-8785}]{Poznan University of Technology, Institute of Mathematics, Piotrowo 3A, 60-965 Pozna\'{n} (Poland)}
	\email{\href{mailto:tomasz.kiwerski@gmail.com}{\tt tomasz.kiwerski@gmail.com}}
	
	\author{Pawe{\l} Kolwicz}
	\address[Pawe{\l} Kolwicz \orcidlink{0000-0002-6587-8135}]{Poznan University of Technology, Institute of Mathematics, Piotrowo 3A, 60-965 Pozna\'{n} (Poland)}
	\email{\href{mailto:pawel.kolwicz@put.poznan.pl}{\tt pawel.kolwicz@put.poznan.pl}}
	
	\maketitle
	
	\begin{abstract}
		Let $\mathcal{X} = \{ X_{\gamma} \}_{\gamma \in \Gamma}$ be a family of Banach spaces and let $\mathcal{E}$ be a Banach sequence space defined on $\Gamma$.
		The main aim of this work is to investigate the abstract Kadets--Klee properties, that is, the Kadets--Klee type properties in which the weak
		convergence of sequences is replaced by the convergence with respect to some linear Hausdorff topology, for the direct sum construction
		$(\bigoplus_{\gamma \in \Gamma} X_{\gamma})_{\mathcal{E}}$.
		As we will show, and this seems to be quite atypical behavior when compared to some other geometric properties, to lift the Kadets--Klee properties
		from the components to whole direct sum it is not enough to assume that all involved spaces have the appropriate Kadets--Klee property.
		Actually, to complete the picture one must add a dichotomy in the form of the Schur type properties for $X_{\gamma}$'s supplemented by the variant
		of strict monotonicity for $\mathcal{E}$.
		Back down to earth, this general machinery naturally provides a blue print for other topologies like, for example, the weak topology or the topology
		of local convergence in measure, that are perhaps more commonly associated with this type of considerations.
		Furthermore, by limiting ourselves to direct sums in which the family $\mathcal{X}$ is constant, that is, $X_{\gamma} = X$ for all $\gamma \in \Gamma$
		and some Banach space $X$, we return to the well-explored ground of K{\" o}the--Bochner sequence spaces $\mathcal{E}(X)$.
		Doing all this, we will reproduce, but sometimes also improve, essentially all existing results about the classical Kadets--Klee properties
		in K{\" o}the--Bochner sequence spaces.
	\end{abstract}
	
	\tableofcontents
	
	\renewcommand{\thefootnote}{\fnsymbol{footnote}}\footnotetext[0]{
		{\it Date:} \longdate{\today}.
		
		
	}
	
	\section{{\bf Introduction}} \label{SECTION: Introduction}
	
	\subsection{The property ${\bf H}$}
	
	To set up the scene, let us recall that a Banach space $X$ is said to have the {\bf Kadets--Klee property}
	(hereinafter, we will sometimes abbreviate this by just saying that the space $X$ has the {\bf property $\mathbf{H}$};
	{\it cf}. Section~\ref{SUBSECTION : recollections KKP} and Appendix~\ref{APPENDIX : C}) if for sequences on the unit sphere
	of $X$ the weak topology and the norm topology agree.
	
	Among the most conspicuous examples, it is perfectly clear that any Banach space with the Schur property also has the property ${\bf H}$.
	Moreover, but this is slightly less obvious, the same can be said about Banach spaces which are uniformly convex
	(this observation is due to {\v S}mulian \cite{Smu39}; cf. \cite[Theorem~5.2.18, p.~453]{Meg98} and \cite[Theorem~5.3.7, p.~463]{Meg98}).
	However, although uniformly convex spaces are necessarily reflexive, it turns out that many classical non-reflexive spaces like,
	for example, $\bullet$ the Hardy space $H_1$; $\bullet$ the trace class $\mathscr{C}_1$; $\bullet$ and any member of the family
	of Lorentz spaces $\{ L_{p,1} \}_{1 < p < \infty}$, all share the Kadets-Klee property.
	
	Thus, from this perspective, the Kadets--Klee property can be seen as a non-reflexive analogue of uniform convexity.
	This comparison is actually quite memorable, because it would be not an exaggeration to say that among all the knots of a seemingly
	endless tangle of various geometric properties of Banach spaces considered throughout the 20th century,
	it is difficult to choose a more classical and recognizable one than the uniform convexity.
	It is therefore not particularly surprising that so much interest has been devoted to research on Kadets--Klee properties
	in various Banach spaces
	(see, for example,
	$\bullet$ \cite{CDSS96} and \cite{Suk95} for the non-commutative setting;
	$\bullet$ \cite{CKP15} and \cite{Kol12} for rearrangement invariant spaces;
	$\bullet$ \cite{DK99} for the space of bounded linear operators between $\ell_p$ and $\ell_q$;
	$\bullet$ \cite{DDSS04} and \cite{DDSS04a} for some interpolation spaces;
	$\bullet$ \cite{DHLMS03} and \cite{MS81} for Orlicz spaces;
	$\bullet$ and \cite{GP05} for JB$^{*}$-triples).
	
	Here, however, we will be mainly interested in the so-called {\bf abstract Kadets--Klee properties} ${\bf H}(\mathfrak{T})$.
	By this we will mean the ordinary version of the Kadec--Klee property ${\bf H}$ in which the weak topology is replaced by an
	arbitrary linear Hausdorff topology $\mathfrak{T}$ (see Definition~\ref{DEF: Kadets--Klee property}).
	There are quite a few papers devoted to the study of some special instances of the property ${\bf H}(\mathfrak{T})$,
	where $\mathfrak{T}$ is, for example,
	$\bullet$ the weak$^{*}$ topology (see \cite{BDDL94}, \cite{DK99} and \cite{Len91});
	$\bullet$ the topology generated by some order ideals (see \cite{CDSS96});
	$\bullet$ the topology of local convergence in measure (see \cite{CKP15}, \cite{DHLMS03}, \cite{FH99}, \cite{FHS10}, \cite{Len91} and \cite{Suk95});
	$\bullet$ or the topology of global convergence in measure (see \cite{CKP15} and \cite{Suk95}).
	Nevertheless, the selection of works dealing with the property ${\bf H}(\mathfrak{T})$ in full\footnote{Even if this means assuming the local
	convexity of $\mathfrak{T}$, or that the unit ball $\text{Ball}(X)$ is sequentially $\mathfrak{T}$-closed, {\it etc}.} generality is much more
	modest (see, for example, \cite{DDSS04}, \cite{DDSS04a} and \cite{Len91}).
	The fairly obvious reason for this state of affairs is the bewildering array of available topologies, which may generally have little in common.
	In order not to be unfounded, let $X$ be a Banach function space over a complete measure space $(\Omega,\Sigma,\mu)$
	and let us consider three topologies, namely,
	$\bullet$ the weak topology;
	$\bullet$ the topology of local convergence in measure;
	$\bullet$ and the topology of global convergence in measure,
	defined on $X$.
	This situation is tabulated in the following table.
	\begin{center}
		\begin{table}[H]
			\begin{tabular}{ | c | c | c | c | } 
				\hline \xrowht{15pt}
				\diagbox[width=10em]{{\bf property}}{{\bf topology}}	& {\bf weak topology} & {\bf \makecell{topology of \\ local convergence \\ in measure}} & {\bf \makecell{topology of \\ global convergence \\ in measure}} \\
				\hline \xrowht{15pt}
				{\bf linear} & {\it yes} & {\it yes} & {\it no} (see (A)) \\
				\hline \xrowht{15pt}
				{\bf Hausdorff} & {\it yes} & {\it no} (see (B)) & {\it no} (see (B)) \\
				\hline \xrowht{15pt}
				{\bf \makecell{coarser \\ than the norm \\ topology}} & {\it yes} & {\it yes} & {\it no} (see (C)) \\
				\hline \xrowht{15pt}
				{\bf locally convex} & {\it yes} & {\it no}  & {\it no} \\
				\hline \xrowht{15pt}
				{\bf locally solid} & {\it no} (see (D)) & {\it yes}  & {\it yes} \\
				\hline
			\end{tabular}
			\captionof{table}{Summary of the properties of the three most important topologies appearing in the context of the Kadets--Klee properties} \label{TABLE : trzy H}
		\end{table}
	\end{center}
	Hereby,
	\begin{itemize}
		\item[(A)] Under some additional assumptions (for example, if $\mu(\Omega)$ is finite; see \cite[245Y(e), p.~183]{Fre01}).
		\item[(B)] if, and only if, $(\Omega,\Sigma,\mu)$ is semi-finite (see \cite[245E, p.~176]{Fre01}).
		\item[(C)] Under some additional assumptions (for example, if $X$ is a separable rearrangement invariant space; {\it cf}. \cite[Theorem~4.1]{CDSS96}).
		\item[(D)] if, and only if, $X$ is finite dimensional (see \cite[Theorem~6.9, p.~42]{AB03}).
	\end{itemize}
	
	\subsection{Clich{\' e}s from K{\" o}the--Bochner space theory}
	
	Since the main actors for the undertaken considerations will be vector-valued Banach spaces in the form of the so-called
	{\bf $\mathcal{E}$-direct sums} $(\bigoplus_{\gamma \in \Gamma} X_{\gamma})_{\mathcal{E}}$ (see Definition~\ref{DEFINITION: direct sum} for details),
	the themes and problems typical for the thoroughly developed theory of K{\" o}the--Bochner spaces will provide a motivation
	and inspiration for our research.
	A simple reason for this fact comes from the observation that the space $(\bigoplus_{\gamma \in \Gamma} X )_{\mathcal{E}}$ coincide,
	up to the equality of norms, with the K{\" o}the--Bochner construction $\mathcal{E}(X)$.
	Therefore, it seems natural to briefly recall the most important problems that, in a sense, give shape to this theory.
	
	Let $X$ be a Banach space, and $E$ be a Banach function space over a complete measure space $(\Omega,\Sigma,\mu)$.
	For the sake of completeness, let us recall following \cite[Chapter~3]{Lin04} that the {\bf K{\" o}the--Bochner space} $E(X)$
	is understood as a vector space of all strongly measurable functions $f \colon \Omega \rightarrow X$ such that $\omega \mapsto \norm{f(\omega)}_X$
	belongs to $E$.
	As usual, we furnish the space $E(X)$ with the norm $\norm{f}_{E(X)} \coloneqq \norm{\omega \mapsto \norm{f(\omega)}_X}_E$.
	
	\vspace{10pt} \noindent
	{\bf Lifting Problem.}
	{\it Let ${\bf P}$ be a certain \enquote{geometric} property.
		Does the K{\" o}the--Bochner construction $E(X)$ have the property ${\bf P}$ provided both spaces $E$ and $X$ have the same property ${\bf P}$?
		If no, under what extra conditions on $E$ and, perhaps, $X$ does the space $E(X)$ have the property ${\bf P}$?}
	\vspace{10pt} \noindent
	
	Without a doubt, Lifting Problem is the most fundamental and prolific among the problems regarding the structure of K{\" o}the--Bochner spaces.
	Indeed, for ${\bf P}$ being
	$\bullet$ the Dunford--Pettis property;
	$\bullet$ the Radon--Nikodym property;
	$\bullet$ the Kadets--Klee property;
	$\bullet$ rotundity;
	$\bullet$ strict monotonicity;
	$\bullet$ uniform convexity;
	$\bullet$ uniform monotonicity;
	$\bullet$ smoothness; to name just a few, this problem has been analyzed exhaustively by a multitude of authors over many decades
	(see, for example,
	\cite{Boa40}, \cite{CP96}, \cite{CHM96}, \cite{Day41}, \cite{DPS07}, \cite{DK16}, \cite{HL92}, \cite{Kol03}, \cite{KP97}, \cite{KL92}, \cite{LL85}, \cite{MP22}, \cite{ST80} and \cite{Leo76};
	we refer also to Lin's monograph \cite{Lin04} for much more comprehensive discussion).
	All this gives rise to the following
	
	\begin{notation}
		We will say that the geometric property ${\bf P}$ is {\bf stable} with respect to the K{\" o}the--Bochner construction
		if the property ${\bf P}$ lifts from $E$ and $X$ to $E(X)$.
	\end{notation}

	Of course, formally it also makes sense to consider the following problem which can be seen as a reverse version of Lifting Problem
	
	\vspace{10pt} \noindent
	{\bf Inheritance Problem.}
	{\it Let ${\bf P}$ be a certain \enquote{geometric} property.
		Does the fact that the K{\" o}the--Bochner space $E(X)$ has the property ${\bf P}$ imply that both spaces $E$ and $X$ also have the
		property ${\bf P}$?}
	\vspace{10pt} \noindent
	
	Note, that for many geometric properties Inheritance Problem has the trivial solution.
	This is due to the plain fact that the space $E(X)$ contains a complemented subspaces isometrically isomorphic to $E$ and $X$ ({\it cf}. \cite[p.~178]{Lin04}).
	In any case, this is not so obvious for the property ${\bf H}(\mathfrak{T})$ (hence Definition~\ref{DEF: admissible topology}). 
	
	\begin{notation}
		We will say that the geometric property ${\bf P}$ is {\bf hereditary} with respect to the K{\" o}the--Bochner construction
		provided both spaces $E$ and $X$ inherit the property ${\bf P}$ from the space $E(X)$.
	\end{notation}
	
	\subsection{Goals}
	
	This paper takes up the related question which, after what we have already said above, can be concisely formulated as follows
	
	\begin{question} \label{QUESTION : MAIN}
		{\it Under what assumptions the abstract Kadets-Klee property ${\bf H}(\mathfrak{T})$ is stable or hereditary with respect
		to the direct sum construction $(\bigoplus_{\gamma \in \Gamma} X_{\gamma})_{\mathcal{E}}$?}
	\end{question}

	Although our primary goal is basically to answer Question~\ref{QUESTION : MAIN}, achieving this requires some intermediate steps
	(which appear to be the subject of an independent interest):
	\begin{itemize}
		\item[(G1)] Analysis of the direct sums construction $(\bigoplus_{\gamma \in \Gamma} X_{\gamma})_{\mathcal{E}}$
		(see Section~\ref{direct sums});
		\item[(G2)] Collecting and organizing certain facts about the property ${\bf H}(\mathfrak{T})$ for Banach sequence spaces
		(see Section~\ref{SUBSECTION : recollections KKP} and Appendix~\ref{APPENDIX : A});
		\item[(G3)] Introducing the concept of $\oplus$-compatible topologies and studying an abstract framework
		(see Section~\ref{SUBSECTION : abstract framework}; {\it cf}. Section~\ref{SECTION : Play Doh});
		\item[(G4)] Deducing characterizations of classical Kadets--Klee properties in the special case of K{\" o}the--Bochner spaces
		and comparing them with existing results (see Section~\ref{SECTION : EXAMPLES});
		\item[(G5)] Proposing directions of development for future research (see Section~\ref{SECTION : Open ends}).
	\end{itemize}
	
	\subsection{Overview of main results}
	
	This outline should be considered only as a relatively non-technical indication of our main results from Section~\ref{SUBSECTION : abstract framework}.
	Moreover, some notation and details used below may differ slightly from what we will present later.
	
	Throughout this section, let us fix the following data:
	$\bullet$ $\{ X_{\gamma} \}_{\gamma \in \Gamma}$ is a family of Banach spaces equipped with a linear Hausdorff
	topologies $\mathfrak{T}_{\gamma}$ coarser than the corresponding norm topologies on $X_{\gamma}$'s;
	$\bullet$ $\mathcal{E}$ is a Banach sequence space on $\Gamma$ equipped with a linear Hausdorff topology
	$\mathfrak{T}_{\mathcal{E}}$ coarser than the norm topology on $\mathcal{E}$ (we refer to Section~\ref{SUBSECTION : BSS} for details)
	$\bullet$ $\mathfrak{T}$ is the $\oplus$-compatible topology on the space $(\bigoplus_{\gamma \in \Gamma} X_{\gamma})_{\mathcal{E}}$
	(see Definition~\ref{DEF: admissible topology}).
	Let us also consider the following two conditions:
	\begin{itemize}
		\item[($\rook$)] {\it the set $\Gamma$ can be decomposed into two disjoint subsets, say $\Gamma_1$ and $\Gamma_2$,
			in such a way that all $X_{\gamma}$'s with $\gamma \in \Gamma_1$ have the Schur property with respect to $\mathfrak{T}_{\gamma}$
			(see Definition~\ref{DEF : abstract Schur})
			and $\mathcal{E}$ is ${\bf SM(\gamma)}$ for $\gamma \in \Gamma_2$ (see Definition~\ref{DEF : SM(gamma)});}
	\end{itemize}
	and
	\begin{itemize}
		\item[($\bishop$)] {\it the mapping $\{ x_{\gamma }\}_{\gamma \in \Gamma} \mapsto \sum_{\gamma \in \Gamma} \norm{x_{\gamma}}_{\gamma} \boldsymbol{e}_{\gamma}$
		is $\mathfrak{T}$-to-$\mathfrak{T}_{\mathcal{E}}$ sequentially continuous when acting from the unit sphere in $(\bigoplus_{\gamma \in \Gamma} X_{\gamma})_{\mathcal{E}}$
		into $\mathcal{E}$.}
	\end{itemize}
	Note that ($\rook$) is nothing else but the condition (2) from Theorem~\ref{THM: H is hereditary},
	while ($\bishop$) is the assumption (A3) from Theorem~\ref{THM: main theorem}.
	
	After this short preparation, we are ready to formulate our two main results.
	
	\begin{theoremletters}[Inheritance Problem; see Theorem~\ref{THM: H is hereditary}] \label{THEOREM : IP}
		{\it The property ${\bf H}(\mathfrak{T})$ is hereditary and} ($\rook$) {\it holds}.
	\end{theoremletters}

	\begin{theoremletters}[Lifting Problem; see Theorem~\ref{THM: main theorem}] \label{THEOREM : LP}
		{\it Under the assumption} ($\bishop$), {\it the property ${\bf H}(\mathfrak{T})$ is stable.}
	\end{theoremletters}

	Anyway, since ($\rook$) and ($\bishop$) have little in common at first glance, the usefulness of both Theorem~\ref{THEOREM : IP}
	and Theorem~\ref{THEOREM : LP} is mainly due to the following somewhat surprising observation.
	
	\begin{theoremletters}[Compatibility result; see Theorem~\ref{THM : oba warunki sa rownowazne}] \label{THEOREM : CR}
		{\it Under some mild but essentially technical assumptions, both properties} ($\rook$) {\it and } ($\bishop$) {\it coincide.}
	\end{theoremletters}
	
	From here, there is a direct path to application in concrete situations
	(we will how to do this in Section~\ref{SECTION : EXAMPLES}; see Theorem~\ref{THM : H weak} and Theorem~\ref{THEOREM : H(measure)}).
	
	\subsection{Outline}
	
	Let us now briefly describe the organization of this work.
	
	Overall, the paper is divided into six sections excluding three appendices and bibliography.
    And so, in Section~\ref{SECTION : Toolbox} we will recall the necessary terminology and some indispensable facts that will be useful later.
    Next part, that is, Section~\ref{SECTION : Main results}, constitutes the main part of the work and presents a general approach to the abstract
    Kadets--Klee properties in direct sums.
    Section~\ref{SECTION : Play Doh} is devoted to the study of $\oplus$-compatible topologies.
    Results from Section~\ref{SECTION : Main results} will be applied in Section~\ref{SECTION : EXAMPLES} which is dedicated to classical Kadets-Klee
    properties with respect to the weak topology and the topology of local convergence in measure.
    Finally, in Section~\ref{SECTION : Open ends} we include some open questions that refer to the previous sections.
    The whole work is bound together with Appendices~\ref{APPENDIX : A}, \ref{APPENDIX : B} and \ref{APPENDIX : C}.
     
	\subsection{Acknowledgments}
	
	The first named author is grateful to Jakub Tomaszewski for several enlightening discussions.
	
	The research of Tomasz Kiwerski was supported by Pozna{\' n} University of Technology: Grant number 0213/SBAD/0120.
	The research of Pawe{\l} Kolwicz was also supported by Pozna{\' n} University of Technology: Grant number 0213/SBAD/0119.
	
	\section{{\bf Toolbox}} \label{SECTION : Toolbox}
	
	In this section we will provide a handbook of notation (Section~\ref{SUBSECTION : Notation}) and then we collect some terminology and basic results about:
	$\bullet$ Banach sequence spaces (Section~\ref{SUBSECTION : BSS});
	$\bullet$ some rotundity and monotonicity properties (Section~\ref{SUBSECTION : GP});
	$\bullet$ and direct sums of families of Banach spaces (Section~\ref{direct sums}).
	Other, possibly unfamiliar, definitions and concepts will be introduced in the sections they are used.
	
	\subsection{Notation} \label{SUBSECTION : Notation}
	
	As usual, $\mathbb{N} = \{1,2, ...\}$, $\mathbb{Z}$ and $\mathbb{R}$ denote the set of natural numbers, integers and reals, respectively.

	\begin{notation}
		To lighten the notation at times, we will briefly denote the fact that the sequence {\it $x_n$ converges
		to $x$ in the topology $\mathfrak{T}$} as $x_n \overset{\mathfrak{T}}{\rightarrow} x$ or $\mathfrak{T}$-$\lim_{n \rightarrow \infty} x_n = x$.
	\end{notation}

	\begin{notation}
		Interchangeably, depending on the notational convenience, we will denote a vector-valued function $x \colon \Gamma \rightarrow X$,
		where $\Gamma$ is a countable set and $X$ is a Banach space, as $\{ x(\gamma) \}_{\gamma \in \Gamma}$ or $\{ x_\gamma \}_{\gamma \in \Gamma}$.
		This should not lead to any confusion, as it should always be clear from the context what exactly we mean.
	\end{notation}
	
	\begin{notation}
		Defined terms are usually distinguished from the rest of text by {\bf bold} font.
	\end{notation}
	
	\begin{notation}
		By tradition, we will use \enquote{halmos} $\blacksquare$ at the end of each proof.
		However, we find it useful to also use the symbol {\Large $\blacktriangle$} at the end of each remark, digression and example.
	\end{notation}

	\begin{notation}[Unit vector basis]
		Let $c_{00}(\Gamma)$ be the space of sequences of real numbers all but finitely many of which are zero.
		We shall let $\{ \boldsymbol{e}_{\gamma} \}_{\gamma \in \Gamma}$ stand for the unit vector basis of this vector space.
		In other words,
		\begin{equation*}
			\boldsymbol{e}_{\gamma} \coloneqq (0,..., 0, \underbrace{1}_{\gamma^{\text{th}} \text{ position}}, 0, ...).
		\end{equation*}
		Moreover, we will use the following notation
		\begin{equation*}
			x \otimes \boldsymbol{e}_{\gamma} \coloneqq (0, ..., 0, \underbrace{x}_{\gamma^{\text{th}} \text{ position}}, 0, ...).
		\end{equation*}
	\end{notation}

	\begin{notation}
		Unless explicitly stated otherwise, by a Banach space $X$ we will always understand an infinite-dimensional real Banach space,
		while by a subspace of a Banach space $X$ we will mean a closed linear subspace of $X$.
	\end{notation}

	\begin{notation}[Duality bracket]
		If $X$ is a Banach space, we write $X^{*}$ for the topological dual space of $X$
		and $\text{Ball}(X)$ for the closed unit ball of $X$.
		We use brackets $\langle \tiny\bullet, \tiny\bullet \rangle$ for the pairing between Banach space and its dual,
		that is, for $x \in X$ and $x^{*} \in X^{*}$ we shall write $\langle x, x^{*} \rangle \coloneqq x^{*}(x)$
		for the action of $x^{*}$ on $x$.
		We write $(X,\textit{weak})$ for the topological vector space obtained by equipping
		the Banach space $X$ with its weak topology.
	\end{notation}
	
\subsection{Banach sequence spaces} \label{SUBSECTION : BSS}

Let $(\Gamma, 2^{\Gamma},\#)$, where $\Gamma$ is a countable set and $\#$ is a counting measure on $\Gamma$, that is,
$\#(A) \coloneqq \sum_{\gamma \in \Gamma} \delta_{\gamma}(A)$ for $A \subset \Gamma$ (here, $\delta_{\gamma}$ is the Dirac 
delta concentrated at $\gamma \in \Gamma$), be a purely atomic measure space.
Further, let $\omega(\Gamma)$, briefly just $\omega$, be the set of all real-valued functions defined on $2^{\Gamma}$.
We equip the space $\omega$ with the topology of point-wise convergence, that is, the topology of convergence in measure on sets of finite measure.
This makes $\omega$ an $F$-space.

A Banach space $X$ is called a {\bf Banach sequence space} (or, using another nomenclature, a {\bf K{\" o}the sequence space})
if the following three conditions hold:
\begin{itemize}
	\item[(1)] $X$ is a linear subspace of $\omega(\Gamma)$;
	\item[(2)] for any finite set $F \subset \Gamma$ the characteristic function ${\bf 1}_F$ belongs to $X$;
	\item[(3)] if $\abs{x(\gamma)} \leqslant \abs{y(\gamma)}$ for all $\gamma \in \Gamma$ and $\sum_{\gamma \in \Gamma} y(\gamma) \boldsymbol{e}_{\gamma} \in X$,
	then $\sum_{\gamma \in \Gamma} x(\gamma) \boldsymbol{e}_{\gamma} \in X$ belongs to $X$ and
	$\lVert \sum_{\gamma \in \Gamma} x(\gamma) \boldsymbol{e}_{\gamma} \rVert \leqslant \Vert \sum_{\gamma \in \Gamma} y(\gamma) \boldsymbol{e}_{\gamma} \rVert$
	(the so-called {\bf ideal property}).
\end{itemize}

Due to the closed graph theorem, a formal inclusion of two Banach sequence spaces $X$ and $Y$ is a continuous operator,
that is, the quantity $\norm{\text{id} \colon X \rightarrow Y} \coloneqq \sup \{ \norm{x}_Y \colon \norm{x}_X = 1 \}$ is finite.
To clearly emphasize this fact we will sometimes write $X \hookrightarrow Y$ instead of just $X \subset Y$.
We will use the symbol $X = Y$ to indicate that the spaces $X$ and $Y$ are the same as vector spaces and their norms are equivalent,
that is, $X \hookrightarrow Y$ and $Y \hookrightarrow X$.

A function $x\in X$ from a Banach sequence space $X$ is said to be {\bf order continuous} (or has an {\bf order continuous norm}) if,
for any sequence $\{ x_{n} \}_{n=1}^{\infty}$ of positive and disjoint functions from $X$ that is order bounded by  $|x| $
and converges point-wisely to zero, it follows that $\{ x_{n} \}_{n=1}^{\infty}$ is norm null sequence. By $X_o$ we denote a closed subspace of all order continuous sequences from $X$.
We will say that the space $X$ is {\bf order continuous} if $X = X_o$.
Equivalently, the space $X$ is order continuous provided for each $x\in X$ and for any decreasing sequence $\{ \Gamma_n \}_{n=1}^{\infty}$ of subsets of $\Gamma$
with empty intersection such that $\#(\Gamma \setminus \Gamma_n)$ is finite, one has
\begin{equation*}
   \lim_{n \rightarrow \infty} \lVert \sum_{\gamma \in  \Gamma_n} x(\gamma) \boldsymbol{e}_{\gamma} \rVert = 0. 
\end{equation*}

Since purely atomic measure spaces with at most countable number of atoms are separable, so a Banach sequence space $X$
is order continuous if, and only if, it is separable (see \cite[Theorem~5.5, p.~27]{BS88}).

By the {\bf K{\" o}the dual} $X^{\times}$ of a given Banach sequence space $X$ we will understand a vector space all sequences
$\{ x(\gamma) \}_{\gamma \in \Gamma}$ such that $\sum_{\gamma \in \Gamma} \abs{x(\gamma) y(\gamma)}$ is finite for all $\{ y(\gamma) \}_{\gamma \in \Gamma} \in X$
equipped with the norm
\begin{equation*}
   \lVert \sum_{\gamma \in \Gamma} x(\gamma) \boldsymbol{e}_{\gamma} \rVert_{X^{\times}}
\coloneqq
\sup \{ \sum_{\gamma \in \Gamma} \abs{x(\gamma) y(\gamma)} \colon \lVert \sum_{\gamma \in \Gamma} y(\gamma) \boldsymbol{e}_{\gamma} \rVert_X \leqslant 1 \}. 
\end{equation*}

Recall that $X = X^{\times \times}$ if, and only if, the norm in $X$ has the {\bf Fatou property}, that is,
for any increasing sequence $\{ x_n \}_{n=1}^{\infty}$ of non-negative functions from $X$ that converges point-wisely to $x$
and $\sup \{ \norm{x_n} \colon n \in \mathbb{N} \}$ is finite, it follows that $x$ belongs to $X$ and $\norm{x} = \sup \{ \norm{x_n} \colon n \in \mathbb{N} \}$. 

Given a separable Banach sequence space $X$, its K{\" o}the dual $X^{\times}$ can be naturally identified with the topological dual $X^{*}$,
that is, the space of all continuous linear forms on $X$ (see \cite[Corollary~4.3, p.~23]{BS88}).
Moreover, a Banach sequence space $X$ with the Fatou property is reflexive if, and only if, both $X$ and $X^{\times}$ are separable
(see \cite[Corollary~4.4, p.~23]{BS88}).

We refer to the books by Bennett and Sharpley \cite{BS88}, Lindenstrauss and Tzafriri \cite{LT77}, \cite{LT79},
Luxemburg and Zaanen \cite{LZ65}, Kantorovich and Akilov \cite{KA82} for a comprehensive study of Banach sequence and function spaces.
Much more information about order continuity property offers Wnuk's monograph \cite{Wn99} (see also \cite{Con19} and \cite[Section~345]{Fre01}).
The standard reference for the basic theory of Banach spaces is, for example, Albiac and Kalton \cite{AK06} and Megginson \cite{Meg98}.

Furthermore, at some points, we will use some facts about topological vector spaces and Banach lattices.
For the general theory of locally convex spaces we refer to the books by Grothendieck \cite{Gro73} and Jarchow \cite{Jar81}.
For the general theory of abstract Banach lattices and, even more generally, Riesz spaces we recommend taking a look at
Aliprantis and Burkinshaw \cite{AB03} and Meyer-Nieberg \cite{MN91} (see also \cite[Chapters~24 and 35]{Fre01} and \cite[Chapter~1]{LT79}).
		
\subsection{Geometric properties} \label{SUBSECTION : GP}
	
We will briefly recall here the most important convexity and monotonicity properties of normed spaces and, respectively,
normed lattices, which we will use later.

Let $X$ be a normed space.
The space $X$ is said to be {\bf rotund} (briefly, the space $X$ is ${\bf R}$) if $\norm{x+y} < 2$ whenever $x$ and $y$
are different points in $\text{Ball}(X)$.
Moreover, $X$ is said to be {\bf locally uniformly rotund} (briefly, the space $X$ is ${\bf LUR}$) if, for any $x \in \text{Ball}(X)$
and $\varepsilon > 0$, there is $\delta = \delta(x,\varepsilon) > 0$ such that for any $y \in \text{Ball}(X)$
the inequality $\norm{x+y} \geqslant \varepsilon$ imply that $\norm{x+y} \leqslant 2(1+\delta)$.
Equivalently, the space $X$ is ${\bf LUR}$, whenever $\{x_n\}_{n=1}^{\infty}$ and $x$ are in $\text{Ball}(X)$
and $\norm{x + x_n} \rightarrow 2$, it follows that $\norm{x - x_n} \rightarrow 0$.
Finally, the space $X$ is {\bf uniformly rotund} (briefly, the space $X$ is ${\bf UR}$) if, for every $\varepsilon > 0$,
there is $\delta = \delta(\varepsilon) > 0$ such that $\norm{x+y} \leqslant 2(1+\delta)$, whenever $x,y \in \text{Ball}(X)$
and $\norm{x-y} > \varepsilon$.
Equivalently, the space $X$ is ${\bf UR}$, whenever $\{x_n\}_{n=1}^{\infty}$ and $\{y_n\}_{n=1}^{\infty}$ are sequences
in $\text{Ball}(X)$ and $\norm{x_n + y_n} \rightarrow 2$, it follows that $\norm{x_n - y_n} \rightarrow 0$.
It is clear that every uniformly rotund normed space is locally uniformly rotund, and every locally uniformly rotund normed
space is rotund.
Pictographically, ${\bf UR} \Rightarrow {\bf LUR} \Rightarrow {\bf R}$.

Now, let $X$ be a normed lattice. The space $X$ is called {\bf strictly monotone} (briefly, the space $X$ is ${\bf SM}$)
if, for any two different elements $x$ and $y$ from $X$ such that $0 \leqslant y \leqslant x$, we have $\norm{x} < \norm{y}$.
Furthermore, $X$ is said to be {\bf uniformly monotone} (briefly, the space $X$ is ${\bf UM}$) if, for every $\varepsilon > 0$,
there is $\delta = \delta(\varepsilon) > 0$ such that $\norm{x - y} \leqslant 1 - \delta$, whenever $0 \leqslant y \leqslant x$,
$x \in \text{Ball}(X)$ and $\norm{y} \geqslant \varepsilon$.
Clearly, ${\bf UM} \Rightarrow {\bf SM}$.

It is also known that, when restricted to the couples of compatible and non-negative elements, monotonicity properties
are equivalent to the corresponding convexity properties.
And so, it follows from \cite[Theorem~1]{HKM00}, that if the {\bf positive cone} $X_+ \coloneqq \{ x \in X \colon x \geqslant 0 \}$
is ${\bf R}$ or ${\bf UR}$, then the space $X$ is ${\bf SM}$ or, respectively, ${\bf UM}$.
In symbols, ${\bf R} \Rightarrow {\bf SM}$ and ${\bf UR} \Rightarrow {\bf UM}$.

Diagrammatically speaking, the above discussion can be summarized as follows
\begin{equation*}
	\begin{tikzcd}
		{\bf UR} \arrow[d, Rightarrow] \arrow[r, Rightarrow] & {\bf LUR} \arrow[r, Rightarrow] & {\bf R} \arrow[d, Rightarrow] \\
		{\bf UM} \arrow[rr, Rightarrow]          &               & {\bf SM}         
	\end{tikzcd}
\end{equation*}

Later we will need the following \enquote{localized} version of the ${\bf SM}$ property.
	
\begin{definition} \label{DEF : SM(gamma)}
	{\it Let $X$ be a Banach sequence space defined on $\Gamma$. Fix $\gamma_0 \in \Gamma$.
	The space $X$ is said to be {\bf strictly monotone on $\gamma_0^{\text{th}}$ coordinate} (briefly, the space $X$ is ${\bf SM}(\gamma_0)$)
	if, for any $x, y \in X_+$ such that $x(\gamma) \leqslant y(\gamma)$ for $\gamma \in \Gamma \setminus \{ \gamma_0 \}$
	and $x(\gamma_0) < y(\gamma_0)$, it follows that
	$\lVert x \rVert < \lVert y \rVert$.}
\end{definition}

Of course, a Banach sequence space $X$ defined on $\Gamma$ is ${\bf SM}$ if, and only if,
$X$ is ${\bf SM}(\gamma)$ for each $\gamma \in \Gamma$.
Slightly less obvious is the following geometric lemma, which was mentioned without detailed proof in \cite{DPS07}.
Since it will play a certain role also in our considerations, let us complete this little detail.
	
	\begin{lemma} \label{LEMMMA : geometric lemma SM}
		{\it Let $X$ be a Banach sequence space defined on $\Gamma$.
		Then the space $X$ is ${\bf SM}(\gamma_0)$ for some $\gamma_0 \in \Gamma$ if, and only if,
		whenever $x, y \in X_+$ are such that $x(\gamma) = y(\gamma)$ for $\gamma \in \Gamma \setminus \{ \gamma_0 \}$
		and $x(\gamma_0) = 0 < y(\gamma_0)$, it follows that
		$\lVert x \rVert < \lVert y \rVert$.}
	\end{lemma}
	\begin{proof}
		One implication is obvious, so let us focus on the second one.
		To do this, suppose that the space $X$ is not ${\bf SM}(\gamma_0)$ for some $\gamma_0 \in \Gamma$.
		This means that there are $x$ and $y$ in $X_+$ such that $x(\gamma) \leqslant y(\gamma)$
		for $\gamma \in \Gamma \setminus \{ \gamma_0 \}$ and $x(\gamma_0) < y(\gamma_0)$,
		but $\lVert x \rVert = \lVert y \rVert = 1$.
		Clearly, we can assume that $x(\gamma_0) > 0$, because otherwise there is nothing to prove. 
       	Let
            \begin{equation*}
               z\left( \gamma\right) \coloneqq \left\{ 
               \begin{array}{ccc}
                    x \left( \gamma\right) & \text{if} & \gamma\neq \gamma_0 \\ 
                    y \left( \gamma_0\right) & \text{if} & \gamma=\gamma_0.
               \end{array}
                     \right.
             \end{equation*}
      	We have $x \leqslant z\leqslant y$, whence $\left\Vert z \right\Vert = 1$.
      	Take 
                  \begin{equation*}
                   \widetilde{z} \left( \gamma\right) \coloneqq \left\{ 
                        \begin{array}{ccc}
                         x \left( \gamma\right) & \text{if} & \gamma\neq \gamma_0 \\ 
                           0 & \text{if} & \gamma=\gamma_0.
                          \end{array}
                           \right.
                     \end{equation*}
        Plainly, $\widetilde{z}(\gamma) = z(\gamma)$ for $\gamma \neq \gamma_0$ with $\widetilde{z}(\gamma_0) = 0$ and $z(\gamma_0) > 0$.
        Then, due to the ideal property of $X$, $\left\Vert \widetilde{z} \right\Vert \leqslant \left\Vert z \right\Vert$.
        We claim that $\norm{\widetilde{z}} = \norm{z}$.
        Suppose, for the sake of contradiction, that $\left\Vert \widetilde{z} \right\Vert < \norm{z}$.
        Take $0 < \lambda < 1$ with $\lambda y(\gamma_0) = x(\gamma_0)$.
	    Invoking the primordial wisdom that balls in normed spaces are convex, we have
        \begin{equation*}
        	\norm{\lambda z + (1 - \lambda) \widetilde{z}}
        	\leqslant \lambda \norm{z} + (1 - \lambda) \norm{\widetilde{z}}
        	< 1.
        \end{equation*}
    	On the other hand,
    	\begin{equation*}
    		\norm{\lambda z + (1 - \lambda) \widetilde{z}}
    			= \lVert \sum_{\gamma \neq \gamma_0} x(\gamma) \boldsymbol{e}_{\gamma} + x(\gamma_0) \boldsymbol{e}_{\gamma_0} \rVert
    			= \norm{x}
    			= 1,
    	\end{equation*}
    	which is obviously nonsense.
        Therefore, $\left\Vert\widetilde{z}\right\Vert = \norm{z}$ and our claim follows.
	\end{proof}
	
	For much more information and relationships between various rotundity and monotonicity properties,
	we refer to the papers \cite{CHKM98}, \cite{HKM00}, Megginson's book \cite{Meg98}, Lin's monograph \cite[Chapter~2]{Lin04}
	and references therein.
	
	\subsection{Direct sums} \label{direct sums}
	
	Recall the following definition, which is a direct generalization of the well-known construction
	of the $\ell_p$-direct sum of a family of Banach spaces (see, for example, \cite[pp.~59--70]{Meg98};
	{\it cf}. \cite[p.~211]{Pel60} and \cite[Example~E, p.~27]{Wn99}).
	
	\begin{definition}[Direct sum] \label{DEFINITION: direct sum}
		{\it Let $\mathcal{X} = \{ X_{\gamma} \}_{\gamma \in \Gamma}$ be a countable family of Banach spaces.
		Moreover, let $\mathcal{E}$ be a Banach sequence space over $\Gamma$.
		By the {\bf $\mathcal{E}$-direct sum} of the family $\mathcal{X}$ we will understand here a vector space
		\begin{equation*}
			( \bigoplus_{\gamma \in \Gamma} X_{\gamma} )_{\mathcal{E}}
				\coloneqq \left\{ \{ x(\gamma) \}_{\gamma \in \Gamma} \in \prod_{\gamma \in \Gamma} X_{\gamma}
				\colon \sum_{\gamma \in \Gamma} \norm{ x(\gamma) }_{\gamma} \boldsymbol{e}_{\gamma} \in \mathcal{E} \right\}
		\end{equation*}
		furnished with the norm}
		\begin{equation*}
			\norm{ \{ x(\gamma) \}_{\gamma \in \Gamma} }
				\coloneqq \norm{\sum_{\gamma \in \Gamma} \norm{x(\gamma)}_{\gamma} \boldsymbol{e}_{\gamma}}_{\mathcal{E}}.
		\end{equation*}
	\end{definition}

	It is straightforward to see that when equipped with the coordinate-wise defined addition and scalar multiplication
	$( \bigoplus_{\gamma \in \Gamma} X_{\gamma} )_{\mathcal{E}}$ become a Banach space itself.
	
	\begin{remark}
		It is clear that when the family $\mathcal{X}$ from Definition~\ref{DEFINITION: direct sum} is \enquote{constant}, that is,
		$X_{\gamma} = X$ for all $\gamma \in \Gamma$ and some Banach space $X$, the direct sum construction $( \bigoplus_{\gamma \in \Gamma} X )_{\mathcal{E}}$
		degenerates to the well-known {\it K{\" o}the--Bochner sequence space} $\mathcal{E}(X)$.
		A lot is known about the structure of these spaces (see, for example, the monographs by Cembranos and Mendoza \cite{CM97} and Lin \cite{Lin04};
		see also \cite{CP96}, \cite{CHM96}, \cite{Gre69}, \cite{Kol03}, \cite{KP97}, \cite{Leo76}, \cite{LL85}, \cite{Now07} and \cite{ST80}).
		Of course, in a very particular situation when $X$ is just $\mathbb{R}$, the space $\mathcal{E}(X)$ is isometrically isomorphic to $\mathcal{E}$.
		\demo
	\end{remark}

	\begin{remark}
		Note that Definition~\ref{DEFINITION: direct sum} differs slightly from that used, for example, in \cite[Definition~1.5]{Lau01}
		({\it cf}. \cite[p.~14]{DK16}).
		Indeed, on the one hand, in any Banach sequence space $X$ the sequence $\{ \boldsymbol{e}_{\gamma} \}_{\gamma \in \Gamma}$
		is an unconditional basis of its separable part $X_o$.
		However, Definition~\ref{DEFINITION: direct sum} includes also non-separable spaces like, to name just one, $\ell_{\infty}$.
		On the other, any Banach space $X$ with a unconditional basis, say $\{ \boldsymbol{x}_{\gamma} \}_{\gamma \in \Gamma}$,
		can be seen as a Banach sequence space itself.
		To see this, let us consider a vector space $\text{BSS}[X]$ of all sequences $\{a(\gamma)\}_{\gamma \in \Gamma}$ of scalars with
		$\sum_{\gamma \in \Gamma} a(\gamma) \boldsymbol{x}_{\gamma} \in X$.
		If we endow $\text{BSS}[X]$ with the norm
		$\norm{\{a(\gamma)\}_{\gamma \in \Gamma}}_{\text{BSS}[X]} \coloneqq \sup \{ \Vert \sum_{\gamma \in \Gamma} b(\gamma) \boldsymbol{x}_{\gamma} \Vert_X
		\colon \abs{b(\gamma)} \leqslant \abs{a(\gamma)} \}$,
		then it is routine to verify that $\text{BSS}[X]$ is a Banach sequence space isomorphic to $X$.
		\demo
	\end{remark}
	
	Before moving any further, let us establish a few things.
	First, some examples\footnote{We urge the Readers not interested in concrete examples to skip this part.
	Otherwise, a conglomerate in the form of Table~\ref{TABLE : H} along with Theorem~\ref{THM: main theorem}
	may produce a multitude of examples of direct sums $( \bigoplus_{\gamma \in \Gamma} X_{\gamma} )_{\mathcal{E}}$
	which posses (or not) some Kadets--Klee properties. Since our goal is to build a general framework rather than to
	implement specific situations, we will not return to this topic later.}.
	
	\begin{example}[Orlicz spaces]
		Let $\ell_F$ be an {\it Orlicz sequence space} (see \cite{Ch96}, \cite[Chapter~4]{LT77} and \cite{Mal89} for details).
		Then, by the {\bf $\ell_F$-direct sum} of the family $\{X_n\}_{n=1}^{\infty}$ of Banach spaces we understand
             \begin{equation} \label{l_F - direct sum}
			( \bigoplus_{n=1}^{\infty} X_{n} )_{\ell_F}
			\coloneqq \left\{ \{ x(n) \}_{n=1}^{\infty} \in \prod_{n=1}^{\infty} X_{n} \colon \sum_{n=1}^{\infty} F( \lambda \norm{x(n)}_{n}) < \infty \text{ for some $\lambda > 0$} \right\}
		\end{equation}		
		with the {\bf Luxemburg--Nakano norm}
		$$\norm{ \{ x(n) \}_{n=1}^{\infty} } \coloneqq \inf \left\{ \lambda > 0 \colon \sum_{n=1}^{\infty} F(\norm{x(n)}_{n} / \lambda) \leqslant 1 \right\}.$$
		(One can also consider the space $\ell_F$ equipped with an equivalent, but in general not equal, norm called the {\it Orlicz norm};
		see \cite[Theorem~8.6, p.~55]{Mal89}.)
		Below we admit the degerated Orlicz functions, which may vanish outside zero and may jump to infinity, whence we need the notations
        \begin{equation}\label{a_phi}
            a_F \coloneqq \sup \{t \geqslant 0 \colon F(t) = 0 \}
            \quad \text{ and } \quad
            b_F \coloneqq \sup\{t \geqslant 0 \colon F(t) < \infty \}.
        \end{equation}
    	Following \cite[Section~12]{Mal89}, we will call them the {\bf Young functions}.
		In particular, by taking as a Young function $F$ a power function, the above construction degenerates to the well-known object.
		Indeed, according to \eqref{l_F - direct sum}, when $F(t) = t^{p}$ for $1 \leqslant p < \infty$, then
		\begin{equation*}
			( \bigoplus_{n=1}^{\infty} X_{n} )_{\ell_{p}}
				= \left\{ \{ x(n) \}_{n=1}^{\infty} \in \prod_{n=1}^{\infty} X_{n} \colon \sum_{n=1}^{\infty} \norm{x(n)}_{n}^{p} < \infty \right\}
		\end{equation*}
		with the norm $\norm{ \{ x(n) \}_{n=1}^{\infty} } = ( \sum_{n=1}^{\infty} \norm{x(n)}_{n}^{p} )^{1/p}$.
		On the other hand, if $F(t) = 0$ for $0 \leqslant t \leqslant 1$ and $F(t) = \infty$ for $t > 1$ (which corresponds to $p = \infty$),
		the formula \eqref{l_F - direct sum} gives
		\begin{equation*}
			( \bigoplus_{n=1}^{\infty} X_{n} )_{\ell_{\infty}}
				= \left\{ \{ x(n) \}_{n=1}^{\infty} \in \prod_{n=1}^{\infty} X_{n} \colon \sup_{n} \norm{x(n)}_{n} < \infty \right\}
		\end{equation*}
		with the norm $\norm{ \{ x(n) \}_{n=1}^{\infty} } = \sup_{n \in \mathbb{N}} \norm{x(n)}_{n}$.

        Recall also that the Orlicz function $\varphi$ satisfies the so-called {\bf $\delta_2$-condition} for small arguments provided
        $\limsup_{t \rightarrow 0} F(2t) / F(t) < \infty$.
        In particular, this condition implies that $a_F = 0$.
        It is part of a common knowledge that an Orlicz space is separable if, and only if, the Young function $F$ satisfies
        the $\delta_2$-condition (see \cite[Proposition~4.a.4, p.~138]{LT77} and \cite[p.~22]{Mal89}).
		\demo
	\end{example}

	\begin{example}[Lorentz spaces]
		Let $d(w,p)$ be the {\it Lorentz sequence space} (see \cite[pp.~175--179]{LT77}).
		Here, $1 \leqslant p < \infty$ and $\{ w_n \}_{n=1}^{\infty}$ is a decreasing sequence of non-negative
		real numbers such that $w_1 = 1$ and $\lim_{n \rightarrow \infty} w_n = 0$.
		The {\bf $d(w,p)$-direct sum} of the family $\{X_n\}_{n=1}^{\infty}$ of Banach spaces we understood as
		\begin{equation*}
			( \bigoplus_{n=1}^{\infty} X_{n} )_{d(w,p)}
			\coloneqq \left\{ \{ x(n) \}_{n=1}^{\infty} \in \prod_{n=1}^{\infty} X_{n} \colon \sum_{n=1}^{\infty} \left( \norm{x(n)}_{n}^{\star} \right)^p w_n < \infty \right\}
		\end{equation*}
		together with the norm $\norm{ \{ x(n) \}_{n=1}^{\infty} } \coloneqq ( \sum_{n=1}^{\infty} \left( \norm{x(n)}_{n}^{\star} \right)^p w_n )^{1/p}$,
		where $\{\norm{x(n)}_{n}^{\star}\}_{n=1}^{\infty}$ is the non-increasing rearrangement of the sequence $\{\norm{x(n)}_{n}\}_{n=1}^{\infty}$
		(see \cite{BS88}).
		\demo
	\end{example}

	\begin{example}[Nakano spaces]
		Let $\{ p_n \}_{n=1}^{\infty}$ be a sequence of positive integers such that $1 \leqslant p \leqslant \infty$.
		Let $\ell_{\{ p_n \}}$ be the {\it Nakano sequence space} (alias {\it variable exponent Lebesgue space}).
		The {\bf $\ell_{\{ p_n \}}$-direct sum} of the family $\{X_n\}_{n=1}^{\infty}$ of Banach spaces is defined as
		\begin{equation*}
			( \bigoplus_{n=1}^{\infty} X_{n} )_{\ell_{\{ p_n \}}} = \left\{ \{ x(n) \}_{n=1}^{\infty} \in \prod_{n=1}^{\infty} X_{n}
				\colon \sum_{n=1}^{\infty} \left( \lambda\norm{x(n)}_{n} \right)^{p_n} < \infty  \text{ for some } \lambda>0 \right\}
		\end{equation*}
		with the Luxemburg--Nakano norm
		$\norm{ \{ x(n) \}_{n=1}^{\infty} } \coloneqq \inf\{ \lambda > 0 \colon \sum_{n=1}^{\infty} \left( \norm{x(n)}_{n} / \lambda \right)^{p_n} \leqslant 1 \}$. 
		Note that if $\limsup_{n \rightarrow \infty} p_n = \infty$, the space $\ell_{\{ p_n \}}$ is not separable.
		\demo
	\end{example}

	\begin{example}[Ces{\' a}ro spaces]
		Let $ces_p$ with $1 \leq p < \infty$ denotes the {\it Ces{\' a}ro sequence space} (see, for example, \cite{FHS10}, \cite{KT17} and \cite{KT24} and their references).
		The {\bf $ces_p$-direct sum} of the family $\{X_n\}_{n=1}^{\infty}$ of Banach spaces is defined as
		\begin{equation*}
			( \bigoplus_{n=1}^{\infty} X_{n} )_{ces_p} \coloneqq \left\{ \{ x(n) \}_{n=1}^{\infty} \in \prod_{n=1}^{\infty} X_{n}
			\colon \sum_{n=1}^{\infty} \left( \frac{1}{n} \sum_{k=1}^n \norm{x(k)}_{k} \right) \boldsymbol{e}_n \in \ell_{p} \right\}
		\end{equation*}
		with the norm $\norm{ \{ x(n) \}_{n=1}^{\infty} } \coloneqq (\sum_{n=1}^{\infty} (\frac{1}{n} \sum_{k=1}^n \norm{x(k)}_{k} )^{p} )^{1/p}$.
		The end-point space $ces_{\infty}$ (note that $ces_1$ is trivial) is not separable, so the corresponding
		direct sum should be understand as
		\begin{equation*}
			( \bigoplus_{n=1}^{\infty} X_{n} )_{ces_{\infty}}
				\coloneqq \left\{ \{ x(n) \}_{n=1}^{\infty} \in \prod_{n=1}^{\infty} X_{n}
					\colon \sup_{n} \frac{1}{n} \sum_{k=1}^n \norm{x(k)}_{k} < \infty \right\}
		\end{equation*}
		with the norm $\norm{ \{ x(n) \}_{n=1}^{\infty} } \coloneqq \sup_{n \in \mathbb{N}} \frac{1}{n} \sum_{k=1}^n \norm{x(k)}_{k}$.
		\demo
	\end{example}

	\begin{notation}[Some useful notation regarding direct sums] \label{NOTATION : direct sums}
		As a general rule, we will also treat the elements $x$ living inside $( \bigoplus_{\gamma \in \Gamma} X_{\gamma} )_{\mathcal{E}}$
		as functions defined on $\Gamma$ with values in $X_{\gamma}$'s.
		Let us denote by 
   		$\lfloor \bullet \rceil \colon ( \bigoplus_{\gamma \in \Gamma} X_{\gamma} )_{\mathcal{E}} \rightarrow \mathcal{E}$
		the mapping
		\begin{equation}
			\lfloor \bullet \rceil \colon x \mapsto \left[ \gamma \mapsto \lfloor x \rceil(\gamma) \coloneqq \norm{x(\gamma)}_{\gamma} \boldsymbol{e}_{\gamma} \right].
		\end{equation}
		In other words,
        \begin{equation*}
                \lfloor x \rceil = \sum_{\gamma \in \Gamma} \norm{x(\gamma)}_{\gamma} \boldsymbol{e}_{\gamma}.
        \end{equation*}
		Clearly, the mapping $\lfloor \bullet \rceil$ is not(!) linear, but it is sublinear and
		\begin{equation*}
			\norm{\lfloor \bullet \rceil \colon ( \bigoplus_{\gamma \in \Gamma} X_{\gamma} )_{\mathcal{E}} \rightarrow \mathcal{E}}
			 	= \sup \left\{ \norm{\lfloor x \rceil}_{\mathcal{E}} \colon x \in \text{Ball} \, \left[ ( \bigoplus_{\gamma \in \Gamma} X_{\gamma} )_{\mathcal{E}} \right] \right\}
				= 1.
		\end{equation*}
		Further, let 
         \begin{equation*}
             \pi_{\gamma} \colon ( \bigoplus_{\gamma \in \Gamma} X_{\gamma} )_{\mathcal{E}} \rightarrow X_{\gamma}
         \end{equation*}        
		be the {\bf projection onto $\gamma^{\text{th}}$-coordinate}, that is, 
        $$\pi_{\gamma}\left( x \right) \coloneqq x(\gamma)$$
		for $\gamma \in \Gamma$ and $x \in ( \bigoplus_{\gamma \in \Gamma} X_{\gamma} )_{\mathcal{E}}$,
		while 
        \begin{equation*}
            j_{\gamma} \colon X_{\gamma} \rightarrow ( \bigoplus_{\gamma \in \Gamma} X_{\gamma} )_{\mathcal{E}} 
        \end{equation*}
        be the
		{\bf $\gamma^{\text{th}}$-coordinate embedding}, that is, $$j_{\gamma}(x) \coloneqq x \boldsymbol{e}_{\gamma}$$
		for $\gamma \in \Gamma$ and $x \in X_{\gamma}$.
		Now, it is straightforward to see that $X_{\gamma}$ is isometrically isomorphic to
		$j_{\gamma}X_{\gamma} \subset ( \bigoplus_{\gamma \in \Gamma} X_{\gamma} )_{\mathcal{E}}$
		and this subspace is complemented via
       \begin{equation*}
           	j_{\gamma} \circ \pi_{\gamma} \colon ( \bigoplus_{\gamma \in \Gamma} X_{\gamma} )_{\mathcal{E}} \rightarrow ( \bigoplus_{\gamma \in \Gamma} X_{\gamma} )_{\mathcal{E}}.
       \end{equation*}	
		Similarly, for a given sequence $\{ {x}_{\gamma} \}_{\gamma \in \Gamma}$
		of norm one vectors with ${x}_{\gamma} \in X_{\gamma}$ for $\gamma \in \Gamma$,
		one can define the mapping 
       \begin{equation*}
            j_{\mathcal{E}} \colon \mathcal{E} \rightarrow ( \bigoplus_{\gamma \in \Gamma} X_{\gamma} )_{\mathcal{E}}
       \end{equation*}     
       in the following way 
       $$j_{\mathcal{E}} ( \sum_{\gamma \in \Gamma} a(\gamma) \boldsymbol{e}_{\gamma} ) \coloneqq \{ a(\gamma) {x}_{\gamma} \}_{\gamma \in \Gamma}.$$
		Hereby, $\sum_{\gamma \in \Gamma} a(\gamma) \boldsymbol{e}_{\gamma} \in \mathcal{E}$.
		For this reason, the space $\mathcal{E}$ is isometrically isomorphic to a complemented subspace of
		$( \bigoplus_{\gamma \in \Gamma} X_{\gamma} )_{\mathcal{E}}$.
		Note that formally the mapping $j_{\mathcal{E}}$ depends upon the sequence $\{ {x}_{\gamma} \}_{\gamma \in \Gamma}$.
		However, this is basically irrelevant, because we can select one such a sequence once and for all.
		\demo
	\end{notation}

	Let us conclude this section by saying a few words about the duality of direct sums.
	Following Lausten \cite{Lau01}, take $x = \{ x_{\gamma} \}_{\gamma \in \Gamma}$ from $( \bigoplus_{\gamma \in \Gamma} X_{\gamma} )_{\mathcal{E}}$
	and $\varphi = \{ \varphi_{\gamma} \}_{\gamma \in \Gamma}$ from $( \bigoplus_{\gamma \in \Gamma} X_{\gamma}^{*} )_{\mathcal{E}^{\times}}$.
	Then, using H{\" o}lder--Rogers's inequality, we have
	\begin{align*}
		\sum_{\gamma \in \Gamma} \abs{ \langle \varphi_{\gamma}, x_{\gamma} \rangle }
			& \leqslant \sum_{\gamma \in \Gamma} \norm{x_{\gamma}}_{X_{\gamma}} \norm{\varphi_{\gamma}}_{X_{\gamma}^{*}} \\
			& \leqslant \norm{ \sum_{\gamma \in \Gamma} \norm{x_{\gamma}}_{X_{\gamma}} \boldsymbol{e}_{\gamma} }_{\mathcal{E}}
				\norm{ \sum_{\gamma \in \Gamma} \norm{\varphi_{\gamma}}_{X_{\gamma}^{*}} \boldsymbol{e}_{\gamma} }_{\mathcal{E}^{\times}} \\
			& = \norm{x} \norm{\varphi}.
	\end{align*}
	This means that if we define a linear form $\Upsilon(\varphi)$ on $( \bigoplus_{\gamma \in \Gamma} X_{\gamma} )_{\mathcal{E}}$
	in the following way $\langle \Upsilon(\varphi), x \rangle \coloneqq \sum_{\gamma \in \Gamma} \langle \varphi_{\gamma}, x_{\gamma} \rangle$,
	then $\norm{\Upsilon(\varphi)} \leqslant \norm{\varphi}$.
	Thus, the mapping $\Upsilon \colon \varphi \mapsto \Upsilon(\varphi)$ is a norm one operator from
	$( \bigoplus_{\gamma \in \Gamma} X_{\gamma}^{*} )_{\mathcal{E}^{\times}}$ into(!) $\left[ ( \bigoplus_{\gamma \in \Gamma} X_{\gamma} )_{\mathcal{E}} \right]^{*}$.
	The question of when the mapping $\Upsilon$ is surjective is resolved by the following
	
	\begin{proposition}[Duality of direct sums] \label{PROPOSITION : duality of direct sums}
		{\it Let $\{ X_{\gamma} \}_{\gamma \in \Gamma}$ be a family of Banach spaces.
		Further, let $\mathcal{E}$ be a separable Banach sequence space defined on $\Gamma$.
		Then the topological dual of $( \bigoplus_{\gamma \in \Gamma} X_{\gamma} )_{\mathcal{E}}$ is
		naturally isometrically isomorphic to $( \bigoplus_{\gamma \in \Gamma} X_{\gamma}^{*} )_{\mathcal{E}^{\times}}$.}
	\end{proposition}
	\begin{proof}
		Take $x^{*}$ from $\left[ ( \bigoplus_{\gamma \in \Gamma} X_{\gamma} )_{\mathcal{E}} \right]^{*}$.
		We claim that $x^{*}$ can be uniquely represented as $x^{*} = \{ x_{\gamma}^{*} \}_{\gamma \in \Gamma}$ for some $x_{\gamma}^{*} \in X_{\gamma}^{*}$.
		Too see this, note that the set of all finite linear combinations of vectors from
		$\{ x_{\gamma} \otimes \boldsymbol{e}_{\gamma} \colon x_{\gamma} \in X_{\gamma} \text{ and } \gamma \in \Gamma \}$
		is dense in $( \bigoplus_{\gamma \in \Gamma} X_{\gamma} )_{\mathcal{E}}$.
		In fact, take $x = \{ x_{\gamma} \}_{\gamma \in \Gamma}$ from $( \bigoplus_{\gamma \in \Gamma} X_{\gamma} )_{\mathcal{E}}$
		and let $\{ \Gamma_n \}_{n=1}^{\infty}$ be a family of finite subsets of $\Gamma$ such that $\Gamma_1 \subset \Gamma_2 \subset ...$
		and $\bigcup_{n=1}^{\infty} \Gamma_n = \Gamma$.
		Note that since the space $\mathcal{E}$ is order continuous, so
		\begin{equation*}
			\norm{x - \sum_{\gamma \in \Gamma_n} x_{\gamma} \otimes \boldsymbol{e}_{\gamma}}
				= \norm{\sum_{\gamma \in \Gamma \setminus \Gamma_n} x_{\gamma} \otimes \boldsymbol{e}_{\gamma}}
				= \norm{\sum_{\gamma \in \Gamma \setminus \Gamma_n} \norm{x_{\gamma}}_{X_{\gamma}} \boldsymbol{e}_{\gamma}}_{\mathcal{E}} \rightarrow 0.
		\end{equation*}
		Denoting $x_{\gamma}^{*} \coloneqq \left. x^{*} \right|_{X_\gamma}$ we get 
		\begin{align*}
			x^{*}(x)
				& = x^{*} ( \sum_{\gamma \in \Gamma} x_{\gamma} \otimes \boldsymbol{e}_{\gamma} ) \\
				& = \sum_{\gamma \in \Gamma} x^{*}(x_{\gamma}) \boldsymbol{e}_{\gamma} \\
				& = \sum_{\gamma \in \Gamma} x_{\gamma}^{*}(x) \boldsymbol{e}_{\gamma} \\
				& = ( \sum_{\gamma \in \Gamma} x_{\gamma}^{*} \otimes \boldsymbol{e}_{\gamma} )(x).
		\end{align*}
		Our claim follows.
		It remains to show that $\sum_{\gamma \in \Gamma} \Vert x_{\gamma}^{*} \Vert_{X_{\gamma}^{*}} \boldsymbol{e}_{\gamma} \in \mathcal{E}^{\times}$.
		Suppose this is not the case.
		Then there exists a sequence $\{ y_n \}_{n=1}^{\infty}$ of functions from $\text{Ball}(\mathcal{E})$ such that
		\begin{equation} \label{EQ : nie w Kothe dualu}
			\sum_{\gamma \in \Gamma} \Vert x_{\gamma}^{*} \Vert_{X_{\gamma}^{*}} \abs{y_n(\gamma)} > n.
		\end{equation}
		Moreover, for every $\gamma \in \Gamma$, there is $f^{(\gamma)} \in \text{Ball}(X_{\gamma})$ with
		$\abs{\langle f^{(\gamma)}, x^{*}_{\gamma} \rangle} = \Vert x_{\gamma}^{*} \Vert_{X_{\gamma}^{*}}$.
		Of course, $\sum_{\gamma \in \Gamma} f^{(\gamma)} y_n(\gamma) \boldsymbol{e}_{\gamma}$ belongs to 
		$\text{Ball} \left[ ( \bigoplus_{\gamma \in \Gamma} X_{\gamma} )_{\mathcal{E}} \right]$ for each $n \in \mathbb{N}$.
		Remembering about \eqref{EQ : nie w Kothe dualu}, we have
		\begin{align*}
			\norm{x^{*}}
				& = \sup \left\{ \abs{ \langle x, x^{*} \rangle} \colon x \in \text{Ball} \left[ ( \bigoplus_{\gamma \in \Gamma} X_{\gamma} )_{\mathcal{E}} \right] \right\} \\
				& \geqslant \sum_{\gamma \in \Gamma} \abs{ \langle f^{(\gamma)} y_n(\gamma), x_{\gamma}^{*} \rangle} \\
				& = \sum_{\gamma \in \Gamma} \abs{y_n(\gamma)} \abs{\langle f^{(\gamma)}, x_{\gamma}^{*} \rangle} \\
				& = \sum_{\gamma \in \Gamma} \abs{y_n(\gamma)} \Vert x_{\gamma}^{*} \Vert_{X_{\gamma}^{*}} > n.
		\end{align*}
		But this is nonsense.
		The proof has been completed.
	\end{proof}

	The structure of direct sums is non-trivial, interesting and has been studied from various perspectives by many authors
	(see, for example, \cite{AA22}, \cite{Day41}, \cite{DD67}, \cite{DPS07}, \cite{DK16}, \cite{DV86}, \cite{HLR91}, \cite{KL92},
	\cite{KT24}, \cite{KL92}, \cite{Lau01}, \cite{MP22} and \cite{San23}).

	\section{{\bf Main results}} \label{SECTION : Main results}
	
	In this section we will first present some general facts about the abstract Kadets--Klee property ${\bf H}(\mathfrak{T})$
	in Banach spaces that we will use later (Section~\ref{SUBSECTION : recollections KKP}), and then we will present results
	about the property ${\bf H}(\mathfrak{T})$ in direct sums (Section~\ref{SUBSECTION : abstract framework}).
	This last section forms the main body of our work.
	
	\subsection{Recollections on Kadets--Klee properties} \label{SUBSECTION : recollections KKP}
	
	Recall that a Banach space $X$ is said to have the {\bf Kadets--Klee property} (briefly, {\bf the property}\footnote{To avoid confusion,
	note that the letter \enquote{H} means essentially nothing and we only use it for historical reasons
	(see, for example, \cite[pp.~220--221]{Meg98} for a more detailed discussion).}
	$\mathbf{H}$) provided for sequences on the unit sphere of $X$ the weak topology and the norm topology agree.
	
	\begin{remark}
		Let $X$ be a separable Banach space.
		It can be deduced from Rosenthal's work \cite{Ros78} that the space $X$ contains no subspace isomorphic to $\ell_1$ if,
		and only if, every bounded subset of $X$ is weakly sequentially dense in its weak closure (precisely, see \cite[Theorem~3]{Ros78}).
		Thus, even though the weak topology is not sequential\footnote{See \cite[Theorem~1.5]{GKP16}; cf. \cite[Proposition~2.5.15, p.~215]{Meg98}.
		Roughly speaking, sequential spaces are those topological spaces whose topology can be completely described by in terms of convergent sequences.
		For example, Fr{\' e}chet--Urysohn spaces and first-countable spaces (in particular, metric spaces) are sequential spaces.},
		in the class of separable Banach spaces without isomorphic copies of $\ell_1$ the property ${\bf H}$ can be formulated in a seemingly
		stronger way, namely, that the weak and the norm topology coincide on the unit sphere in $X$.
		\demo
	\end{remark}

	The abstract variant of the classical Kadets--Klee property ${\bf H}$ mentioned above is as follows

	\begin{definition}[Abstract Kadets--Klee property] \label{DEF: Kadets--Klee property}
		{\it Let $X$ be a Banach space and let $\mathfrak{T}$ be a linear Hausdorff topology on $X$ coarser than the norm topology on $X$.
		We will say that the space $X$ has the {\bf Kadets--Klee property} with respect to $\mathfrak{T}$
		(abbreviated to the {\bf property $\mathbf{H}(\mathfrak{T})$})
		if for sequences on the unit sphere of $X$ the norm topology agree with $\mathfrak{T}$.}
	\end{definition}
	
	Evidently, the Kadets--Klee property ${\bf H}(\textit{weak})$ coincide with the usual Kadets--Klee property ${\bf H}$.
	Moreover, since the topology $\mathfrak{T}$ is assumed to be linear\footnote{If we do not do this, we will end up with two in general non-equivalent
	definitions of the Kadets--Klee property.}, we can reformulate the above definition in the following equivalent way
	
	\begin{remark}[Equivalent definition of $\mathbf{H}(\mathfrak{T})$]
		Let $X$ be a Banach space and let $\mathfrak{T}$ be a linear Hausdorff topology on $X$ coarser than the norm topology on $X$.
		The space $X$ has the Kadets--Klee property with respect to $\mathfrak{T}$ if, and only if, for any sequence $\{x_n\}_{n=1}^{\infty}$
		from $X$ converging with respect to $\mathfrak{T}$ to $x \in X$ such that $\norm{x_{n}}$ converges to $\norm{x}$,
		one has that $x_{n}$ converges to $x$ in the norm topology on $X$.
		\demo
	\end{remark}

	\begin{remark}[On Hausdorffness assumption in Definition~\ref{DEF: Kadets--Klee property}] \label{REMARK : minimal topologies}
		The assumption that the topology $\mathfrak{T}$ is Hausdorff is much more useful than is might seem.
		Due to this, there is a minimal\footnote{In the sense that having the property ${\bf H}(\textit{point-wise})$ implies the property ${\bf H}(\mathfrak{T})$
		for any linear Hausdorff topology $\mathfrak{T}$.}
		Kadets--Klee property in the class of Banach sequence spaces, namely, the property ${\bf H}(\textit{point-wise})$.
		Indeed, this is due to the following observation:
        \vspace{2pt}
		\begin{quote}
		    {\it The minimal Hausdorff topology for the class of Banach sequence spaces coincide - when restricted to the unit ball - with the topology of point-wise convergence}.
		\end{quote}
        \vspace{2pt}
		The proof goes like this.
		Let us equip $X$ with the topology of point-wise convergence.
		Then, as follows from the definition of the product topology, the unit ball $\text{Ball}(X)$ can be seen as a closed subspace of the product $\prod_{n=1}^{\infty} [-1,1]$.
		Since, due to Tychonoff's theorem, the product space $\prod_{n=1}^{\infty} [-1,1]$ is compact, so $\text{Ball}(X)$ is compact as well.
		However, it is well-known that compact Hausdorff topologies are minimal, that is, if $\widetilde{\mathfrak{T}}$ is any coarser Hausdorff
		topology on $\text{Ball}(X)$, then $\widetilde{\mathfrak{T}}$ must actually coincide with the topology of point-wise convergence.
		This completes the proof of the above observation.
		\demo
	\end{remark}
	
	\begin{digression}
		For obvious reasons there is no point in considering the Kadets--Klee properties for topologies finer than the given norm topology.
		Formally, however, nothing prevents us from examining topologies that are incomparable with the given norm topology
		(in fact, the topology of convergence in measure is sometimes incomparable with the norm topology).
		We decided not to do this here.
		Moreover, determining whether a certain type of convergence is topological or not can be sometimes confusing.
		For example, it is a common knowledge that one of the roots for measure theory, that is, the convergence almost everywhere,
		it not(!) topological (see, for example, \cite{Ord66}).
		Because of this, one may wonder about considering the Kadets--Klee type properties with respect to some convergence structures on $X$
		instead of topologies.
		\demo
	\end{digression}

	The lower semi-continuity of the norm will often play an important role.
	Therefore, let us recall the general definition.

	\begin{definition}
	{\it Let $(X, \mathfrak{T})$ be a topological space.
	We say that the real-valued function $f \colon X \rightarrow \mathbb{R}$ is {\bf sequentially lower semi-continuous} with
	respect to $\mathfrak{T}$ if whenever $\{x_n\}_{n=1}^{\infty}$ is a sequence from $X$ converging to $x \in X$ with
	respect to $\mathfrak{T}$, then
     $$f(x) \leqslant \liminf_{n \rightarrow \infty} f(x_n).$$}
	\end{definition}
	
	The following lemma will prove useful many times
	(this is undoubtedly folklore, but based on our knowledge, the only readily available proof can be found in \cite[Proposition~2.2]{DDSS04};
	{\it cf}. \cite{BDDL94} and \cite{Len88}).
	
	\begin{lemma} \label{LEMMA: H tau => lower semicontinuous}
		{\it Let $X$ be a Banach space equipped with a linear Hausdorff topology $\mathfrak{T}$ coarser than the norm topology.
		Suppose that the space $X$ has the property $\mathbf{H}(\mathfrak{T})$.
		Then the norm function $x \mapsto \norm{x}$ is sequentially lower semi-continuous with
		respect to $\mathfrak{T}$.}
	\end{lemma}
	
	The following definition generalized the classical notion of order continuity.

	\begin{definition}[Order continuity] \label{DEF : OC Tau}
		{\it Let $X$ be a Banach sequence space.
		Further, let $\mathfrak{T}$ be a linear Hausdorff topology on $X$ coarser than the norm topology on $X$.
		We will say that the space $X$ is {\bf order continuous}\footnote{Pedantically speaking, $\sigma$-order continuous with respect to $\mathfrak{T}$.}
		with respect to $\mathfrak{T}$ (briefly, the space $X$ is ${\bf OC}(\mathfrak{T})$) if whenever $\{x_n\}_{n=1}^{\infty}$
		is a sequence of positive functions from $X$ that is order bounded by $x \in X_{+}$
		and converges to zero with respect to $\mathfrak{T}$, it follows that $\{x_n\}_{n=1}^{\infty}$ is a null sequence with respect to the norm topology on $X$.}
	\end{definition}
	
	\begin{remark}[About Definition~\ref{DEF : OC Tau}]
		Let $X$ be a Banach sequence space.
		Obviously, if $\mathfrak{T}$ is the topology of point-wise convergence, then the property ${\bf OC}(\mathfrak{T})$ for $X$
		is nothing else but the well-known order continuity property of $X$.
		This property is usually abbreviated as ${\bf OC}$, hence our notation.
		Moreover, due to Remark~\ref{REMARK : minimal topologies}, if $X$ has the property ${\bf OC}$ then it also has ${\bf OC}(\mathfrak{T})$.
		The reverse implication holds for certain topologies (like, for example, the weak topology and pre-Lebesgue solid topologies;
		see Example~\ref{EXAMPLE : OC(tau) => OC} and Lemma~\ref{LEMMA : H(tau) => OC}, respectively), but in general might fail badly.
		For instance, fix $1 \leqslant p,q < \infty$, and consider a Banach sequence space $\mathcal{Z}(p,q)$ defined via the norm
		\begin{equation*}
			\norm{ \sum_{n=1}^{\infty} x(n) \boldsymbol{e}_n }_{\mathcal{Z}(p,q)}
				\coloneqq \left[ \left( \sum_{n=1}^{\infty} \abs{x(2n)}^p \right)^{q/p} + \sup_{n \in \mathbb{N}} \abs{x(2n-1)}^q \right]^{1/q}.
		\end{equation*}
		Plainly, $\mathcal{Z}(p,q)$ is not ${\bf OC}$.
		However, the space $\mathcal{Z}(p,q)$ has the property ${\bf H}(\textit{uniform})$ (note that the uniform topology is a linear Hausdorff topology on $\mathcal{Z}(p,q)$
		that is coarser than the norm topology).
		To see this, just note that $\ell_p$ has the property  ${\bf H}(\textit{point-wise})$ (so, also ${\bf H}(\textit{uniform})$; see Table~\ref{TABLE : H})
		and $\ell_{\infty}$ has the property ${\bf H}(\textit{uniform})$ (which is a trivial observation).
		Thus, since $\mathcal{Z}(p,q)$ is isometrically isomorphic to $\ell_p \oplus_{q} \ell_{\infty}$ and the property ${\bf H}(\textit{uniform})$
		can be lifted from $X$ to $X \oplus_{q} \ell_{\infty}$, so $\mathcal{Z}(p,q)$ has the property ${\bf H}(\textit{uniform})$ as well.
		\demo
	\end{remark}

     \begin{definition}[Pre-Lebesgue topology] \label{pre-Lebesgue}
     	{\it Let $X$ be a Banach sequence space equipped with a linear Hausdorff topology $\mathfrak{T}$ coarser than the norm topology on $X$.
     	We say that $\mathfrak{T}$ is the {\bf pre-Lebesgue topology}\footnote{Typically, this property applies to locally solid topologies
     	(see, for example, \cite[Definition~8.1, p.~52]{AB03} and \cite[Definition~2.1]{Con19}). However, since one can find examples of topologies which,
		even though are not locally solid, still behave like pre-Lebesgue topologies, we will stick to the same terminology
		(see Example~\ref{EXAMPLE : OC(tau) => OC}; {\it cf}. \cite[Proposition 2.6]{DDSS04} and \cite[Proposition 2.1]{DHLMS03}).}
		provided any disjoint, positive and order bounded sequence from $X$ converges to zero with respect to topology $\mathfrak{T}$.}
     \end{definition}

	In view of Remark~\ref{DEF : OC Tau} and with the aid of Definition~\ref{DEF : OC Tau} it is not difficult to show the following
	
	\begin{lemma} \label{LEMMA : H(tau) => OC}
		{\it Let $X$ be a Banach sequence space equipped with a linear Hausdorff topology $\mathfrak{T}$ coarser than the norm topology on $X$.
		Suppose that $X$ has the property ${\bf H}(\mathfrak{T})$.
		Then $X$ is ${\bf OC}(\mathfrak{T})$.
		Furthermore, if the topology $\mathfrak{T}$ is pre-Lebesgue, then $X$ is ${\bf OC}$.}
	\end{lemma}
	\begin{proof}
    	Let $\{x_n\}_{n=1}^{\infty}$ be a sequence of positive functions from $X$ that is order bounded by $x \in X_{+}$
    	and converges to zero with respect to $\mathfrak{T}$.
		Then $0 \leqslant x - x_n \leqslant x$ for $n \in \mathbb{N}$ and, since the topology $\mathfrak{T}$ is linear,
		the sequence $\{ x - x_n \}_{n=1}^{\infty}$ tends to $x$ with respect to $\mathfrak{T}$.
		Using Lemma~\ref{LEMMA: H tau => lower semicontinuous},
		\begin{equation*}
			\norm{x}
				\leqslant \liminf_{n \rightarrow \infty} \norm{x - x_n}
				\leqslant \limsup_{n \rightarrow \infty} \norm{x - x_n}
				\leqslant \norm{x}.
		\end{equation*}
		In consequence, $\norm{x - x_n} \rightarrow \norm{x}$.
		However, since the space $X$ has the property ${\bf H}(\mathfrak{T})$ and $\norm{x_n} = \norm{x - (x - x_n)}$, so $\norm{x_n} \rightarrow 0$.
		To sum up, all this means that the space $X$ is ${\bf OC}(\mathfrak{T})$.
		Plainly, if the topology $\mathfrak{T}$ is pre-Lebesgue, then, applying Remark \ref{REMARK : minimal topologies} and Theorem 2.1 in \cite{Kol18}, we conclude that the property ${\bf OC}(\mathfrak{T})$ coincide with the usual ${\bf OC}$.
	\end{proof}
   
    Let us now discuss briefly the relations between Kadets--Klee property and some geometric properties.
    Recall that every separable Banach space has an equivalent {\bf LUR}-norm (see \cite[p.~220]{Meg98}).
    Moreover, it is well-known that if a Banach space $X$ is ${\bf LUR}$, then $X$ has the property $\mathbf{H}(\textit{weak})$
    (see, for example, \cite[Proposition~5.3.7, p.~463]{Meg98}).
    However, much more is true.
 
	\begin{lemma}(See \cite[Proposition 2.3]{DDSS04})
		{\it Let $X$ be a Banach space equipped with a linear Hausdorff topology $\mathfrak{T}$ coarser than the norm topology.
		Suppose that the space $X$ is ${\bf LUR}$ and the norm function $x \mapsto \norm{x}$ is sequentially lower semi-continuous
		with respect to $\mathfrak{T}$.
		Then the space $X$ has the property $\mathbf{H}(\mathfrak{T})$.}
	\end{lemma}

	We will need one more definition.

    \begin{definition}[Abstract Schur property] \label{DEF : abstract Schur}
       {\it We will say that a Banach space $X$ has the {\bf Schur property} with respect to a linear Hausdorff topology $\mathfrak{T}$
       on $X$ coarser than the norm topology if whenever $\{x_n\}_{n=1}^{\infty}$ is a sequence from $X$ converging to $x \in X$
       with respect to $\mathfrak{T}$, then $x_{n}$ converges to $x$ in the norm topology on $X$.}
    \end{definition}
  
  	Of course, the Schur property with respect to the weak topology is what one usually has in mind when referring to the Schur property.
  	For other topologies this property rarely, if ever, holds (for example, it is easy to see that no Banach sequence space has the Schur
  	property with respect to the point-wise topology).
	
	The proof of the following result does not present any difficulty.
	
	\begin{lemma}
		{\it Let $X$ be a Banach space equipped with a linear Hausdorff topology $\mathfrak{T}$ coarser than the norm topology.
		Suppose that the space $X$ has the Schur property with respect to $\mathfrak{T}$.
		Then the space $X$ has the property $\mathbf{H}(\mathfrak{T})$.}
	\end{lemma}
	
	\subsection{Abstract framework} \label{SUBSECTION : abstract framework}
	
	It is a fairly standard observation that geometric properties are usually inherited by closed subspaces.
	In order to mimic this situation in our context, we need the following
	
	\begin{definition}[$\oplus$-compatible topology] \label{DEF: admissible topology}
		{\it Let $\{ X_{\gamma} \}_{\gamma \in \Gamma}$
		be a countably family of Banach spaces equipped with a linear Hausdorff topologies $\mathfrak{T}_{\gamma}$
		coarser than the corresponding norm topologies on $X_{\gamma}$'s.
		Moreover, let $\mathcal{E}$ be a Banach sequence space on $\Gamma$
        equipped with a linear Hausdorff topology
		$\mathfrak{T}_{\mathcal{E}}$ coarser than the norm topology on $\mathcal{E}$.
		We will say that the linear Hausdorff topology $\mathfrak{T}$ on $( \bigoplus_{\gamma \in \Gamma} X_{\gamma} )_{\mathcal{E}}$
		is {\bf $\oplus$-compatible} with the topologies of their components if}
		\begin{enumerate}
			\item[(C1)] {\it the topology $\mathfrak{T}$ is coarser than the norm topology on $( \bigoplus_{\gamma \in \Gamma} X_{\gamma} )_{\mathcal{E}}$;}
			\item[(C2)] {\it for each $\gamma \in \Gamma$ the projection $\pi_{\gamma} \colon ( \bigoplus_{\gamma \in \Gamma} X_{\gamma} )_{\mathcal{E}} \rightarrow X_{\gamma}$
				is $\mathfrak{T}$-to-$\mathfrak{T}_{\gamma}$ sequentially continuous;}
			\item[(C3)] {\it for each $\gamma \in \Gamma$ the embedding $j_{\gamma} \colon X_{\gamma} \rightarrow ( \bigoplus_{\gamma \in \Gamma} X_{\gamma} )_{\mathcal{E}}$
				is $\mathfrak{T}_{\gamma}$-to-$\mathfrak{T}$ sequentially continuous;}
			\item[(C4)] {\it the embedding $j_{\mathcal{E}} \colon \mathcal{E} \rightarrow ( \bigoplus_{\gamma \in \Gamma} X_{\gamma} )_{\mathcal{E}}$
			is $\mathfrak{T}_{\mathcal{E}}$-to-$\mathfrak{T}$ sequentially continuous
			(here, let us recall the discussion included in Notation~\ref{NOTATION : direct sums}).}
		\end{enumerate}
	\end{definition}

	There is a veritable zoo of examples of $\oplus$-compatible topologies, but we will reserve any constructions until Section~\ref{SECTION : Play Doh},
	where some specific topologies will receive due attention.
	Anticipating the facts a bit, the two most important representatives, that is, the weak topology and the topology of local convergence in measure,
	are indeed $\oplus$-compatible (see Theorem~\ref{THM : H weak} and Theorem~\ref{THEOREM : H(measure)}, respectively).
	
	Notwithstanding, to paraphrase A. W. Miller, with the above definition at hand, it is hard not to show the next result
	(or at least (1), because the presence of (2) will become clear a little later).
	
	\begin{theorem}[Inheritance result] \label{THM: H is hereditary}
		{\it Let $\{ X_{\gamma} \}_{\gamma \in \Gamma}$ be a family of Banach spaces equipped with a linear Hausdorff
		topologies $\mathfrak{T}_{\gamma}$ coarser than the corresponding norm topologies on $X_{\gamma}$'s.
		Further, let $ \mathcal{E} $ be a Banach sequence space on $\Gamma$ equipped with a linear Hausdorff topology
		$\mathfrak{T}_{\mathcal{E}}$ coarser than the norm topology on $\mathcal{E}$.
		Suppose that the space $ ( \bigoplus_{\gamma \in \Gamma} X_{\gamma} )_{\mathcal{E}} $
	    has the property ${\bf H(\mathfrak{T})}$, where $\mathfrak{T}$ is the $\oplus$-compatible topology.
		Then}
		\begin{itemize}
			\item[(1)] {\it all $X_{\gamma}$'s have the appropriate property ${\bf H(\mathfrak{T}_{\gamma})}$
			and the space $\mathcal{E}$ has the property ${\bf H(\mathfrak{T}_{\mathcal{E}})}$;}
			\item[(2)] {\it the set $\Gamma$ can be decomposed into two disjoint subsets, say $\Gamma_1$ and $\Gamma_2$,
			in such a way that all $X_{\gamma}$'s with $\gamma \in \Gamma_1$ have the Schur property with respect to $\mathfrak{T}_{\gamma}$
			and $\mathcal{E}$ is ${\bf SM(\gamma)}$ for $\gamma \in \Gamma_2$.}
		\end{itemize}
	\end{theorem}
	\begin{proof}
		For the sake of clarity, we will divide the argument into three parts.
		
		{\bf First part.} Suppose that $\mathcal{E}$ fails to have the property ${\bf H(\mathfrak{T}_{\mathcal{E}})}$.
		This means that we can find a sequence $\{x_n\}_{n=1}^{\infty}$ of elements from $\mathcal{E}$ and $x \in \mathcal{E}$ such that
		\begin{itemize}
			\item[$\bullet$] $x_{n}$ converges in $\mathfrak{T}_{\mathcal{E}}$ topology to $x$;
			\item[$\bullet$] $\norm{x_{n}}_{\mathcal{E}}$ converges to $\norm{x}_{\mathcal{E}}$;
			\item[$\bullet$] and $x_{n}$ does not converge to $x$ in the norm topology of $\mathcal{E}$.
		\end{itemize}
		Since the topology $\mathfrak{T}$ is assumed to be $\oplus$-compatible, so
		\begin{equation} \label{EQ: 1}
			\mathfrak{T}\text{-}\lim_{n \rightarrow \infty} j_{\mathcal{E}}(x_n) = j_{\mathcal{E}}(\mathfrak{T}_{\mathcal{E}}\text{-}\lim_{n \rightarrow \infty} x_n) = j_{\mathcal{E}}(x).
		\end{equation}
		Moreover,
		\begin{equation} \label{EQ: 2}
			\lim_{n \rightarrow \infty} \norm{j_{\mathcal{E}}(x_n)}
			= \lim_{n \rightarrow \infty} \norm{x_n}_{\mathcal{E}}
			= \lim_{n \rightarrow \infty} \norm{x}_{\mathcal{E}}
			= \lim_{n \rightarrow \infty} \norm{j_{\mathcal{E}}(x)}
		\end{equation}
		and
		\begin{equation} \label{EQ: 3}
			\lim_{n \rightarrow \infty} \norm{j_{\mathcal{E}}(x) - j_{\mathcal{E}}(x_n)}
				= \lim_{n \rightarrow \infty} \norm{x - x_n}_{\mathcal{E}} > 0.
		\end{equation}
		However, \eqref{EQ: 1} along with \eqref{EQ: 2} and \eqref{EQ: 3} immediately imply that
		$( \bigoplus_{\gamma \in \Gamma} X_{\gamma} )_{\mathcal{E}}$ does not have the property ${\bf H(\mathfrak{T})}$.
		
		{\bf Second part.} In a completely analogous way, it can be shown that if $X_{\gamma}$ does not have the property
		${\bf H(\mathfrak{T}_{\gamma})}$ for some $\gamma \in \Gamma$, then $( \bigoplus_{\gamma \in \Gamma} X_{\gamma} )_{\mathcal{E}}$
		does not have the property ${\bf H(\mathfrak{T})}$.
		
		{\bf Third part.}
		Suppose that there is $\gamma_0 \in \Gamma$ such that $X_{\gamma_0}$ fails to have the Schur property
		with respect to $\mathfrak{T}_{\gamma_0}$ topology and $\mathcal{E}$ is not ${\bf SM}(\gamma_0)$.
		This means that there is a sequence $\{ x_n^{(\gamma_0)} \}_{n=1}^{\infty}$ in $X_{\gamma_0}$
		such that $x_{n}^{(\gamma_0)}$ is a null sequence with respect to $\mathfrak{T}_{\gamma_0}$ topology,
		but $\vert \vert x_n(\gamma_0) \vert \vert_{\gamma_0} = 1$ for all $n \in \mathbb{N}$.
		Now, Lemma~\ref{LEMMMA : geometric lemma SM} guarantee the existence of a positive norm one sequence
		$\sum_{\gamma \in \Gamma} a(\gamma) \boldsymbol{e}_{\gamma} \in \mathcal{E}$ such that $a(\gamma_0) > 0$ and
		\begin{equation*}
			\norm{ \sum_{\gamma \in \Gamma} a(\gamma) \boldsymbol{e}_{\gamma} }_{\mathcal{E}}
				= \norm{ \sum_{\gamma \in \Gamma \setminus \{ \gamma_0 \}} a(\gamma) \boldsymbol{e}_{\gamma} }_{\mathcal{E}}.
		\end{equation*}
		Now, for each $\gamma \in \Gamma \setminus \{ \gamma_0 \}$, let us choose a norm one vector, say $x^{(\gamma)}$,
		in $X_{\gamma}$ and set
		\begin{equation*}
			y \coloneqq \sum_{\gamma \in \Gamma \setminus \{ \gamma_0 \}} a(\gamma) x^{(\gamma)} \otimes \boldsymbol{e}_{\gamma}
		\end{equation*}
		and
		\begin{equation*}
			\quad y_n \coloneqq y + a(\gamma_0) x^{(\gamma_0)}_n \otimes \boldsymbol{e}_{\gamma_0} \quad \text{ for } \quad n \in \mathbb{N}.
		\end{equation*}
		Plainly, $\norm{y} = 1$ and $\norm{y_n} = 1$ for all $n \in \mathbb{N}$.
		Moreover, since the topology $\mathfrak{T}$ is $\oplus$-compatible, so $y_n$ converges in $\mathfrak{T}$ topology to $y$.
		However,
		\begin{equation*}
			\norm{y - y_n} = a(\gamma_0) \norm{\boldsymbol{e}_{\gamma_0}}_{\mathcal{E}} > 0,
		\end{equation*}
		so $( \bigoplus_{\gamma \in \Gamma} X_{\gamma} )_{\mathcal{E}} $ fails to have the property ${\bf H}(\mathfrak{T})$.
	\end{proof}

	The next theorem is the main tool of our work.

	\begin{theorem}[Lifting result] \label{THM: main theorem}
		{\it Let $\{ X_{\gamma} \}_{\gamma \in \Gamma}$ be a countably family of Banach spaces equipped with a linear Hausdorff topologies $\mathfrak{T}_{\gamma}$
		coarser than the corresponding norm topologies on $X_{\gamma}$'s.
		Further, let $\mathcal{E}$ be a Banach sequence space on $\Gamma$ equipped with a linear Hausdorff topology $\mathfrak{T}_{\mathcal{E}}$
		coarser than the norm topology on $\mathcal{E}$.
		Suppose that}
		\begin{itemize}
            \item[(A1)] {\it the topology $\mathfrak{T_{\mathcal{E}}}$ is pre-Lebesgue (see Definition \ref{pre-Lebesgue});}
			\item[(A2)] {\it the space $( \bigoplus_{\gamma \in \Gamma} X_{\gamma} )_{\mathcal{E}}$ is equipped with the $\oplus$-compatible topology;}
			\item[(A3)] {\it the mapping
			$\lfloor \bullet \rceil \colon ( \bigoplus_{\gamma \in \Gamma} X_{\gamma} )_{\mathcal{E}} \rightarrow \mathcal{E}$
			is $\mathfrak{T}$-to-$\mathfrak{T}_{\mathcal{E}}$ sequentially continuous when restricted to the unit sphere.}
		\end{itemize}
		{\it Then the space $( \bigoplus_{\gamma \in \Gamma} X_{\gamma} )_{\mathcal{E}}$ has the property ${\bf H(\mathfrak{T})}$
		provided all $X_{\gamma}$'s have the appropriate property ${\bf H}(\mathfrak{T}_{\gamma})$
		and $\mathcal{E}$ has the property ${\bf H}(\mathfrak{T}_{\mathcal{E}})$.}
	\end{theorem}
	\begin{proof}
		Take the element $x$ and the sequence $\{x_n\}_{n=1}^{\infty}$ from the unit sphere of $( \bigoplus_{\gamma \in \Gamma} X_{\gamma} )_{\mathcal{E}}$
		such that $x_{n}$ converges in $\mathfrak{T}$ topology to $x$.
		Our goal is to show that
		\begin{equation} \tag{$\spadesuit$} \label{EQ: goal main thm}
			x_{n} \text{ converges to } x \text{ in the norm topology of } (\bigoplus_{\gamma \in \Gamma} X_{\gamma} )_{\mathcal{E}}.
		\end{equation}
		To see this, note that since the mapping
		$\lfloor \bullet \rceil \colon ( \bigoplus_{\gamma \in \Gamma} X_{\gamma} )_{\mathcal{E}} \rightarrow \mathcal{E}$
		is assumed to be $\mathfrak{T}$-to-$\mathfrak{T}_{\mathcal{E}}$ sequentially continuous when restricted to the unit
		sphere and all involved topologies are linear, so
		\begin{equation}
			\lfloor x_{n} \rceil \text{ converges to } \lfloor x \rceil \text{ in the topology } \mathfrak{T}_{\mathcal{E}}.
		\end{equation}
		In consequence, remembering that the space $\mathcal{E}$ has the property ${\bf H}(\mathfrak{T_{\mathcal{E}}})$
		and $\norm{ \lfloor x_{n} \rceil }_{\mathcal{E}} = \norm{ \lfloor x \rceil }_{\mathcal{E}}=1$,
		we infer that
		\begin{equation}
			\lfloor x_{n} \rceil \text{ converges to } \lfloor x \rceil \text{ in the norm topology of } \mathcal{E}.
		\end{equation}
		This means, due to \cite[Lemma~2, p.~97]{KA82}, that there exists a subsequence $\left\{ \lfloor x_{n_k} \rceil \right\}_{k=1}^{\infty}$,
		an element $y \in \mathcal{E}$ and a decreasing null sequence $\{ \varepsilon_k \}_{k=1}^{\infty}$ of positive reals such that
		\begin{equation} \label{EQ : zerowe szacowanie}
			\abs{\lfloor x \rceil - \lfloor x_{n_k} \rceil} \leqslant \varepsilon_k y \quad \text{ for each } \quad k \in \mathbb{N}.
		\end{equation}
		(Using the language of the Riesz spaces, \eqref{EQ : zerowe szacowanie} means that the sequence $\lfloor x_{n} \rceil$ is order convergent to $\lfloor x \rceil$;
		{\it cf}. \cite[Definition~1.1]{Con19}.)
		Denote still this subsequence by $\left\{ \lfloor x_{n} \rceil \right\}_{n=1}^{\infty}$. Fix $\varepsilon > 0$.
		Take a decreasing family $\{\Gamma_n\}_{n=1}^{\infty}$ of subsets of $\Gamma$ with empty intersection such that
		$\#(\Gamma \setminus \Gamma_n) < \infty$ for all $n \in \mathbb{N}$. Since $\mathcal{E}$ has the property $H(\mathfrak{T}_{\mathcal{E}})$,
		by Lemma \ref{LEMMA : H(tau) => OC} and the assumption (A1), the space $\mathcal{E}$ is separable (or, which is one thing in this setting,
		$\mathcal{E}$ is ${\bf OC}$).
		In consequence, there is $n_0 = n_0(\varepsilon) \in \mathbb{N}$ such that
		\begin{equation} \label{EQ : pierwsze szacowanie}
			\norm{(2 \lfloor x \rceil + \varepsilon_1 y) {\bf 1}_{\Gamma_{n_0}}}_{\mathcal{E}} < \frac{\varepsilon}{2}.
		\end{equation}
		Next, due to (A2), the $\mathfrak{T}$ topology is $\oplus$-compatible,
		so $x_n(\gamma)$ converge in $\mathfrak{T}_{\gamma}$ topology to $x(\gamma)$ for each $\gamma \in \Gamma$.
		Furthermore, by (\ref{EQ : zerowe szacowanie}), we conclude that $ \lfloor x_{n} \rceil \rightarrow \lfloor x \rceil $ pointwisely.
		Thus $$\norm{ x(\gamma) - x_n(\gamma) }_{\gamma} \rightarrow 0$$ for each $\gamma \in \Gamma$, because all $X_{\gamma}$'s have the property ${\bf H}(\mathfrak{T}_{\gamma})$.
		Therefore, since $\#(\Gamma \setminus \Gamma_{n_0}) < \infty$, so there is $N = N(\varepsilon) \in \mathbb{N}$ such that
		\begin{equation} \label{EQ : 2 szacowanie}
			\norm{ \sum_{\gamma \in \Gamma \setminus \Gamma_{n_0}} \norm{ x(\gamma) - x_n(\gamma) }_{\gamma} \boldsymbol{e}_{\gamma} }_{\mathcal{E}} < \frac{\varepsilon}{2}
				\quad \text{ for } \quad n \geqslant N.
		\end{equation}
		Thus, for $n \geqslant N$, we have
		\begin{align*}
			\norm{x - x_n}
				& = \norm{ \sum_{\gamma \in \Gamma} \norm{ x(\gamma) - x_n(\gamma) }_{\gamma} \boldsymbol{e}_{\gamma} }_{\mathcal{E}} \\
				& \leqslant \norm{ \sum_{\gamma \in \Gamma \setminus \Gamma_{n_0}} \norm{ x(\gamma) - x_n(\gamma) }_{\gamma} \boldsymbol{e}_{\gamma} }_{\mathcal{E}}
					+ \norm{ \sum_{\gamma \in \Gamma_{n_0}} \norm{ x(\gamma) - x_n(\gamma) }_{\gamma} \boldsymbol{e}_{\gamma} }_{\mathcal{E}} \\
				& < \frac{\varepsilon}{2}
					+ \norm{ \sum_{\gamma \in \Gamma_{n_0}} \abs{ \lfloor x \rceil(\gamma) + \lfloor x_n \rceil(\gamma) } \boldsymbol{e}_{\gamma}}_{\mathcal{E}}
						\quad (\text{by \eqref{EQ : 2 szacowanie} and the $\triangle$-inequality}) \\
				& \leqslant \frac{\varepsilon}{2} + \norm{(2 \lfloor x \rceil + \varepsilon_1 y) {\bf 1}_{\Gamma_{n_0}}}_{\mathcal{E}} \quad (\text{using \eqref{EQ : zerowe szacowanie}}) \\
				& < \frac{\varepsilon}{2} + \frac{\varepsilon}{2} = \varepsilon \quad (\text{using \eqref{EQ : pierwsze szacowanie}}).
		\end{align*}
		This, due to $\varepsilon$'s arbitrariness, proves \eqref{EQ: goal main thm}.
		The proof is complete.
	\end{proof}

	Perhaps only assumptions (A1) and (A3) from Theorem~\ref{THM: main theorem} requires some illustration.
	Let us focus on the former first.
	As we have already mentioned, any locally solid topology that is pre-Lebesgue in the sense of \cite[Definition~8.1, p.~52]{AB03}
	immediately meets condition (A1).
	Thus, for example, the topology of point-wise convergence falls into this pattern.

	\begin{example}[The weak topology satisfies (A1) from Theorem~\ref{THM: main theorem}] \label{EXAMPLE : OC(tau) => OC}
		Take a sequence $\{ x_n \}_{n=1}^{\infty}$ of positive and pairwise disjoint functions from $X$ that is order bounded by $x \in X$.
		Then, for any positive functional $\varphi \in X^*$,
		\begin{equation*}
			\sum_{n=1}^{N} \langle x_n, \varphi \rangle
			= \langle \sum_{n=1}^{N} x_n, \varphi \rangle
			\leqslant \langle x, \varphi \rangle.
		\end{equation*}
		Consequently, $\sum_{n=1}^{\infty} \langle x_n, \varphi \rangle < \infty$ and $\langle x_n, \varphi \rangle \rightarrow 0$.
		However, since any $x^* \in X^* \setminus \{ 0 \}$ is a difference of two positive functionals,
		so it follows that $\langle x_n, x^* \rangle \rightarrow 0$.
		In other words, $\{x_n\}_{n=1}^{\infty}$ is a weakly null sequence.
		That's all.
		\demo
	\end{example}
	
	Let us move on to the condition (A3) from Theorem~\ref{THM: main theorem}.
	Formally, Theorem~\ref{THM: main theorem} completes Theorem~\ref{THM: H is hereditary}.
	However, from a practical point of view, checking the assumption (A3), that is, whether
	the mapping
	$\lfloor \bullet \rceil \colon ( \bigoplus_{\gamma \in \Gamma} X_{\gamma} )_{\mathcal{E}} \rightarrow \mathcal{E}$
	is actually $\mathfrak{T}$-to-$\mathfrak{T}_{\mathcal{E}}$ sequentially continuous when restricted to the unit sphere can be quite tedious.
	In this respect, the conditions (1) and (2) from Theorem~\ref{THM: H is hereditary}, which split into the Schur property
	for $X_{\gamma}$'s and strict monotonicity for $\mathcal{E}$, seems to be much more useful.
	Worse still, at first glance the conditions (1) and (2) from Theorem~\ref{THM: H is hereditary} and the condition (A3) from Theorem~\ref{THM: main theorem}
	do not seem to have much in common.
	But, under some mild assumptions, they have.

 \begin{definition}
   Let $ \mathcal{E} $ be a Banach sequence space.
   Following \cite{KA82}, we say that the unit ball $\text{Ball}(\mathcal{E})$ {\bf is sequentially closed with respect to the point-wise topology}
   provided for each $x \in \omega$ and any sequence $\{ x_n \}_{n=1}^{\infty}$ in $\text{Ball}(\mathcal{E})$ converging point-wisely to $x$
   it follows that $x \in \text{Ball}(\mathcal{E})$.
 \end{definition}

   Note that the unit ball $\text{Ball}(\mathcal{E})$ is sequentially closed with respect to the point-wise topology if, and only if, the space $\mathcal{E}$
   has the Fatou property (see Remark \ref{REMARK : o warunkach w twierdzeniu}(A1), {\it cf}. \cite[Lemma~1.5, p.~4]{BS88}). 

	\begin{theorem}[Compatibility result] \label{THM : oba warunki sa rownowazne}
		{\it Let $\{ X_{\gamma} \}_{\gamma \in \Gamma}$ be a family of Banach spaces equipped with a linear Hausdorff topologies $\mathfrak{T}_{\gamma}$
		coarser than the corresponding norm topologies on $X_{\gamma}$'s.
		Further, let $ \mathcal{E} $ be a Banach sequence space equipped with a linear Hausdorff topology $\mathfrak{T}_{\mathcal{E}}$
		coarser than the norm topology on $\mathcal{E}$.
		Suppose that}
		\begin{itemize}
			\item[(A1)] {\it the unit ball } $\text{Ball}(\mathcal{E})$ {\it is sequentially closed with respect to the point-wise topology;}
			\item[(A2)] {\it the space $ ( \bigoplus_{\gamma \in \Gamma} X_{\gamma} )_{\mathcal{E}} $ is equipped with the
			$\oplus$-compatible topology $\mathfrak{T}$;}
			\item[(A3)] {\it for positive sequences on the unit sphere in $\mathcal{E}$ the topology $\mathfrak{T}_{\mathcal{E}}$
			agree with the topology of point-wise convergence;}
			\item[(A4)] {\it for every $\gamma \in \Gamma$ the norm function $x \mapsto \norm{x}_{\gamma}$
			is sequentially lower semi-continuous with respect to $\mathfrak{T}_{\gamma}$;}
			
		\end{itemize}
		{\it Then the mapping
		$\lfloor \bullet \rceil \colon ( \bigoplus_{\gamma \in \Gamma} X_{\gamma} )_{\mathcal{E}} \rightarrow \mathcal{E}$
		is $\mathfrak{T}$-to-$\mathfrak{T}_{\mathcal{E}}$ sequentially continuous when restricted to the unit sphere if,
		and only if, the set $\Gamma$ can be decomposed into two disjoint subsets, say $\Gamma_1$ and $\Gamma_2$,
		in such a way that all $X_{\gamma}$'s with $\gamma \in \Gamma_1$ have the Schur property with respect to $\mathfrak{T}_{\gamma}$
		and $\mathcal{E}$ is ${\bf SM(\gamma)}$ for $\gamma \in \Gamma_2$.}
	\end{theorem}
	\begin{proof}
		The necessity is quite clear.
		To see this, it is enough to properly interpret the third part of the proof of Theorem~\ref{THM: H is hereditary}. Note that $\lfloor y_{n} \rceil$ does not converge pointwisely to $\lfloor y \rceil$ and apply the assumption (A3).
		
		Now we prove the sufficiency.
		Suppose that the set $\Gamma$ can be decomposed into two disjoint subsets, say $\Gamma_1$ and $\Gamma_2$,
		in such a way that all $X_{\gamma}$'s with $\gamma \in \Gamma_1$ have the Schur property with respect to $\mathfrak{T}_{\gamma}$
		and $\mathcal{E}$ is ${\bf SM(\gamma)}$ for $\gamma \in \Gamma_2$.
		Take the sequence $\{x_n\}_{n=1}^{\infty}$ from the unit sphere of $( \bigoplus_{\gamma \in \Gamma} X_{\gamma} )_{\mathcal{E}} $
		that converge in $\mathfrak{T}$ topology to the norm one vector, say $x$.
		Clearly, both $\lfloor x_{n} \rceil$ with $n \in \mathbb{N}$ and $\lfloor x \rceil$ are of norm one in $\mathcal{E}$.
		We want to show that
		\begin{equation} \tag{$\clubsuit$} \label{EQ : point-wise}
			\lfloor x_{n} \rceil \text{ converges to } \lfloor x \rceil \text{ in the point-wise topology}.
		\end{equation}
		Fix $\gamma_0 \in \Gamma$.
		Clearly, since $x_{n}$ converge in $\mathfrak{T}$ topology to $x$ and, due to (A2),
		the topology $\mathfrak{T}$ is $\oplus$-compatible, so
		\begin{equation} \label{EQ : 321}
			x_{n}(\gamma_0) \text{ converges to } x(\gamma_0) \text{ in the topology } \mathfrak{T}_{\gamma_0}.
		\end{equation}
		Now, let us consider two situations.
		
		{\bf Suppose that $\gamma_0 \in \Gamma_1$.} Then the space $X_{\gamma_0}$ has the Schur property with respect to
		the topology $\mathfrak{T}_{\gamma_0}$.
		Consequently, remembering about \eqref{EQ : 321}, $x_{n}(\gamma_0)$ converge to $x(\gamma_0)$ in norm.
		However, since the norm function $\norm{\bullet}_{\gamma} \colon X_{\gamma} \rightarrow \mathbb{R}$ is continuous,
		so $\lfloor x_{n} \rceil$ converge in the point-wise topology to $\lfloor x \rceil$ on $\Gamma_1$.
		
		{\bf Next, suppose that $\gamma_0 \in \Gamma_2$.}
		Applying the assumptions (A2) and (A4) we conclude that
		\begin{equation}\label{X_tal-lower-semicont}
			\norm{x(\gamma)}_{\gamma} \leqslant \liminf_{n \rightarrow \infty} \norm{x_n(\gamma)}_{\gamma}
		\end{equation}
		for each $\gamma \in \Gamma$. Moreover,
		\begin{equation*}
			\norm{x_{n}(\gamma)}_{\gamma} \norm{\boldsymbol{e}_{\gamma}}_{\mathcal{E}}
				= \norm{\lfloor x_{n} \rceil (\gamma)}_{\mathcal{E}}
				\leqslant \norm{\lfloor x_{n} \rceil}_{\mathcal{E}},
		\end{equation*}
		so the sequence $\{ \norm{x_n(\gamma)}_{\gamma} \}_{n=1}^{\infty}$ is bounded for each $\gamma \in \Gamma$.
		Thus $\{ \norm{x_n(\gamma)}_{\gamma} \}_{n=1}^{\infty}$ contains a convergent subsequence.
		Using the diagonal argument, and passing to a subsequence if necessary, we can find a function $f = \{ f(\gamma) \}_{\gamma \in \Gamma}$
		such that $\lfloor x_{n} \rceil$ converges to $f$ in the point-wise topology.
		Now, remembering about the assumption (A1), we infer that $f \in \mathcal{E}$ and
		\begin{equation} \label{EQ : <= 1}
			\norm{f}_{\mathcal{E}} \leqslant 1.
		\end{equation}
		We claim that
		\begin{equation} \tag{$\heartsuit$} \label{EQ : claim}
			\lfloor x \rceil(\gamma_0) = \abs{f}(\gamma_0).
		\end{equation}
		Clearly, $\lfloor x \rceil(\gamma) \leqslant \abs{f}(\gamma)$ for each $\gamma \in \Gamma$, because otherwise we get a contradiction with (\ref{X_tal-lower-semicont}).
		Therefore, $\lfloor x \rceil(\gamma_0) \leqslant \abs{f}(\gamma_0)$ and it remains to show the reverse inequality.
		To see this, suppose that $\lfloor x \rceil(\gamma_0) < \abs{f}(\gamma_0)$.
		Then, since $\mathcal{E}$ is ${\bf SM}(\gamma_0)$, so $\norm{\lfloor x \rceil}_{\mathcal{E}} < \norm{f}_{\mathcal{E}}$.
		But then, using \eqref{EQ : <= 1},
		\begin{equation*}
			1 = \norm{\lfloor x \rceil}_{\mathcal{E}}
				< \norm{f}_{\mathcal{E}}
				\leqslant 1,
		\end{equation*}
		which is impossible.
		In consequence, the equality \eqref{EQ : claim} holds.
		
		In view of $\gamma_0$'s arbitrariness the claim \eqref{EQ : point-wise} follows.
		However, since we assumed that the topology $\mathfrak{T}_{\mathcal{E}}$ restricted to the unit
		sphere agree with the topology of point-wise convergence, so actually $\lfloor x_{n} \rceil$ converge
		in $\mathfrak{T_{\mathcal{E}}}$ topology to $\lfloor x \rceil$.
		In other words, the mapping
		$\lfloor \bullet \rceil \colon ( \bigoplus_{\gamma \in \Gamma} X_{\gamma} )_{\mathcal{E}} \rightarrow \mathcal{E}$
		is $\mathfrak{T}$-to-$\mathfrak{T}_{\mathcal{E}}$ sequentially continuous when restricted to the unit sphere.
		The proof is complete.
	\end{proof}

	Let us comment on the assumptions (A1), (A3) and (A4) that appear in Theorem~\ref{THM : oba warunki sa rownowazne}.
	Even though is does not look like this, with some rather weak and natural assumptions on $\mathcal{E}$ and $X_{\gamma}$'s
	they are all automatically satisfied.

	\begin{remark}[About Theorem~\ref{THM : oba warunki sa rownowazne}] \label{REMARK : o warunkach w twierdzeniu}
		\hfill
		
		{\bf (A1)} Let $X$ be a Banach sequence space.
		It follows from \cite[Lemma~5, p.~99]{KA82} that $\text{Ball}(X)$ is sequentially closed with respect to
		the point-wise topology, if, and only if, the space $X$ has the {\it weak Fatou property} (this property is sometimes called
		the {\it semi-Fatou property} or the {\it order semi-continuity}) and is {\it monotone complete}\footnote{Solid linear topologies
		with this property are called the {\it Levi topologies} (see, for example, \cite[Definition~9.3, p.~61]{AB03}).}.
		This means, due to \cite[Lemma~4]{KA82}, that
  		\begin{itemize}
  			\item[$\bullet$] the norm function $x \mapsto \norm{x}_X$ is sequentially
  			lower semi-continuous with respect to the point-wise topology (the {\bf weak Fatou property});
  			\item[$\bullet$] for any increasing sequence $\{ x_n \}_{n=1}^{\infty}$
  			of positive functions from $\text{Ball}(X)$, it follows that $x_n$ converges to some $x \in X$ in the point-wise topology
  			(the {\bf monotone completeness}).
  		\end{itemize}
  		In the realm of Banach ideal spaces one usually refer to the conglomerate of both properties, that is, the weak Fatou property
		together with monotone completeness, as the {\it Fatou property}.
		Note that the classical sequence spaces, like Lebesgue spaces $\ell_p$, Orlicz spaces $\ell_F$ and Lorentz spaces $d(w,p)$, have the Fatou property.
		On the other hand, the generic\footnote{One can show that a separable Banach sequence space fails to have the Fatou property if, and only if,
		the space $X$ contains a subspace isomorphic to $c_0$ (see, for example, \cite[Theorem~2.4.12, p.~92]{MN91}). Note also that separable
		Banach sequence spaces with the Fatou property are sometimes called the {\it KB-spaces} (see \cite[Definition~2.4.11, p.~92]{MN91}).}
		example of the Banach sequence space which fails to have the Fatou property is $c_0$.
		The culprit for this is the sequence
		\begin{equation*}
			x_1 = (1, 0, 0, 0, ...), \quad x_2 = (1, 1, 0, 0, ...), \quad x_3 = (1, 1, 1, 0, ...), \quad \text{ and so on.}
		\end{equation*}
		However, this is not the only specimen of this type.
		In fact, one can easily construct many more examples in the following way:
		For a given Banach sequence space $X$ with $X_o \neq X$ put
		\begin{equation} \label{EQ : not Fatou}
			\vertiii{x} \coloneqq \norm{x}_X + \lambda \, \text{dist}(x,X_o),
		\end{equation}
		where $\lambda > 0$ and $x \mapsto \text{dist}(x,X_o) \coloneqq \inf \{ \norm{x - y}_X \colon y \in X_o \}$ is the distance from $x \in X$ to the ideal $X_o$.
		Evidently, \eqref{EQ : not Fatou} is an equivalent norm on $X$.
		Moreover, the space $X$ furnished with this norm fails to have the Fatou property.
		
		{\bf (A3)} This assumption seems to be specifically adapted to the Kadets--Klee property.
		Plainly, if the space $\mathcal{E}$ has the property ${\bf H}(\textit{point-wise})$
		then for positive sequences on the unit sphere in $\mathcal{E}$ the topology $\mathfrak{T}_{\mathcal{E}}$
		agree with the topology of point-wise convergence.
		No frills, one can just assume that $\mathcal{E}$ has the property ${\bf H}(\textit{point-wise})$ (see Remark~\ref{REMARK : minimal topologies}).
		
		{\bf (A4)} This is a very lenient requirement.
		Once again, in view of \cite[Lemma~4]{KA82}, a Banach sequence space $X$ has the weak Fatou property if, and only if,
		the norm function $x \mapsto \norm{x}_X$ is sequentially lower semi-continuous with respect to point-wise topology.
		Note that $c_0$ and, more generally, any subspace $X_o$ of order continuous element in $X$, have also this property
		(but, as we already explained above, fails to have the Fatou property!).
		\demo 
	\end{remark}

	\section{{\bf On $\oplus$-compatible topologies}} \label{SECTION : Play Doh}
	
	The practical utility of Theorems~\ref{THM: H is hereditary}, \ref{THM: main theorem} and \ref{THM : oba warunki sa rownowazne}
	is still somehow doubtful unless we can find some natural implementations of Definition~\ref{DEF: admissible topology}
	in the form of $\oplus$-compatible topologies.
	Fortunately, there is a plethora of natural and important examples of such topologies.
	Of the most obvious ones, which can be depicted below
	\begin{center}
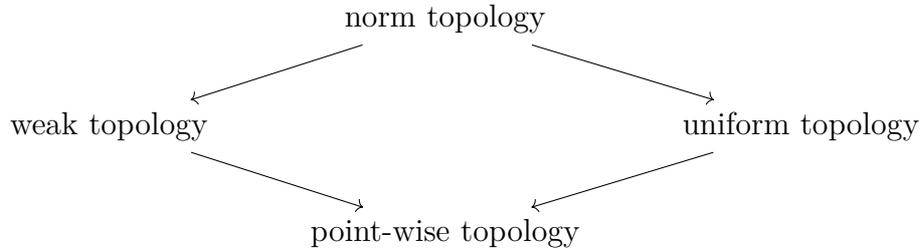

		\begin{tikzcd}
			& \text{norm topology} \arrow[ld] \arrow[rd] &              \\
			\text{weak topology} \arrow[rd] &                         & \text{uniform topology} \arrow[ld] \\
			& \text{point-wise topology}                       &             
		\end{tikzcd}
		\captionof{figure}{Three natural topologies on a Banach sequence space}
	\end{center}
	we will be mainly interested here in topologies that, so to speak, fill the space between them.
	Anyway, we refer the reader not interested in topological divagations directly to Section~\ref{SECTION : EXAMPLES} where we will
	focus exclusively on the weak topology and the point-wise topology.
	(Note also that the topology of uniform convergence defies naive analysis, so we will postpone its study for another occasion; {\it cf}. Table~\ref{TABLE : trzy H}.)

	\subsection{Projective topologies} \label{SUBSECTION : Projective toplogies}
	
	Let $X$ be a vector space and let $\{ X_j \}_{j \in J}$ be a family of topological vector spaces.
	Further, let $\{ \varphi_j \colon X \rightarrow X_j \}_{j \in J}$ be a family of linear mappings hereinafter called a {\bf spectrum}.
	The coarsest topology on $X$ for which all $\varphi_j$'s are continuous is called the {\bf projective topology}
	generated by the spectrum $\{ \varphi_j \colon X \rightarrow X_j \}_{j \in J}$.
	In other words, the projective topology generated by the spectrum $\{ \varphi_j \colon X \rightarrow X_j \}_{j \in J}$
	is the least upper bound of those topologies which are inverse images of the topologies of the space $X_j$ by the mapings $\varphi_j$.
	Without much effort one can show that the projective topology is linear (see, for example, \cite[Proposition~1, p.~35]{Jar81}).
	When equipped with this topology, the space $X$ is called the {\bf projective limit} of $\{ \varphi_j \colon X \rightarrow X_j \}_{j \in J}$.
	
	Due to the universal property of the product topology, the spectrum $\{ \varphi_j \colon X \rightarrow X_j \}_{j \in J}$
	determines a unique continuous mapping $\varphi \colon X \rightarrow \prod_{j \in J} X_j$ via $\varphi(x) \coloneqq \{ \varphi_j(x) \}_{j \in J}$.
	We will call it the {\bf evaluation map}.
	The evaluation map $\varphi$ is injective if, and only if, the spectrum $\{ \varphi_j \colon X \rightarrow X_j \}_{j \in J}$
	{\bf separates points}, that is, for all $x,y \in X$, there exists $j \in J$ such that $\varphi_j(x) \neq \varphi_j(y)$, whenever $x \neq y$.
	Actually, in such a situation, the evaluation map $\varphi$ is an isomorphism of $X$ endowed with the projective topology
	onto a subspace of the product $\prod_{j \in J} X_j$.
	
	Let us now explain how to generate the topology on the direct sum $(\bigoplus_{\gamma \in \Gamma} X_{\gamma})_{\mathcal{E}}$
	that agree with a given projective topologies on their components.
	The construction is rather straightforward, but writing out the details requires a bit of honest work.
	
	\begin{construction}[Projective topology on direct sum] \label{CONSTRUCTION : projective}
		Let $\{ X_{\gamma} \}_{\gamma \in \Gamma}$ be a family of Banach spaces and let $\mathcal{E}$ be a Banach sequence space on $\Gamma$.
		We equip the spaces $X_{\gamma}$ and the space $\mathcal{E}$ with the projective topologies $\mathfrak{T}_{\gamma}(\textit{proj})$ and $\mathfrak{T}_{\mathcal{E}}(\textit{proj})$
		generated by the spectra $\{ \varphi_j^{\gamma} \colon X_{\gamma} \rightarrow Y_j \}_{j \in J}$ and, respectively, $\{ \psi_j \colon \mathcal{E} \rightarrow Y_j \}_{j \in J}$.
		Suppose that
		\begin{itemize}
			\item[(P1)] all topological spaces $Y_j$ are Hausdorff;
			\item[(P2)] both families $\{ \varphi_j^{\gamma} \colon X_{\gamma} \rightarrow Y_j \}_{j \in J}$ and $\{ \psi_j \colon \mathcal{E} \rightarrow Y_j \}_{j \in J}$
			separates points;
			\item[(P3)] for all $\gamma \in \Gamma$ the mapping $\varphi_j^{\gamma} \colon X_{\gamma} \rightarrow Y_j$ is continuous when $X_{\gamma}$
			is equipped with the norm topology;
			\item[(P4)] and, similarly, all mappings $\psi_j \colon \mathcal{E} \rightarrow Y_j$ are continuous when $\mathcal{E}$
			is equipped with the norm topology.
		\end{itemize}
		
		It is straightforward to see, that the above four conditions guarantee that the projective topologies generated by the spectra
		$\{ \varphi_j^{\gamma} \colon X_{\gamma} \rightarrow Y_j \}_{j \in J}$ and $\{ \psi_j \colon \mathcal{E} \rightarrow Y_j \}_{j \in J}$
		are Hausdorff and coarser than the respective norm topologies (see \cite[Proposition~2.4.4, p.~204]{Meg98}.
		
		Now, let us consider the following diagram
		\begin{equation} \label{EQ : diagram projective}
			\begin{tikzcd}
				\mathcal{E} \arrow[rd, "\psi_j"']
				& ( \bigoplus_{\gamma \in \Gamma} X_{\gamma} )_{\mathcal{E}} \arrow[d, dotted, shift left] \arrow[l, "\pi"'] \arrow[r, "\pi_{\gamma}"] \arrow[d, dotted, shift right]
				& X_{\gamma} \arrow[ld, "\varphi_j^{\gamma}"] \\
				& Y_j                                                  &                  
			\end{tikzcd}
		\end{equation}
		The presence of an arrow $\pi$ between $( \bigoplus_{\gamma \in \Gamma} X_{\gamma} )_{\mathcal{E}}$ and $\mathcal{E}$ is due to the fact that
		the space $( \bigoplus_{\gamma \in \Gamma} X_{\gamma} )_{\mathcal{E}}$ contains a complemented subspace isometrically isomorphic to $\mathcal{E}$.
		Looking at the diagram \eqref{EQ : diagram projective}, it is natural to equip the space $( \bigoplus_{\gamma \in \Gamma} X_{\gamma} )_{\mathcal{E}}$
		with the projective topology generated by the, so to speak, joint spectrum
		\begin{equation} \label{EQ : spectrum direct sum}
			\{ \varphi_j^{\gamma} \circ \pi_{\gamma} \colon ( \bigoplus_{\gamma \in \Gamma} X_{\gamma} )_{\mathcal{E}} \rightarrow Y_j \}_{(j,\gamma) \in J \times \Gamma}
			\cup \{ \psi_j \circ \pi \colon ( \bigoplus_{\gamma \in \Gamma} X_{\gamma} )_{\mathcal{E}} \rightarrow Y_j \}_{j \in J}.
		\end{equation}
		In what follows we will refer to this topology as $\mathfrak{T}(\textit{proj})$ topology.
		\demo
	\end{construction}
	
	\begin{lemma}[$\oplus$-compatible projective topology] \label{projective+compatible}
		{\it Suppose all assumptions of Construction \ref{CONSTRUCTION : projective} are satisfied.
		Then the projective topology $\mathfrak{T}(\textit{proj})$ on the space $( \bigoplus_{\gamma \in \Gamma} X_{\gamma} )_{\mathcal{E}}$
		generated by the spectrum \eqref{EQ : spectrum direct sum} is the $\oplus$-compatible topology with the projective topologies of their components.}
	\end{lemma}
	\begin{proof}
		Due to (P1) and (P2) the topology $\mathfrak{T}(\textit{proj})$ is a linear Hausdorff topology on $( \bigoplus_{\gamma \in \Gamma} X_{\gamma} )_{\mathcal{E}}$.
		Now, according to Definition~\ref{DEF: admissible topology}, it remains to show the conditions (C1), (C2), (C3) and (C4).
				
		{\bf (C1)} 
		Since the mapping $\pi_{\gamma} \colon ( \bigoplus_{\gamma \in \Gamma} X_{\gamma} )_{\mathcal{E}} \rightarrow X_{\gamma}$
		is norm-to-norm continuous, so using (P3), we infer that for all $\gamma \in \Gamma$ and $j \in J$, the mapping $\varphi_j^{\gamma} \circ \pi_{\gamma}$
		is continuous when acting from $( \bigoplus_{\gamma \in \Gamma} X_{\gamma} )_{\mathcal{E}}$ with the norm topology into $Y_j$.
		Clearly, remembering about (P4), the same can be said about $\psi_j \circ T$.
		Consequently, the projective topology on $( \bigoplus_{\gamma \in \Gamma} X_{\gamma} )_{\mathcal{E}}$ is coarser than the norm topology.
		
		{\bf (C2)}
		Next, we need to show that for each $\gamma \in \Gamma$ the projection
		$\pi_{\gamma} \colon ( \bigoplus_{\gamma \in \Gamma} X_{\gamma} )_{\mathcal{E}} \rightarrow X_{\gamma}$
		is $\mathfrak{T}(\textit{proj})$-to-$\mathfrak{T}_{\gamma}(\textit{proj})$ sequentially continuous.
		Let us consider the following diagram
		\begin{equation*}
			\begin{tikzcd}
				(\bigoplus_{\gamma \in \Gamma} X_{\gamma})_{\mathcal{E}} \arrow[d, "\varphi_j^{\gamma} \circ \pi_{\gamma}"'] \arrow[r, "\pi_{\gamma}"] & X_{\gamma} \arrow[ld, "\varphi_j^{\gamma}"] \\
				Y_j                   &                       
			\end{tikzcd}
		\end{equation*}
		Due to \cite[Proposition~1(1), p.~2]{Gro73}, $\pi_{\gamma}$ is continuous if, and only if, $\pi_{\gamma} \circ \varphi_j^{\gamma}$ is continuous.
		However, directly from the definition \eqref{EQ : spectrum direct sum}, the mapping $\pi_{\gamma} \circ \varphi_j^{\gamma}$ is continuous when acting
		from the space $( \bigoplus_{\gamma \in \Gamma} X_{\gamma} )_{\mathcal{E}}$ equipped with the projective topology $\mathfrak{T}(\textit{proj})$.
		
		{\bf (C3)}
		Now, we are going to show that the embedding
		$\jmath_{\gamma} \colon X_{\gamma} \rightarrow ( \bigoplus_{\gamma \in \Gamma} X_{\gamma} )_{\mathcal{E}}$
		is $\mathfrak{T}_{\gamma}(\textit{proj})$-to-$\mathfrak{T}(\textit{proj})$ sequentially continuous.
		Let us take a look at the diagram
		\begin{equation*}
			\begin{tikzcd}
				(\bigoplus_{\gamma \in \Gamma} X_{\gamma})_{\mathcal{E}} \arrow[d, "\varphi_j^{\gamma} \circ \pi_{\gamma}"'] & X_{\gamma} \arrow[l, "j_{\gamma}"'] \arrow[ld, "\varphi_j^{\gamma}"] \\
				Y_j                   &                       
			\end{tikzcd}
		\end{equation*}
		Again, due to \cite[Proposition~1(1), p.~2]{Gro73}, $\jmath_{\gamma}$ is continuous if, and only if, $(\varphi_j^{\gamma} \circ \pi_{\gamma}) \circ j_{\gamma}$
		is continuous.
		However, since
		\begin{equation*}
			(\varphi_j^{\gamma} \circ \pi_{\gamma}) \circ j_{\gamma} = \varphi_j^{\gamma},
		\end{equation*}
		so our claim follows.   
		
		{\bf (C4)} Finally, it remains to show that the embedding 
		$j_{\mathcal{E}} \colon \mathcal{E} \rightarrow ( \bigoplus_{\gamma \in \Gamma} X_{\gamma} )_{\mathcal{E}}$
		is $\mathfrak{T}_{\mathcal{E}}(\textit{proj})$-to-$\mathfrak{T}(\textit{proj})$ sequentially continuous.
		Luckily, this is no different from the proof of (C3).
	\end{proof}
	
	Several topological constructions like, for example:
	$\bullet$ the subspace topology;
	$\bullet$ the product topology;
	$\bullet$ the inverse limit;
	$\bullet$ and the weak topology, can be seen as just a special cases of the projective topology (see, for example, \cite[Chapter~0]{Gro73}).
	
	To conclude this section, let us try to say something concrete.
	
	\begin{example}[Projective topology between norm and weak topologies]
		Let $X$ be a Banach space. Further, let $Y$ be a closed subspace of $X$ that has infinite dimension and infinite co-dimension.
		Let us consider the projective topology $\mathfrak{T}$ on $X$ generated by the joint spectrum
		$\{ \varphi \colon X \rightarrow \mathbb{K} \}_{\varphi \in X^{*}}$ and $\{ q \colon X \rightarrow X / Y \}$,
		where $q \colon X \rightarrow X / Y$ is the quotient map.
		We claim that $\mathfrak{T}$ is coarser than the norm topology and finer than the weak topology.
		Indeed, since $Y \subset X$, so $\mathfrak{T}$ restricted to $Y$ is nothing else, but the weak topology.
		Thus, since $Y$ is infinite dimensional, it follows that $\mathfrak{T}$ is coarser than the norm topology.
		Going ahead, just from the definition, $q \colon X \rightarrow X / Y$ is $\mathfrak{T}$-to-norm continuous.
		However, since $X / Y$ is infinite dimensional, so $\mathfrak{T}$ must be finer than the weak topology.
		\demo
	\end{example}

	\subsection{Inductive topologies}
	
	Let $X$ be a vector space.
	Further, let $\{ X_j \}_{j \in J}$ be a family of topological vector spaces and let $\varphi_j \colon X_j \rightarrow X$
	be a linear mapping.
	We call the finest vector topology on $X$ for which all $\varphi_j$'s are continuous the {\bf inductive topology} for
	the spectrum $\{ \varphi_j \colon X_j \rightarrow X \}_{j \in J}$ (such a topology always exists; see, for example, \cite[Proposition~1, p.~74]{Jar81}).
	When equipped with this topology, the space $X$ is called the {\bf inductive limit} of $\{ \varphi_j \colon X_j \rightarrow X \}_{j \in J}$.
	
	By dualizing Construction~\ref{CONSTRUCTION : projective} we obtain
	
	\begin{construction}[Inductive topology on direct sum] \label{CONSTRUCTION : inductive}
		Let $\{ X_{\gamma} \}_{\gamma \in \Gamma}$ be a countably family of Banach spaces and let $\mathcal{E}$ be a Banach sequence space on $\Gamma$.
		We equip the spaces $X_{\gamma}$ and the space $\mathcal{E}$ with the inductive topologies $\mathfrak{T}_{\gamma}(\textit{ind})$ and $\mathfrak{T}_{\mathcal{E}}(\textit{ind})$
		generated by the spectra $\{ \varphi_j^{\gamma} \colon Y_j \rightarrow X_{\gamma} \}_{j \in J}$ and, respectively, $\{ \psi_j \colon Y_j \rightarrow \mathcal{E} \}_{j \in J}$.
		Suppose that
		\begin{itemize}
			\item[(I1)] all topological spaces $Y_j$ are Hausdorff;
			\item[(I2)] both families $\{ \varphi_j^{\gamma} \colon Y_j \rightarrow X_{\gamma} \}_{j \in J}$ and $\{ \psi_j \colon Y_j \colon \rightarrow \mathcal{E} \}_{j \in J}$
			separates points;
			\item[(I3)] for all $\gamma \in \Gamma$ the mapping $\varphi_j^{\gamma} \colon Y_j \rightarrow X_{\gamma}$ is continuous when $X_{\gamma}$
			is equipped with the norm topology;
			\item[(I4)] and, similarly, all mappings $\psi_j \colon Y_j \rightarrow \mathcal{E}$ are continuous when $\mathcal{E}$
			is equipped with the norm topology.
		\end{itemize}
		The inductive topology on $(\bigoplus_{\gamma \in \Gamma} X_{\gamma})_{\mathcal{E}}$ generated by the spectrum
		\begin{equation} \label{EQ : ind spectrum}
			\{ j_{\gamma} \circ \varphi_j^{\gamma} \colon Y_j \rightarrow ( \bigoplus_{\gamma \in \Gamma} X_{\gamma} )_{\mathcal{E}} \}_{(j,\gamma) \in J \times \Gamma}
			\cup \{ j_{\mathcal{E}} \circ \psi_j \colon Y_j \rightarrow ( \bigoplus_{\gamma \in \Gamma} X_{\gamma} )_{\mathcal{E}} \}_{j \in J}.
		\end{equation}
		will be denoted by $\mathfrak{T}(\textit{ind})$.
		\demo
	\end{construction}

	Looking at Lemma~\ref{projective+compatible}, it is routine to verify the following
	
	\begin{lemma}[$\oplus$-compatible inductive topology]
		{\it Suppose all assumptions of Construction \ref{CONSTRUCTION : inductive} are satisfied.
		Then the inductive topology $\mathfrak{T}(\textit{ind})$ on the space $( \bigoplus_{\gamma \in \Gamma} X_{\gamma} )_{\mathcal{E}}$
		generated by the spectrum \eqref{EQ : ind spectrum} is the $\oplus$-compatible topology with the inductive topologies of their components.}
	\end{lemma}

	Many important constructions of topologies like, for example:
	$\bullet$ the quotient topology;
	$\bullet$ the disjoint union topology;
	$\bullet$ the direct limit;
	and $\bullet$ the weak$^{*}$ topology, are just a special instances of the inductive topology construction (see, for example, \cite[Chapter~0]{Gro73}).
	
	\subsection{Topologies generated by order ideals}
	
	Throughout this section, by $X_{\gamma}$ with $\gamma \in \Gamma$ we will mean a Banach function space with the Fatou property defined over
	a complete and $\sigma$-finite measure space $(\Omega_{\gamma},\Sigma_{\gamma},\mu_{\gamma})$ (see Notation~\ref{NOTATION : direct sum as BFS} for references).
	Then, it is routine to verify that the space $( \bigoplus_{\gamma \in \Gamma} X_{\gamma} )_{\mathcal{E}}$ can be viewed as
	a Banach function space itself.
	To see this, just note that the space $( \bigoplus_{\gamma \in \Gamma} X_{\gamma} )_{\mathcal{E}}$ is a linear subspace of $L_0(\Delta,\mathcal{A},m)$
	with the ideal property.
	Here, by $\Delta$ we understand a disjoint union $\bigsqcup_{\gamma \in \Gamma} \Omega_{\gamma}$ of $\Omega_{\gamma}$'s,
	while the measure $m$ is defined as
	$$m(A) \coloneqq \sum_{\gamma \in \Gamma} \mu_{\gamma}(A \cap \Omega_{\gamma})$$
	for $A \subset \bigsqcup_{\gamma \in \Gamma} \Omega_{\gamma}$.
	Moreover, as is easy to check, the space $( \bigoplus_{\gamma \in \Gamma} X_{\gamma} )_{\mathcal{E}}$ has the Fatou property.
	
	\begin{construction}[Weak topologies generated by order ideals]
		Let $X$ be a Banach function space on $(\Omega,\Sigma,\mu)$ with the Fatou property.
		Further, let $\mathcal{J}$ be an order ideal of the K{\" o}the dual $X^{\times}$ of $X$ containing simple functions
		(see \cite[Definition~3.7, p.~16]{BS88}).
		Following \cite[p.~24]{BS88}, let us consider the family of semi-norms $p \colon X \rightarrow \mathbb{R}$ defined as
		\begin{equation} \label{EQ : seminorms}
			p(f) \coloneqq \abs{ \int_{\Omega} fg \, d \mu},
		\end{equation}
		where $f \in X$ and $g \in \mathcal{J}$.
		Since $\mathcal{J}$ is a norm-fundamental\footnote{Recall that a closed linear subspace $Y$ of the topological dual $X^{*}$
		of a Banach space $X$ is said to be {\bf norm-fundamental} if
		$\norm{x}_X = \sup\{ \abs{\langle x^{*}, x \rangle} \colon x^{*} \in Y \text{ and } x^{*} \in \text{Ball}(X^{*}) \}$.
		In other words, $Y$ is norm-fundamental provided it contains enough forms to reproduce the norm of any $x \in X$.}
		subspace of $X^{*}$, so the collection \eqref{EQ : seminorms} is a separating family which endows $X$ with the structure
		of a Hausdorff locally convex topological vector space.
		We will denote this topology by $\sigma(X,\mathcal{J})$.
		Clearly,
		\begin{equation} \label{EQ : WKDC1}
			f_n \overset{\sigma(X,\mathcal{J})}{\rightarrow} f
		\end{equation}
		if, and only if,
		\begin{equation} \label{EQ : WKDC2}
			\int_{\Omega} f_n g \, d \mu \rightarrow \int_{\Omega} f g \, d \mu \quad \text{ for each } \quad g \in \mathcal{J}.
		\end{equation}
		\demo
	\end{construction}
	
	There are perhaps two natural ideals in $X^{\times}$ containing simple functions, namely,
	$\Lambda_{X^{\times}}$, that is, the closure of simple functions in $X^{\times}$, and $X^{\times}$ itself.
	Let us give them some due attention (see also \cite{CDSS96}).
	Plainly, the $\sigma(X,\Lambda_{X^{\times}})$-topology coincide with the weak topology if $X$ has a separable dual
	(this is easy, since $X^{*}$ is separable, so $X$ is separable as well and, in consequence, $X^{*} = X^{\times} = \Lambda_{X^{\times}}$),
	and with the weak$^{*}$ topology if $X$ is a dual space of some separable Banach function space.
	Moreover, if $X$ is separable, then the weak topology on $X$ is nothing else but $\sigma(X,X^{\times})$.
	In general, however, $\sigma(X,X^{\times})$ is finer than the weak topology, but also coarser than the norm topology.
	
	\begin{lemma}
		{\it Both topologies $\sigma(\mathfrak{X},\Lambda_{\mathfrak{X}^{\times}})$ and $\sigma(\mathfrak{X},\mathfrak{X}^{\times})$,
		where $\mathfrak{X} = ( \bigoplus_{\gamma \in \Gamma} X_{\gamma} )_{\mathcal{E}}$, are $\oplus$-compatible.}
	\end{lemma}
	\begin{proof}
		We will show details only for the $\sigma(\mathfrak{X},\mathfrak{X}^{\times})$-topology.
		The rest is quite similar.
		
		Denote $\mathfrak{X} = ( \bigoplus_{\gamma \in \Gamma} X_{\gamma} )_{\mathcal{E}}$.
		Note also that the K{\" o}the dual of $\mathfrak{X}^{\times}$ of $\mathfrak{X}$ can be identified with
		$( \bigoplus_{\gamma \in \Gamma} X_{\gamma}^{\times} )_{\mathcal{E}^{\times}}$
		(see \cite[Proposition~4.a.5]{KT24} and \cite[pp.~175--178]{Lau01}; {\it cf}. Proposition~\ref{PROPOSITION : duality of direct sums}).
		Showing that conditions (C1), (C2) and (C3) from Definition~\ref{DEF: admissible topology} hold is rather obvious,
		so let us focus entirely on the last one, that is, (C4).
		Take a sequence $\{ x_n \}_{n=1}^{\infty}$ from $\mathcal{E}$ that converges in the $\sigma(\mathcal{E},\mathcal{E}^{\times})$-topology
		to some $x$ in $\mathcal{E}$, that is,
		\begin{equation} \label{EQ : seq weak conv}
			\sum_{\gamma \in \Gamma} x_n(\gamma) \psi(\gamma) \rightarrow \sum_{\gamma \in \Gamma} x(\gamma) \psi(\gamma)
		\end{equation}
		for each $\psi = \{ \psi(\gamma) \}_{\gamma \in \Gamma}$ from $\mathcal{E}^{\times}$ (see \eqref{EQ : WKDC1} and \eqref{EQ : WKDC2}).
		We will want to show that the mapping $j_{\mathcal{E}} \colon \mathcal{E} \rightarrow \mathfrak{X}$ is
		$\sigma(\mathcal{E},\mathcal{E}^{\times})$-to-$\sigma(\mathfrak{X},\mathfrak{X}^{\times})$ sequentially continuous.
		In other words, the proof will be complete as soon as we can show that
		\begin{equation*}
			y_n \coloneqq \sum_{\gamma \in \Gamma} x_n(\gamma) f^{(\gamma)} \otimes \boldsymbol{e}_{\gamma}
		\end{equation*}
		converges in the $\sigma(\mathfrak{X},\mathfrak{X}^{\times})$-topology to
		\begin{equation*}
			y \coloneqq \sum_{\gamma \in \Gamma} x(\gamma) f^{(\gamma)} \otimes \boldsymbol{e}_{\gamma}.
		\end{equation*}
		Here, $f^{(\gamma)} \in X_{\gamma}$ for $\gamma \in \Gamma$ with $\lVert f^{(\gamma)} \rVert_{X_{\gamma}} = 1$.
		To see this, take $\varphi = \{ \varphi_{\gamma} \}_{\gamma \in \Gamma}$ from $\mathfrak{X}^{\times}$.
		Observe that
		\begin{equation*}
			\int_{\Delta} y_n \varphi \, dm = \sum_{\gamma \in \Gamma} x_n(\gamma) \left( \int_{\Omega_{\gamma}} f^{(\gamma)} \varphi_{\gamma} \, d\mu_{\gamma} \right).
		\end{equation*}
		Moreover, since
		\begin{equation*}
			\norm{ \sum_{\gamma \in \Gamma} \left( \int_{\Omega_{\gamma}} f^{(\gamma)} \varphi_{\gamma} \, d\mu_{\gamma} \right) \boldsymbol{e}_{\gamma} }_{\mathcal{E}^{\times}}
			\leqslant \norm{ \sum_{\gamma \in \Gamma} \lVert f^{(\gamma)} \rVert_{X_{\gamma}} \lVert \varphi_{\gamma} \rVert_{X^{\times}_{\gamma}} \boldsymbol{e}_{\gamma} }_{\mathcal{E}^{\times}}
			\leqslant \norm{\varphi}_{\mathfrak{X}^{\times}},
		\end{equation*}
		so $\sum_{\gamma \in \Gamma} ( \int_{\Omega_{\gamma}} f^{(\gamma)} \varphi_{\gamma} \, d\mu_{\gamma} ) \boldsymbol{e}_{\gamma} \in \mathcal{E}^{\times}$.
		In consequence, using \eqref{EQ : seq weak conv}, we get
		\begin{equation*}
			\sum_{\gamma \in \Gamma} x_n(\gamma) \left( \int_{\Omega_{\gamma}} f^{(\gamma)} \varphi_{\gamma} \, d\mu_{\gamma} \right)
			\rightarrow
			\sum_{\gamma \in \Gamma} x(\gamma) \left( \int_{\Omega_{\gamma}} f^{(\gamma)} \varphi_{\gamma} \, d\mu_{\gamma} \right).
		\end{equation*}
		But this means that
		\begin{equation*}
			\int_{\Delta} y_n \varphi \, dm \rightarrow \int_{\Delta} y \varphi \, dm.
		\end{equation*}
		Due to $\varphi$'s arbitrariness the proof has been completed.
	\end{proof}
	
	\subsection{Mixed topologies} \label{SUBSECTION : Mixed topologies}
	
	Let $X$ be a vector space.
	Recall, following Cooper \cite[Definition~1.3, p.~5]{Coo78} that a {\bf convex bornology} (or simply {\bf bornology})
	$\mathcal{B}_X$ on $X$ is a family of balls in $X$, that is, an absolutely convex subsets of $X$ which does not contain
	a non-trivial subspace, so that:
	\begin{itemize}
		\item[(B1)] $\mathcal{B}_X$ covers $X$, that is, $X = \bigcup \{ B \colon B \in \mathcal{B}_X \}$;
		\item[(B2)] $\mathcal{B}_X$ is directed on the right by inclusions, that is, if $B, C \in \mathcal{B}_X$,
		there exists $D \in \mathcal{B}_X$ such that $B \cup C \subset D$;
		\item[(B3)] $\mathcal{B}_X$ is hereditary under inclusion, that is, if $B \in \mathcal{B}_X$ and $C$ is a ball
		contained in $B$, then $C \in \mathcal{B}_X$;
		\item[(B4)] $\mathcal{B}_X$ is stable under scalar multiplication, that is, if $B \in \mathcal{B}_X$ and $\lambda > 0$,
		then $\lambda B \in \mathcal{B}_X$.
	\end{itemize}
	Moreover, a {\bf basis} for a convex bornology $\mathcal{B}_X$ is a sub-family $\widetilde{\mathcal{B}_X}$ of $\mathcal{B}_X$
	so that for each $B \in \mathcal{B}_X$ there is $C \in \widetilde{\mathcal{B}_X}$ with $B \subset C$.
	We will say that the bornology $\mathcal{B}_X$ is of {\bf countable type} if $\mathcal{B}_X$ has a countable basis.
	
	\begin{example}[Von Neumann bornology]
		Let $(X,\mathfrak{T})$ be a locally convex topological vector space.
		It is routine to verify that the family of all $\mathfrak{T}$-bounded, absolutely convex subsets of $X$
		is a bornology on $X$.
		We will call this bornology the {\bf von Neumann bornology}.
		In particular, the von Neumann bornology of a normed space $(X, \lVert \cdot \rVert)$ is of countable type
		(to see this, just note that in this situation a basis has the form $\{ n \text{Ball}(X) \}_{n=1}^{\infty}$).
		\demo
	\end{example}
	
	\begin{construction}[Mixed topology] \label{CONSTRUCTION : mixed topology}
		Let $\mathfrak{T}$ and $\mathfrak{U}$ be two linear Hausdorff topologies defined on a vector space $X$.
		Suppose that the topology $\mathfrak{U}$ is coarser than $\mathfrak{T}$.
		According to Antoni Wiweger \cite{Wiw61}, the {\bf mixed topology} $\gamma[\mathfrak{T},\mathfrak{U}]$ on $X$
		is defined by the family of all sets of the form
		\begin{equation} \label{EQ : basis mixed topology}
			\gamma(U_1,U_2, ...; V) \coloneqq \bigcup_{n=1}^{\infty} ( U_1 \cap V + U_2 \cap 2V + ... + U_n \cap nV ),
		\end{equation}
		where $V$ and all $U_n$'s with $n \in \mathbb{N}$ are sets from bases of neighbourhoods for $0$ in topologies
		$\mathfrak{T}$ and, respectively, $\mathfrak{U}$.
		In fact, it is straightforward to see that the family of sets \eqref{EQ : basis mixed topology} satisfies
		the conditions required for a basis of neighbourhoods for $0$ ({\it cf}. \cite[pp.~49--50]{Wiw61}).
		Thus, the standard topological toolkit guarantees the existence of a unique linear topology determined by this basis.
		\demo
	\end{construction}

	With Construction~\ref{CONSTRUCTION : mixed topology} at our disposal, it is not difficult to check that:
	\begin{itemize}
		\item[(W1)] $\gamma[\mathfrak{T},\mathfrak{U}]$ is a linear Hausdorff topology on $X$;
		\item[(W2)] $\gamma[\mathfrak{T},\mathfrak{U}]$ is coarser than $\mathfrak{T}$ and finer than $\mathfrak{U}$;
		\item[(W3)] $\gamma[\mathfrak{T},\mathfrak{U}]$ coincide with $\mathfrak{U}$ on $\mathfrak{T}$-bounded subsets of $X$.
	\end{itemize}
	Actually, under the additional assumption that all basis neighbourhoods of the topology $\mathfrak{T}$ are $\mathfrak{T}$-bounded,
	the above condition (W3) provides a characterization of the mixed topology $\gamma[\mathfrak{T},\mathfrak{U}]$ (see \cite[2.2.2]{Wiw61});
	more precisely,
	\begin{itemize}
		\item[(W3*)] $\gamma[\mathfrak{T},\mathfrak{U}]$ is the finest of all linear topologies on $X$
		that coincide with $\mathfrak{U}$ on $\mathfrak{T}$-bounded subsets of $X$.
	\end{itemize}

	A few comments seem in order.
	
	\begin{remark}[On Cooper's construction]
		Slightly less general construction of mixed topologies was proposed by Cooper (see \cite[Chapter~I]{Coo78}).
		Roughly speaking, the difference is that the topology $\mathfrak{T}$ in Construction~\ref{CONSTRUCTION : mixed topology}
		is replaced by a bornology $\mathcal{B}_X$ on $X$.
		It is worth mentioning that this construction naturally generalize the class of (DF)-spaces introduced by Grothendieck
		(see \cite[Remark~1.27, p.~19]{Coo78} and references therein).
		Probably the most natural mixed topology on $X$ is $\gamma(\lVert \cdot \rVert, \mathfrak{T})$, where $\lVert \cdot \rVert$
		is the von Neumann bornology of a normed space $(X, \lVert \cdot \rVert)$ (see \cite[Section~1.4]{Coo78}).
		\demo
	\end{remark}

	From our perspective, however, the most important observation about the mixed topologies can be summarized in the following
	
	\begin{lemma}[$\oplus$-compatible mixed topologies]
		{\it Let $\mathfrak{T}$ and $\mathfrak{U}$ be two $\oplus$-compatible topologies on $( \bigoplus_{\gamma \in \Gamma} X_{\gamma} )_{\mathcal{E}}$.
		Suppose that the topology $\mathfrak{U}$ is coarser than $\mathfrak{T}$.
		Then the mixed topology $\gamma[\mathfrak{T},\mathfrak{U}]$ is also $\oplus$-compatible.}
	\end{lemma}
	\begin{proof}
		With (W1) and (W2) in mind, just look at Definition~\ref{DEF: admissible topology}.
	\end{proof}

	For some examples of mixed topologies we refer to \cite[pp.~20--25]{Coo78} and \cite[pp.~64--67]{Wiw61}.
	
	\section{{\bf Classical Kadets-Klee properties}} \label{SECTION : EXAMPLES}
	
	In this section we will focus on the two undoubtedly most natural Kadets--Klee properties, that is to say, the properties
	${\bf H}(\textit{weak})$ and ${\bf H}(\textit{point-wise})$ (see Section~\ref{SUBSECTION : H(weak)} and Section~\ref{SECTION : local convergence in measure}, respectively).
	We will also take this opportunity to show how earlier, usually, only partial, results about the Kadets--Klee properties
	in direct sums or, in particular, K{\" o}the--Bochner sequence spaces, follows from our abstract considerations.
	
	\subsection{Weak topology} \label{SUBSECTION : H(weak)}
	
	Remembering about the historical origins of the Kadets--Klee property, let us first give due attention to weak topologies.

	\begin{theorem}[The property ${\bf H}(\textit{weak})$] \label{THM : H weak}
		{\it Let $\{ X_{\gamma} \}_{\gamma \in \Gamma}$ be a sequence of Banach spaces.
		Further, let $\mathcal{E}$ be a Banach sequence space with the property ${\bf H}(\textit{point-wise})$.
		Suppose that the space $\mathcal{E}$ is monotone complete (see Remark \ref{REMARK : o warunkach w twierdzeniu}).
		Then the space $( \bigoplus_{\gamma \in \Gamma} X_{\gamma} )_{\mathcal{E}}$ has the property ${\bf H}(\textit{weak})$
		if, and only if,}
		\begin{itemize}
			\item[(1)] {\it all the spaces $X_{\gamma}$ have the property ${\bf H}(\textit{weak})$;}
			\item[(2)] {\it the set $\Gamma$ can be decomposed into two disjoint subsets, say $\Gamma_1$ and $\Gamma_2$,
			in such a way that all $X_{\gamma}$'s with $\gamma \in \Gamma_1$ have the Schur property and $\mathcal{E}$ is ${\bf SM}(\gamma)$ for $\gamma \in \Gamma_2$.}
		\end{itemize}
	\end{theorem}
	\begin{proof}
		All that needs to be done is to check the assumptions (A1), (A2), (A3) and (A4) of Theorem~\ref{THM : oba warunki sa rownowazne}
		and call Theorems~\ref{THM: H is hereditary} and \ref{THM: main theorem} on stage.
		However, since this requires a bit of patience, let us take a moment to explain this in more detail.
  
        $\bigstar$
		Plainly, the weak topology is a linear Hausdorff topology coarser than the norm topology.
  
        $\bigstar$
		Next, we need to show that the weak topology on the direct sum $( \bigoplus_{\gamma \in \Gamma} X_{\gamma} )_{\mathcal{E}}$ is $\oplus$-compatible with the weak topologies on their components. It is crystal clear that all three mappings from the Definition~\ref{DEF: admissible topology} are norm-to-norm continuous. Note also that an operator acting between Banach spaces is norm-to-norm continuous if, and only if,
		it is weak-to-weak continuous (see, for example, \cite[Theorem~2.5.11, p.~214]{Meg98}; {\it cf}. \cite[p.~344]{AK06}). The conclusion is clear.

        Thus the necessity follows from Theorem~\ref{THM: H is hereditary}.
        
        Now, we will prove the sufficiency.
        First of all we will check the assumptions (A1)-(A4) from of Theorem~\ref{THM : oba warunki sa rownowazne}.

        {\bf (A1)}
		Since  the space $\mathcal{E}$ has the property ${\bf H}(\textit{point-wise})$, by Lemma \ref{LEMMA: H tau => lower semicontinuous},
		$\mathcal{E}$ has the weak Fatou property.
		To conclude that the unit ball $\text{Ball}(\mathcal{E})$ is sequentially closed in the point-wise topology it is enough
		to take a look at (A1) in Remark~\ref{REMARK : o warunkach w twierdzeniu}.

       	{\bf (A2)} The assumption (A2) has been already checked above.
		
		{\bf (A3)}
		Since we assumed that the space $\mathcal{E}$ has the property ${\bf H}(\textit{point-wise})$,
		so it is clear that for sequences on the unit sphere in $\mathcal{E}$ the weak topology coincide with the point-wise topology.
		
		{\bf (A4)}  
		Finally, it is well-known that the norm function $x \mapsto \norm{x}$ is always sequentially lower-semicontinuous with respect to 	the weak topology.
		(Let us note, by the way, that since the property ${\bf H}(\textit{point-wise})$ imply ${\bf H}(\textit{weak})$,
		so this can also be deduced directly from Lemma~\ref{LEMMA: H tau => lower semicontinuous} (see also \cite[Lemma~3.1, p.~27]{Coo78}).)
       
        $\bigstar$
        Finally, the assumption (A1) from Theorem~\ref{THM: main theorem} follows from Example~\ref{EXAMPLE : OC(tau) => OC}.
        
        In consequence, the sufficiency follows from Theorems~\ref{THM: main theorem} and \ref{THM : oba warunki sa rownowazne}.
	\end{proof}

    \begin{remark}[About assumptions in Theorem~\ref{THM : H weak}]
    	The above theorem remains true if we replace the assumption that $\mathcal{E}$ has the property ${\bf H}(\textit{point-wise})$
    	by the formally more general condition (A3) from Theorem~\ref{THM : oba warunki sa rownowazne} which, let us recall, reads as:
    	\begin{quote}
    		\textit{For sequences on the unit sphere in $\mathcal{E}$ the weak topology coincide with the point-wise topology.}
    	\end{quote}
    	One, so to speak, disadvantage of this condition is that it is rather difficult to check in specific situations without appealing
    	to some other, perhaps more natural, properties of Banach spaces.
    	We already had an illustration of this situation above, where we used the property ${\bf H}(\textit{point-wise})$.
    	(A side note is that this choice seems to be quite optimal, especially when compared with the results of other authors;
    	{\it cf}. Corollary~\ref{COR: Krasowska i Pluciennik}.)
          
        Either way, there are other options.
        One of them is reflexivity.
        Note that reflexivity and the property ${\bf H}(\textit{point-wise})$ are not comparable in general
        (to see this, it is enough to consider Orlicz sequence spaces; see Table~\ref{TABLE : H} for details).
        Thus, if we replace the assumption that \enquote{the space $\mathcal{E}$ has the property ${\bf H}(\textit{point-wise})$}
        by \enquote{the space $\mathcal{E}$ is reflexive}, we obtain a little different variant of Theorem~\ref{THM : H weak}.
        Furthermore, in that case, the property ${\bf H}(\textit{point-wise})$ will be even necessary.
        \demo
    \end{remark}
    
    \begin{remark}[About Theorem~\ref{THM : H weak}]
    	What seems somehow unusual is that in the \enquote{proof} of Theorem~\ref{THM : H weak} we did not need to explicitly use the description
    	of the dual space to $( \bigoplus_{\gamma \in \Gamma} X_{\gamma} )_{\mathcal{E}}$!
    	Such a description is known, but only under some additional assumptions on $\mathcal{E}$
    	(see \cite[Proposition~4.8]{Lau01} and Proposition~\ref{PROPOSITION : duality of direct sums}; {\it cf}. \cite{DK16}).
    	It should also be mentioned that under the assumption that the space $\mathcal{E}$ is finite-dimensional the above
    	result was proved by Dowling, Photi and Saejung in \cite[Theorem 3.1]{DPS07}.
    	(Of course, this situation is particularly easy and reduces the set of additional assumptions imposed on the space $\mathcal{E}$
    	in Theorem~\ref{THM : H weak} to $\emptyset$.)
    	\demo
    \end{remark}
            
	Several immediate consequences of Theorem~\ref{THM : H weak} seem worth noting.
	First, following van Dulst and de Valk \cite[Proposition~2]{DV86}, let us say a few words about the Kadets--Klee property
	in the $\ell_F$-direct sums of Banach spaces.

	\begin{corollary}(D. van Dulst and V. de Valk, 1986)
		{\it Let $\{ X_{n} \}_{n=1}^{\infty}$ be a family of Banach spaces with the property ${\bf H}(\text{weak})$.
		Further, let $F$ be an Orlicz function satisfying the $\delta_2$-condition.
		Then the space $( \bigoplus_{n=1}^{\infty} X_{n} )_{\ell_F}$ has the property ${\bf H}(\textit{weak})$.}
	\end{corollary}

	Actually, we can deduce much stronger result.

    \begin{corollary}
        {\it Let $\{ X_{n} \}_{n=1}^{\infty}$ be a family of Banach spaces and let $F$ be the Young function.}
    	\begin{itemize}
    		\item[(1)] {\it Suppose that $\ell_F \neq \ell_1$.
    		Then the space $( \bigoplus_{n=1}^{\infty} X_{n} )_{\ell_F}$ has the property ${\bf H}(\text{weak})$ if, and only if,  
    		all the spaces $X_{n}$ have the property ${\bf H}(\text{weak})$,
    		the Orlicz function $F$ satisfies the $\delta_2$-condition for small arguments and $F(b_{F}) \geqslant 1$.}
    		\item[(2)] {\it Suppose that $\ell_F = \ell_1$.
    		Then the space $( \bigoplus_{n=1}^{\infty} X_{n} )_{\ell_F}$ has the property ${\bf H}(\text{weak})$ if, and only if, all the spaces $X_{n}$ have the property ${\bf H}(\text{weak})$.}
    	\end{itemize}
    \end{corollary}
    \begin{proof}
    	This is an easy combination of Theorems~\ref{THM : H weak} and \ref{THM: H is hereditary} along with some of the following observations.
    	
    	$\bigstar$ The property ${\bf H}(\textit{weak})$ in Orlicz sequence spaces has been completely characterized in Theorem~\ref{THEOREM : H in Orlicz sequence}.
    	Thus the necessity follows directly from Theorem~\ref{THM: H is hereditary}.
    	
    	$\bigstar$ The Orlicz sequence spaces are monotonically complete (see, for example, \cite[Theorem~2.6.9, p.~120]{MN91}), in fact they even have the Fatou property.
         
        $\bigstar$ Suppose that $\ell_F \neq \ell_1$. The assumption imposed on the Young function $F$ in the form of the $\delta_2$-condition
        gives that the space $\ell_F$ is {\bf SM} and has the property ${\bf H}(\textit{point-wise})$
        (see \cite[Theorem 1]{CHM95} and Table~\ref{TABLE : H}), respectively).
        In consequence, the sufficiency follows from Theorem~\ref{THM : H weak}.
        
        $\bigstar$ Suppose that $\ell_F = \ell_1$. Obviously, $\ell_1$ has both the property ${\bf H}(\textit{point-wise})$ and the Fatou property
        and is strictly monotone. The rest is just to apply Theorem~\ref{THM : H weak}.
    \end{proof}

	Next, following Krasowska and P{\l}uciennik \cite[Theorem~1]{KP97},
	let us examine the Kadets--Klee property in the K{\" o}the--Bochner sequence spaces.
	
	\begin{corollary}(Due to D. Krasowska and R. P{\l}uciennik, 1997) \label{COR: Krasowska i Pluciennik}
		{\it Let $X$ be a separable Banach space and let $\mathcal{E}$ be a Banach sequence space.
		Suppose that $E$ has the property ${\bf H}(\text{point-wise})$ and is ${\bf SM}$.
		Then the K{\" o}the--Bochner space $\mathcal{E}(X)$ has the property ${\bf H}(\text{weak})$ if, and only if,
		$X$ has the property ${\bf H}(\text{weak})$.}
	\end{corollary}
	
    Again, using Theorem~\ref{THM : H weak}, we can get a slightly more precise version of the above result.
        
	\begin{corollary}\label{Cor:H-weak in E(X)}
       	{\it Let $X$ be a Banach space.
       	Further, let $\mathcal{E}$ be a monotone complete Banach sequence space with the property ${\bf H}(\text{point-wise})$.
		Then the K{\" o}the--Bochner space $\mathcal{E}(X)$ has the property ${\bf H}(\text{weak})$ if, and only if}
		\begin{itemize}
			\item[(1)] {\it the space $X$ has the property ${\bf H}(\text{weak})$;}
			\item[(2)] {\it the space $X$ has the Schur property or the space $\mathcal{E}$ is ${\bf SM}$.}
		\end{itemize}
	\end{corollary}

	In particular, since all $\ell_p$'s with $1 \leqslant p < \infty$ are strictly monotone, monotone complete and have property ${\bf H}(\text{point-wise})$,
	so the above Corollary~\ref{Cor:H-weak in E(X)} reduce further to the following result noted by Leonard \cite[Theorem~3.1]{Leo76}.

	\begin{corollary}[I. E. Leonard, 1976]
		{\it Let $X$ be a Banach space.
		Further, let $1 \leqslant p < \infty$.
		Then the space $\ell_p(X)$ has the property ${\bf H}(\textit{weak})$ if, and only if, $X$ has the property ${\bf H}(\textit{weak})$.}
	\end{corollary}
	
	\subsection{Topology of local convergence in measure} \label{SECTION : local convergence in measure}
	
	Probably the second most popular topology considered in the context of the Kadets--Klee type properties is the topology of local convergence in measure
	(see, for example, \cite{CDSS96}, \cite{DHLMS03}, \cite{FH99}, \cite{FHS10}, \cite{HKL06}, \cite{Kol12} and \cite{Suk95}).
	To say more about this, however, we need a little preparation.
	
	Recall, that a sequence $\{ f_n \}_{n=1}^{\infty}$ of scalar-valued measurable functions defined on a measure space $(\Omega,\Sigma,\mu)$
	is said to {\bf converge locally in measure} to a measurable function $f$ provided, for every $\varepsilon > 0$ and $F \in \Sigma$ with $\mu(F) < \infty$,
	 $$\mu(\{ \omega \in F \colon \abs{f(\omega) - f_n(\omega)} \geqslant \varepsilon \}) \rightarrow 0$$ as $n \rightarrow \infty$.

	On $L_0(\Omega,\Sigma,\mu)$, that is, a vector space of scalar-valued measurable functions defined on $\Sigma$ modulo equality almost everywhere,
	there is a topology called the {\bf topology of local convergence in measure} (see, for example, \cite[245A, p.~173]{Fre01} for details).
	Here we will denote it by $\mathfrak{T}(\mu)$ (or simply by $\mathfrak{T}(\textit{measure})$ if no confusion is possible).
	Note that this topology is linear (see \cite[245D, p.~174]{Fre01}).
	However, it is Hausdorff if, and only if, the measure space $(\Omega,\Sigma,\mu)$ is semi-finite, that is, whenever $A \in \Sigma$ and $\mu(A) = \infty$
	there is $F \subset A$ with $0 < \mu(F) < \infty$ (see \cite[245E, p.~176]{Fre01}).
	Moreover, in general, $\mathfrak{T}(\mu)$ fails to be locally convex (see \cite[Example~2.2.5, p.~162]{Meg98}).
	For more information about the space $L_0$ and the topology of (local) convergence in measure we refer to \cite[241, 245, 364 and 463]{Fre01}.
	
	Let us establish a suitable set-up for the rest of this section.
		
    \begin{notation} \label{NOTATION : direct sum as BFS}
    	Let, unless we say otherwise, $\{ X_{\gamma} \}_{\gamma \in \Gamma}$ be a family of Banach function spaces defined on
    	$(\Omega_{\gamma},\Sigma_{\gamma},\mu_{\gamma})$ for $\gamma \in \Gamma$
    	(see \cite[Definition~1.b.17, p.~28]{LT79} and \cite[Chapter~15]{Mal89} for more details about Banach function spaces;
    	{\it cf}. \cite{KT24} and \cite{Now07}).
    	Here and hereinafter, $(\Omega_{\gamma},\Sigma_{\gamma},\mu_{\gamma})$ is a complete $\sigma$-finite measure space.
    	Moreover, let $\mathcal{E} $ be a Banach sequence space on $\Gamma$.
		Recall that the space $( \bigoplus_{\gamma \in \Gamma} X_{\gamma} )_{\mathcal{E}}$ can be seen as a linear subspace of $L_0(\Delta,\mathcal{A},m)$,
		where $\Delta$ is a disjoint union $\bigsqcup_{\gamma \in \Gamma} \Omega_{\gamma}$ of $\Omega_{\gamma}$'s and the measure $m$ is defined
		as $m(A) \coloneqq \sum_{\gamma \in \Gamma} \mu_{\gamma}(A \cap \Omega_{\gamma})$, where $A \subset \bigsqcup_{\gamma \in \Gamma} \Omega_{\gamma}$.
		(In fact, as we have already indicated earlier, $( \bigoplus_{\gamma \in \Gamma} X_{\gamma} )_{\mathcal{E}}$ is just a Banach function space itself.)
		For this reason, whenever we will refer to the topology of local convergence in measure on $( \bigoplus_{\gamma \in \Gamma} X_{\gamma} )_{\mathcal{E}}$,
		we will always mean the topology of local convergence in measure inherited from $L_0(\Delta,\mathcal{A},m)$.
		\demo
	\end{notation}

	Let us also note the following technical
	
	\begin{lemma}[Topology of local convergence in measure is $\oplus$-compatible] \label{LEMMA : local measure is compatible}
		{\it The topology of local convergence in measure on $( \bigoplus_{\gamma \in \Gamma} X_{\gamma} )_{\mathcal{E}}$
		is $\oplus$-compatible with the topologies of local convergence in measure of their components.}
	\end{lemma}
	\begin{proof}
		It should be clear that conditions (C1), (C2) and (C3) from Definition~\ref{DEF: admissible topology} are met, note only that (C1) follows from \cite[Lemma~2, p.~97]{KA82}.
		Thus we need only to show (C4), that is, that the embedding $j_{\mathcal{E}} \colon \mathcal{E} \rightarrow ( \bigoplus_{\gamma \in \Gamma} X_{\gamma} )_{\mathcal{E}}$
		is $\mathfrak{T}(\#)$-to-$\mathfrak{T}(m)$ sequentially continuous.
		To see this, suppose that $x_n$ converges point-wisely to $x$ in $\mathcal{E}$.
		Our goal is to show that $y_n$ converges to $y$ in $\mathfrak{T}(m)$, where
		\begin{equation*}
			y_n = \sum_{\gamma \in \Gamma} x_n(\gamma) f^{(\gamma)} \otimes \boldsymbol{e}_{\gamma} \quad \text{ for } \quad n \in \mathbb{N},
		\end{equation*}
		and
		\begin{equation*}
			y = \sum_{\gamma \in \Gamma} x(\gamma) f^{(\gamma)} \otimes \boldsymbol{e}_{\gamma}
		\end{equation*}
		and $f^{(\gamma)} \in X_{\gamma}$ for $\gamma \in \Gamma$ with $\lVert f^{(\gamma)} \rVert_{\gamma} = 1$.
		Fix $\varepsilon > 0$ and take $F \in \mathcal{A}$ with $m(F) < \infty$.
		Note that
		\begin{equation} \label{EQ : raz}
			\text{there is } \Gamma_{\text{FIN}} \subset \Gamma \text{ with } \#(\Gamma_{\text{FIN}}) < \infty \text{ such that }
				m(F) \leqslant \sum_{\gamma \in \Gamma_{\text{FIN}}} \mu_{\gamma}(F \cap \Omega_{\gamma}) + \frac{\varepsilon}{2};
		\end{equation}
		\begin{equation} \label{EQ : dwa}
			\text{there is } \delta = \delta(\varepsilon, F) > 0 \text{ with }
				\mu_{\gamma}\left( \omega \in \Omega_{\gamma} \colon \abs{f^{(\gamma)}(\omega)} > \frac{\varepsilon}{\delta} \right) < \frac{\varepsilon}{2 \#(\Gamma_{\text{FIN}})}
				\text{ for all } \gamma \in \Gamma_{\text{FIN}};
		\end{equation}
		\begin{equation} \label{EQ : trzy}
			\text{there is } N = N(\varepsilon, F) \in \mathbb{N} \text{ with } \abs{x(\gamma) - x_n(\gamma)} < \delta \text{ for all }
				\gamma \in \Gamma_{\text{FIN}} \text{ and } n \geqslant N.
		\end{equation}
		For $n \geqslant N$, using \eqref{EQ : raz}, we have
		\begin{align*}
			m \left( \left\{ \omega \in F \colon \abs{y(\omega) - y_n(\omega)} \geqslant \varepsilon \right\} \right)
				& = \sum_{\gamma \in \Gamma} \mu_{\gamma} \left( \left\{ \omega \in F \cap \Omega_{\gamma} \colon \abs{y(\omega) - y_n(\omega)} \geqslant \varepsilon \right\} \right) \\
				& \leqslant \sum_{\gamma \in \Gamma_{\text{FIN}}}
					\mu_{\gamma} \left( \left\{ \omega \in F \cap \Omega_{\gamma} \colon \abs{y(\omega) - y_n(\omega)} \geqslant \varepsilon \right\} \right)
					+ \frac{\varepsilon}{2} \\
				& \leqslant \sum_{\gamma \in \Gamma_{\text{FIN}}} \mu_{\gamma}\left( \omega \in \Omega_{\gamma} \colon \abs{f^{(\gamma)}(\omega)} > \frac{\varepsilon}{\delta} \right)
					+ \frac{\varepsilon}{2} \quad (\text{by \eqref{EQ : trzy}}) \\
				& \leqslant \frac{\varepsilon}{2} + \frac{\varepsilon}{2} = \varepsilon \quad (\text{using \eqref{EQ : dwa}}).
		\end{align*}
		The proof follows.
	\end{proof}
	
	We are finally ready to show the main result of this section.

	\begin{theorem}[The property ${\bf H}(\textit{measure})$] \label{THEOREM : H(measure)}
		{\it Let $\{ X_{\gamma} \}_{\gamma \in \Gamma}$ be a family of Banach function spaces.
		Further, let $\mathcal{E} $ be a monotone complete Banach sequence space defined on $\Gamma$ .
        Then the space $( \bigoplus_{\gamma \in \Gamma} X_{\gamma} )_{\mathcal{E}}$ has the property ${\bf H}(\textit{measure})$ if,
        and only if,}
    	\begin{itemize}
    		\item[(1)] {\it all $X_{\gamma}$'s have the property ${\bf H}(\textit{measure})$;}
    		\item[(2)] {\it the space $\mathcal{E}$ has the property ${\bf H}(\textit{point-wise})$;}
    		\item[(3)] {\it the space $\mathcal{E}$ is ${\bf SM}$.}
    	\end{itemize}
	\end{theorem}
	\begin{proof}
		{\bf The necessity.}
		Suppose that the space $( \bigoplus_{\gamma \in \Gamma} X_{\gamma} )_{\mathcal{E}}$ has the property ${\bf H}(\textit{measure})$.
		This is a straightforward consequence of Theorem~\ref{THM: H is hereditary}.
		The only thing that is probably worth mentioning is that the space $X_{\gamma}$ fails to have the Schur property with respect
		to the topology $\mathfrak{T}(\mu_{\gamma})$ (and, therefore, we must compensate this deficiency \enquote{entirely} by the
		strict monotonicity of the space $\mathcal{E}$).
        
        {\bf The sufficiency.}
        Suppose that all $X_{\gamma}$'s have the property ${\bf H}(\textit{measure})$ and the space $\mathcal{E}$ has the property
        ${\bf H}(\textit{point-wise})$ and is ${\bf SM}$.
        Of course, we need to apply Theorems~\ref{THM: main theorem} and \ref{THM : oba warunki sa rownowazne}.
        Essentially, the argument comes down to a careful check of the relevant assumptions.
        To free the reader from this tedious task, let us do this together now. 
        
        $\bigstar$ The topology of local convergence in measure is a linear Hausdorff topology that is coarser than the norm topology
        (except for the last part that was already explained above; see \cite[Theorem~1, p.~96]{KA82}).

        Next, we will check the assumptions (A1)-(A4) from Theorem~\ref{THM : oba warunki sa rownowazne}.
        
        {\bf (A1)} Recall that if the space $\mathcal{E}$ has the property ${\bf H}(\textit{point-wise})$ then it is order continuous (see Lemma \ref{LEMMA : H(tau) => OC}).
        Since the space $\mathcal{E}$ is assumed to be monotone complete, so the unit ball $\text{Ball}(\mathcal{E})$
        is sequentially closed with respect to the point-wise topology ({\it cf}. Remark~\ref{REMARK : o warunkach w twierdzeniu}).
        
        {\bf (A2)} The topology $\mathfrak{T}(m)$ on $( \bigoplus_{\gamma \in \Gamma} X_{\gamma} )_{\mathcal{E}}$ is $\oplus$-compatible
        with the topologies $\mathfrak{T}(\mu_{\gamma})$ on $X_{\gamma}$ and $\mathfrak{T}(\#)$ on $\mathcal{E}$.
        This is a consequence of Lemma~\ref{LEMMA : local measure is compatible}.
        
        {\bf (A3)} Obviously, the topology $\mathfrak{T}(\textit{measure})$ on $\mathcal{E}$ agree with the topology of point-wise convergence
        ({\it cf}. \cite[245X]{Fre01}).
        
        {\bf (A4)} For all $\gamma \in \Gamma$, the norm function $x \mapsto \norm{x}_{\gamma}$ is sequentially lower semi-continuous
        with respect to the topology $\mathfrak{T}(\mu_{\gamma})$.
        (By the way, one can directly deduce this simple fact, for example, from Lemma~\ref{LEMMA: H tau => lower semicontinuous}.)
	\end{proof}

	\section{{\bf Open ends}} \label{SECTION : Open ends}
	
	Let us list and shortly discuss some problems which arise from this paper.
	
	\subsection{Direct integrals}
	
	In a sense, the space $( \bigoplus_{\gamma \in \Gamma} X_{\gamma} )_{\mathcal{E}}$ can be seen as a discrete version
	of the so-called {\bf direct integrals} $( \int_{\Omega}^{\oplus} X_{\omega} \mu(d\omega) )_{E}$,
	where $E$ stands for a Banach function space defined on a decomposable measure space $(\Omega,\Sigma,\mu)$
	(see \cite[Chapter~21, Definition~211E]{Fre01}).
	We refer to Hydon, Levy and Raynaud's paper \cite[Chapter~6]{HLR91} for details of this construction.
	As practice shows, finding criteria that guarantee some geometric property is significantly simpler in the case of
	sequence spaces (see, for example,  \cite{LL85} and \cite{ST80}).
	For this reason, the following question seems intriguing
	
	\begin{question}
		{\it When does the space $( \int_{\Omega}^{\oplus} X_{\omega} \mu(d\omega) )_{E}$ have the property ${\bf H}(\mathfrak{T})$?}
	\end{question}
	
	Given our results, a possible solution to the above problem requires some new ideas.
	At the moment, unfortunately, we do not have many useful thoughts on this topic.
	
	\subsection{Uniform Kadets--Klee properties}
	
	Let $X$ be a Banach space and let $\mathfrak{T}$ be a linear Hausdorff topology on $X$ coarser than the norm topology.
	Recall that the space $X$ is said to have the {\bf uniform Kadets--Klee property} with respect to $\mathfrak{T}$
	(briefly, the property ${\bf UH}(\mathfrak{T})$) if for each $\varepsilon > 0$ there exists $\delta = \delta(\varepsilon) > 0$
	such that every $\varepsilon$-separated $\mathfrak{T}$-convergent sequence $\{ x_n \}_{n=1}^{\infty}$ in the unit sphere of $X$
	converges to an element of norm less than $1 - \delta$.
	
	\begin{question}
		{\it When does the space $( \bigoplus_{\gamma \in \Gamma} X_{\gamma} )_{\mathcal{E}}$ have the property ${\bf UH}(\mathfrak{T})$?}
	\end{question}

	We believe that to answer the above question it is enough to replace the assumption about strict monotonicity ${\bf SM}$
	in Theorems~\ref{THM: H is hereditary} and \ref{THM: main theorem} by its \enquote{uniform analogue}, that is,
	uniform monotonicity ${\bf UM}$. However, we have not checked any details.
	On the other hand, a particular case of this situation has been already investigated, which could be some help.
	Namely, the criteria for the property ${\bf UH}(\textit{weak})$ in K{\" o}the--Bochner sequence spaces have been proved in \cite{Kol03}.
	
	\subsection{Minimal Hausdorff topologies on function spaces}
	
	Let $X$ be a Banach sequence space.
	At the heart of the proof of Remark~\ref{REMARK : minimal topologies} is the observation that the minimal
	Hausdorff topology on $X$ coincide with the topology of point-wise convergence when restricted to the unit ball $\text{Ball}(X)$.
	Since the topology of point-wise convergence on $X$ coincide with the topology of local convergence in measure,
	so it is natural to ask the following
	
	\begin{question}
		{\it Is the topology of local convergence in measure the coarsest locally solid Hausdorff topology
		in the class of Banach sequence spaces?}
	\end{question}

	Note that in 1987 Labuda showed that the topology of local convergence in measure is the coarsest locally solid topology
	in the class of Orlicz spaces (see \cite[Theorem~14]{Lab87}).
	
	\subsection{Glimpse of Kalton's Zone}
	
	It is relatively easy to notice significant gaps in knowledge regarding Kadets--Klee properties in the world of {\it quasi-Banach spaces}.
	There are actually good reasons for this.
	For example, as Day's classical result shows, the topological dual of the space $L_p$ with $0 < p < 1$ is trivial,
	so the direct analogue of the property ${\bf H}$ is in general invalid.
	However, there are topologies on quasi-Banach spaces weaker than the one generated by the quasi-norm, which are a kind of substitute for
	the weak topology and, for them, the question about the property ${\bf H}(\mathfrak{T})$ makes perfect sense (see \cite{Kal03} and their references).
	Let us just mention about the {\it Mackey topology}, which is the finest locally convex topology coarsest that the quasi-norm topology.
	
	On the other hand, there is actually no problem in considering the Kadets--Klee property with respect to the topology of local convergence in measure.
	In fact, modulo some technicalities, our results from Section~\ref{SECTION : local convergence in measure} remain true in this setting.
	We leave the details to the interested reader.
	
	\subsection{Some other geometric properties}
	
	Half in jest, half seriously, we would like to give some advice in the style of Serge Lang \cite{Lan65}:
	{\it \enquote{Take any book on K{\" o}the--Bochner spaces and prove all the theorems in the more general context of the direct sums}}.
	Sometimes this will be straightforward, but sometimes definitely will not.

	\appendix
	
	\section{A hitchhiker's guide to classical Kadets--Klee properties} \label{APPENDIX : A}
	
	In this first supplementary section we will present criteria that guarantee the Kadets--Klee properties with respect to both weak
	and point-wise topology in some classical sequence spaces (see Table~\ref{TABLE : H}).
	This is by no means new, but requires a compilation of results scattered throughout the literature (with some minor additions).
	Anyway, we do not know of any place that offers such a handy summary.
	
    Let us recall some facts needed if we apply our general results from Section~\ref{SECTION : EXAMPLES} for concrete sequence spaces.
    For definitions of all spaces discussed here we refer to Section~\ref{direct sums} and references therein.
    
    \begin{center}
    	\begin{table}[H]
    		\begin{tabular}{ | c | c | c | c | } 
    			\hline \xrowht{15pt}
    			{\bf No.}	& {\bf Banach sequence space} & ${\bf H}(\textit{weak})$ & ${\bf H}(\textit{point-wise})$ \\
    			\hline \xrowht{15pt}
    			{\bf 1.} & $\ell_F$ & see Theorem~\ref{THEOREM : H in Orlicz sequence} & $F \in \delta_2$ and $F(b_{F}) \geqslant 1$ \\
    			\hline \xrowht{15pt}
    			{\bf 2.} & $\ell_p$ & $1 \leqslant p < \infty$ & $1 \leqslant p < \infty$ \\
    			\hline \xrowht{15pt}
    			{\bf 3.} & $d(w,p)$ & $\sum_{n=1}^{\infty} w(n) = \infty$ & $\sum_{n=1}^{\infty} w(n) = \infty$ \\
    			\hline \xrowht{15pt}
    			{\bf 4.} & $\ell_{\{ p_{n} \}}$ & $\sup_{n \in \mathbb{N}} p_n < \infty$  & $\sup_{n \in \mathbb{N}} p_n < \infty$ \\
    			\hline \xrowht{15pt}
    			{\bf 5.} & $ces_p$ & $1 < p < \infty$  & $1 < p < \infty$ \\
    			\hline
    		\end{tabular}
    		\captionof{table}{A quick summary of Kadets--Klee properties for some (neo)classical Banach sequence spaces} \label{TABLE : H}
    	\end{table}
    \end{center}

	Below we will explain in details why the above table looks the way it does.
	(Of course, this list is not complete in any sense, and the choice of sequence spaces was dictated mainly by the subjective
	preferences of the authors.
	Anyway, this is a good start.
	More criteria for certain combinations of the above-mentioned spaces, such as Orlicz--Lorentz sequence
	spaces or Ces{\' a}ro--Orlicz sequence spaces, and generalizations of Orlicz and Nakano sequence spaces in the form
	of Musielak--Orlicz sequence spaces can be deduced from \cite[Theorem~14]{HKL06} and \cite[Theorem~1]{FHS10}, respectively.)
	\newline
      
    {\bf No.~1: Orlicz sequence spaces $\ell_F$.} This is an immediate consequence of the following
    
    \begin{theorem}[The property ${\bf H}(\textit{weak})$ in Orlicz sequence spaces] \label{THEOREM : H in Orlicz sequence}
    	{\it The Orlicz sequence spaces $\ell_F$ has the property ${\bf H}(\textit{weak})$ if, and only if}
    \begin{enumerate}
      	\item[(1)] {\it the Young function $F$ satisfies the $\delta_2$-condition,}
      	\item[(2)] {\it and either $F$ is linear in some neighborhood of zero or $F(b_{F}) \geqslant 1$.}
    \end{enumerate}
    \end{theorem}
    \begin{proof}
    Since this result should be considered as a supplement to Hudzik and Pallaschke paper \cite{HP97},
    we refer there for all unexplained concepts that will appear below (see also \cite{Ch96} and \cite{Mal89}).
      
   	{\bf The necessity.}
   	Suppose that $\ell_F$ has the property ${\bf H}(\textit{weak})$.
   	Then, due to Remark~\ref{LEMMA : H(tau) => OC}, the space $\ell_F$ is separable.
   	However, it is well-known that the Orlicz space $\ell_F$ is separable if, and only if, the Young function $F$ satisfies the $\delta_2$-condition.
    Thus, (1) follows.
    Going ahead, to show (2), suppose that the Young function $F$ is linear in no neighborhood of zero and $F(b_{F}) < 1$.
    Take $\eta > 0$ such that $F(b_{F}) + F(\eta) \leqslant 1$.
    Define
    \begin{equation*}
        x \coloneqq b_F \boldsymbol{e}_1 \quad \text{ and } \quad x_n \coloneqq x + \eta \boldsymbol{e}_n \quad \text{ for } \quad n \in \mathbb{N}.
    \end{equation*}
    It is clear that $\norm{x}_{\ell_F} = \norm{x_n}_{\ell_F} = 1$ and $\norm{x - x_n} = \eta \norm{\boldsymbol{e}_1}> 0$.
    We claim that $x_n \rightarrow x$ weakly in $\ell_F$.
    Let $x^*$ be any bounded linear form on $\ell_F$.
    Plainly, due to Yosida--Hewitt's type decomposition, $x^* = L + S$, where $L = \{ L(n) \}_{n=1}^{\infty}$ is a function from the K{\" o}the dual
    $(\ell_F)^{\times}$ of $\ell_F$, while $S$ is a singular functional from $(\ell_F)^s$, that is, the space of all singular functionals on $\ell_F$.
    Since $(\ell_F)^s$ is nothing else but the annihilator of an ideal $(\ell_F)_o$, that is, $\langle x, S \rangle = 0$ for all $x \in (\ell_F)_o$, so
    \begin{equation} \label{EQ : slabo do zera na L}
    	\langle x - x_n, x^* \rangle
    		= \langle \eta \boldsymbol{e}_n, L + S \rangle
    		= \eta \langle \boldsymbol{e}_n, L \rangle
    		= \eta L(n).
    \end{equation}
    Now, due to our assumption that the Young function $F$ is linear in no neighborhood of zero, it follows that $\ell_F \not\hookrightarrow \ell_1$.
    But then $\ell_{\infty} \not\hookrightarrow (\ell_F)^{\times}$.
    Therefore, $(\ell_F)^{\times} \hookrightarrow c_0$.
    In consequence, $L(n) \rightarrow 0$ and, in view of \eqref{EQ : slabo do zera na L}, our claim follows.
    All this shows that $\ell_F$ has not the property ${\bf H}(\textit{weak}).$
         
     {\bf The sufficiency.}
     Suppose that (1) and (2) holds.
     Since $F \in \delta_2$, so  $a_F = 0$.
     Now, if the Young function $F$ is linear in some neighborhood of zero, then the space $\ell_F$ coincide, up to an equivalent norm, with $\ell_1$.
     However, $\ell_1$ has the Schur property, so $\ell_F$ has the property ${\bf H}(\textit{weak})$.
     On the other hand, if $F(b_{F}) \geqslant 1$ we can just use \cite[Theorem~2.8]{HP97}.
	 \end{proof}
 
     Moreover, the characterization for the property ${\bf H}(\textit{point-wise})$ in Orlicz sequence spaces can be obtained from \cite[Corollary~13]{Kol12}.
     
     One more thing.
     Looking at the above table, one may get the impression that both properties ${\bf H}(\textit{weak})$ and ${\bf H}(\textit{point-wise})$ coincide
     (at least for classical sequence spaces).
     This is generally not true, as the following example shows ({\it cf}. \cite[Example~2]{HKL06}).
     
     \begin{example} \label{H_c silniejsze niz H}
     Take the Young function $F(t) = t$ for $0 \leqslant t \leqslant 1/2$ and $\varphi(t) = \infty$ for $t > 1/2$.
     By Theorem \ref{THEOREM : H in Orlicz sequence} we conclude that the space $\ell_F$ has the property ${\bf H}(\textit{weak})$
     (actually, since the space $\ell_F$ is nothing else but a certain re-norming of $\ell_1$, so $\ell_F$ has the Schur property).
     However, since $F(b_F) < 1$, so $\ell_F$ fails to have the property ${\bf H}(\textit{point-wise})$.
     \demo
     \end{example}
 
     {\bf No.~2: Lebesgue sequence spaces $\ell_p$.} Plainly, $\ell_p$'s are just particular examples of Orlicz sequence spaces.
     Hence, everything follows from what we said above.
      
     {\bf No.~3: Lorentz sequence spaces $d(w,p)$.} This is an immediate consequence of \cite[Corollary~13]{Kol12}. 
      
     {\bf No.~4: Nakano spaces $\ell_{\{ p_n \}}$.} Note that we can see the Nakano space $\ell_{\{ p_n \}}$ as a particular example of
     a much more general construction of the Musielak--Orlicz sequence space $\ell_{\mathcal{M}}$, where $\mathcal{M} = \{ \mathcal{M}_n \}_{n=1}^{\infty}$
     with $\mathcal{M}_n(t) = t^{p_n}$ for $n \in \mathbb{N}$.
     Knowing this, it is enough to apply Theorem~4.4 from \cite{FH99} together with the obvious fact the Musielak--Orlicz function $\mathcal{M}$
     satisfies the suitable $\delta_2$-condition if, and only if, $\sup_{n \in \mathbb{N}} p_n < \infty$ (we refer to \cite{FH99} for the
     details; {\it cf}. \cite[p.~170--172]{HKL06} and \cite[Definition~4.d.2, p.~167]{LT77}).
      
     {\bf No.~5: Ces{\' a}ro sequence spaces $ces_p$.} Note that the space $ces_p$ is order continuous if, and only if, $1 < p < \infty$
     (see, for example, \cite[Appendix]{KKT22} for a much more general result; {\it cf}. \cite{KT17}).
     Note also that the space $ces_1$ is trivial.
     Thus, this part can be deduced directly from Theorem~1 in \cite{FHS10}, where even more general situation is considered.

\section{Applications for the drop property and approximative compactness} \label{APPENDIX : B}

Here we will explain how our results about the property ${\bf H}(\textit{weak})$ can be applied in the approximation theory (see Theorem~\ref{THEOREM : approximative compactness}). 
This part requires two additional pieces of terminology which we will present now.
 
Recall that a {\bf drop} $D = D(x, \text{Ball}(X))$ induced by a given point $x \in X \setminus \text{Ball}\left( X \right) $
is defined to be the set 
\begin{equation*}
D\left( x,\text{Ball}\left( X \right) \right) \coloneqq \text{conv} \left( \{ x \} \cup \text{Ball}\left( X \right) \right).
\end{equation*}
A Banach space $X$ is said to have the {\bf drop property} if for each closed set $C$ that is disjoint with $\text{Ball}\left( X \right)$,
there is $x \in C$ with $D\left( x,\text{Ball}\left( X\right) \right) \cap C = \{ x \}$.
The notion of the \enquote{drop} has been introduced by Stefan Rolewicz in \cite{Rol87} basing on the so-called Drop Theorem proved earlier by Dane\v{s} in \cite{Dan72}.

Recall also that a nonempty subset $C$ of a Banach space $X$ is called {\bf approximatively compact} if for each sequence $\{ x_{n} \}_{n=1}^{\infty}$
from $C$ and every $y \in X$ satisfying
$\left\Vert y - x_{n} \right\Vert \rightarrow \text{dist}\left( y, C \right)$,
it follows that $\{ x_{n} \}_{n=1}^{\infty}$ has a Cauchy subsequence.
Hereby, $\text{dist}\left( y, C \right) \coloneqq \inf \left\{ \left\Vert y - x \right\Vert \colon x \in C \right\}$.
Plainly, the approximative compactness of a given set ensures the existence of an element of best approximation for any $x \in X$.
A Banach space $X$ is said be {\bf approximatively compact} if every non-empty, closed and convex set in $X$ is approximatively compact.
Moreover, it is worth to mention that if a Banach space $X$ is rotund and approximatively compact then, for each non-empty, convex and closed set $A$,
the metric projection $x \mapsto P_{A}\left( x \right)$, where $P_A(x) \coloneqq \{ y \in A \colon \text{dist}(x,A) = \norm{y - x} \}$,
is continuous (see \cite{HKL06} for more information and references).

It turns out that approximative compactness is equivalent to the combination of the Kadets--Klee property and reflexivity (see \cite[Theorem~3]{HKL06}).
Thus, it follows from \cite{Mon87} that approximative compactness actually coincides with the drop property.
This is the content of the following

\begin{theorem} \label{THEOREM : approximative compactness}
{\it Let $X$ be a Banach space. The following statements are equivalent}
\begin{itemize}
	\item[(1)] {\it $X$ has the drop property;}
	\item[(2)] {\it $X$ is approximatively compact;}
	\item[(3)] {\it $X$ has the propery ${\bf H}(\text{weak})$ and is reflexive.}
\end{itemize}
\end{theorem}

Let us also note the following simple

\begin{lemma} \label{LEMMA : direct sum reflexive}
{\it Let ${\mathcal{E}}$ be a Banach sequence space.
Further, let $\{ X_{\gamma} \}_{\gamma \in \Gamma}$ be a family of Banach spaces.
Then the space $( \bigoplus_{\gamma \in \Gamma} X_{\gamma} )_{\mathcal{E}}$ is reflexive if, and only if, ${\mathcal{E}}$ and all $X_{\gamma}$'s are reflexive.}
\end{lemma}
\begin{proof}
Since ${\mathcal{E}}$ and all $X_{\gamma}$'s with $\gamma \in \Gamma$ are isometrically embedded into $( \bigoplus_{\gamma \in \Gamma} X_{\gamma} )_{\mathcal{E}}$
and, as is widely known, is a three-space property, so the necessity is crystal clear.

Thus, let us focus on the sufficiency.
Suppose that ${\mathcal{E}}$ and all $X_{\gamma}$'s are reflexive.
Our goal is to show that the canonical embedding
$\kappa$ of $( \bigoplus_{\gamma \in \Gamma} X_{\gamma} )_{\mathcal{E}}$ into its bidual is surjective.
Since the space $\mathcal{E}$ is reflexive, so due to James's theorem the basis $\{ \boldsymbol{e}_{\gamma} \}_{\gamma \in \Gamma}$
is shrinking and boundedly complete (see, for example, \cite[Theorem~3.2.13, p.~58]{AK06} and \cite[Theorem~1.2.7, p.~12]{Lin04}).
Thus, it follows from \cite[Proposition~4.8]{Lau01} that the bidual of $( \bigoplus_{\gamma \in \Gamma} X_{\gamma} )_{\mathcal{E}}$
is naturally isometrically isomorphic to $( \bigoplus_{\gamma \in \Gamma} X_{\gamma}^{**} )_{\mathcal{E}}$.
(Alternatively, one can just use here Proposition~\ref{PROPOSITION : duality of direct sums} which is, perhaps, even more straightforward.)
Furthermore, since all $X_{\gamma}$'s are reflexive, so the canonical embeddings
$\kappa_{\gamma} \colon X_{\gamma} \rightarrow X_{\gamma}^{**}$ are all surjective.
Putting these two facts together, it is immediate to see that the canonical embedding $\kappa$ is surjective as well.
\end{proof}

The moral of this short story is:
Knowing Theorem~\ref{THEOREM : approximative compactness} and remembering about Lemma~\ref{LEMMA : direct sum reflexive},
we can use Theorem~\ref{THM : H weak} along with the characterizations from Section~\ref{APPENDIX : examples} to produce
a number of results about approximation properties of direct sums.
We leave the straightforward details to interested readers.

\section{A few bibliographical notes} \label{APPENDIX : C}

Sometimes, the Kadets--Klee property ${\bf H}$ is also called the {\bf Radon--Riesz property}.
The reason for this is the fact that Johann Radon and, independently, the elder of the Riesz brothers, Frigyes, showed that the space
$L_p(\Omega,\Sigma,\mu)$ with $1 < p < \infty$ enjoy it (see \cite{Rad13} and \cite{Rie28/29a}).

Moreover, this property was strained by Mikhail Kadets \cite{Kad59} in his proof that every separable Banach space admits an equivalent locally
uniformly rotund norm (note also that sometimes his surname is transliterated as \enquote{Kadec}; see \cite{KMW} for more informations).

It also seems worth mentioning that implicitly the Kadets--Klee property appear in the 1939 paper of Vitold {\v S}mulian \cite{Smu39} and,
most likely independent, work of Rudolf V{\' y}born{\' y} \cite{Vyb56} (precisely, see \cite[Theorem~5]{Smu39} and \cite[p.~352]{Vyb56}, respectively).


\begin{thebibliography}{100000000}
		
		\bibitem[AA22]{AA22} F. Albiac and J. L. Ansorena, {\it Uniqueness of unconditional basis of infinite-direct sums of quasi-Banach spaces},
		Positivity {\bf 26} (2022), article number 35.
		
		\bibitem[AK06]{AK06} F. Albiac and N. J. Kalton, {\it Topics in Banach Space Theory}, Springer-Verlag, New York 2006.
		
		\bibitem[AB78]{AB03} C. D. Aliprantis and O. Burkinshaw, {\it Locally Solid Riesz Spaces}, Academic Press, New York, 1978.
		
		\bibitem[BDD$^{+}$94]{BDDL94} M. Bebes, S. J. Dilworth, P. N. Dowling and C. J. Lennard, {\it New convexity and fixed point properties in Hardy and Lebesgue--Bochner spaces}, 
		J. Funct. Anal. {\bf 119} (1994), 340--357.
		
		\bibitem[BS88]{BS88} C. Bennett and R. Sharpley, {\it Interpolation of Operators}, Academic Press, Boston, 1988.
		
		\bibitem[Boa40]{Boa40} R. P. Boas, Jr., {\it Some uniformly convex spaces}, Bull. Amer. Math. Soc. {\bf 46} (1940), 304--311.
		
		\bibitem[CP96]{CP96} C. Castaing and R. P{\l}uciennik, {\it Property $(H)$ in K{\" o}the--Bochner spaces},
		Indag. Mathem. N.S. {\bf 7} (1996), no. 4, 447--459.
		
		\bibitem[CM97]{CM97} P. Cembranos and J. Mendoza, {\it Banach spaces of vector-valued functions}, Lectures Notes in Math. {\bf 1676}, Springer-Verlag, Berlin, 1997.
          	
		\bibitem[CHK$^{+}$98]{CHKM98} J. Cerd{\` a}, H. Hudzik, A. Kami{\' n}ska and M. Masty{\l}o,
		{\it Geometric properties of symmetric spaces with applications to Orlicz--Lorentz spaces},
		Positivity {\bf 2} (1998), no. 2, 311--337.

         \bibitem[CHM95]{CHM95} J. Cerd{\` a}, H. Hudzik and M. Masty{\l}o, {\it On the geometry of some Calder\'on-Lozanovski\u{\i} interpolation spaces},
		Indag. Mathem. {\bf 6} (1995), no. 1, 35--49.
		
		\bibitem[CHM96]{CHM96} J. Cerd{\` a}, H. Hudzik and M. Masty{\l}o, {\it Geometric properties in K{\" o}the--Bochner spaces},
		Math. Proc. Camb. Phil. Soc. {\bf 120} (1996), 521--533.

        \bibitem[Che96]{Ch96} S. T. Chen, {\it Geometry of Orlicz spaces}, Dissertationes Math. (Rozprawy Mat.) {\bf 356} (1996), 204 pp.
        		
		\bibitem[CDS$^{+}$96]{CDSS96} V. I. Chillin, P. G. Dodds, A. A. Sadaev and F. A. Sukochev, {\it Characterization of Kadec--Klee properties in symmetric spaces of measurable functions},
		Trans. Amer. Math. Soc. {\bf 348} (1996), 4895--4918.
		
		\bibitem[CKP15]{CKP15} M. Ciesielski, P. Kolwicz and R. P{\l}uciennik, 	{\it Local approach to Kadec--Klee properties in symmetric function spaces},
		J. Math. Anal. Appl. {\bf 426} (2015), 700--726.
		
		\bibitem[Con19]{Con19} J. Conradie, {\it Lebesgue topologies and mixed topologies}, in: Positivity and Noncommutative Analysis,
		Festschrift in Honour of Ben de Pagter on the Occasion of his 65th Birthday, Trends in Mathematics, Birkh{\" a}user, 2019.
		
		\bibitem[Coo78]{Coo78} J. B. Cooper, {\it Saks Spaces and Applications to Functional Analysis}, North-Holland Math. Studies {\bf 28} (1987), pp.~1--74 + vii.
		
		\bibitem[Day41]{Day41} M. M. Day, {\it Some more uniformly convex spaces}, Bull. Amer. Math. Soc. {\bf 47} (1941), 504--507.

        \bibitem[Dan72]{Dan72} J. Dane\v{s}, {\it A geometric theorem useful in nonlinear functional analysis}, Boll. Un. Mat. Ital. {\bf 6} (1972), 369--375.
		
		\bibitem[DD67]{DD67} W. J. Davis and D. W. Dean, {\it The direct sum of Banach spaces with respect to a basis}, Studia Math. {\bf 28} (1967), 209--219.
		
		\bibitem[DK99]{DK99} S. J. Dilworth and D. Kutzarova, {\it Kadec--Klee properties for $L(\ell_p,\ell_q)$}, Lecture Notes in Pure and Appl. Math. {\bf 172} (1999), 71--83.
		
		\bibitem[DDS$^{+}$04]{DDSS04} P. G. Dodds, T. K. Dodds, A. A. Sedaev and F. A. Sukochev, {\it Local uniform convexity and Kadec--Klee type properties
		in $K$-interpolation spaces I: General Theory}, J. Funct. Spaces {\bf 2} (2004), no. 2, 125--173.
		
		\bibitem[DDS$^{+}$04a]{DDSS04a} P. G. Dodds, T. K. Dodds, A. A. Sedaev and F. A. Sukochev,
		{\it Local uniform convexity and Kadec--Klee type properties in $K$-interpolation spaces II},
		J. Funct. Spaces {\bf 2} (2004), no. 3, 323--356.
		
		\bibitem[DHL$^{+}$03]{DHLMS03} T. Dominguez, H. Hudzik, G. Lopez, M. Masty{\l}o and B. Sims, {\it Complete characterization of Kadec--Klee properties in Orlicz spaces},
		Houston J. Math. {\bf 29} (2003), no. 4, 1027--1044.
		
		\bibitem[DPS07]{DPS07} P. N. Dowling, S. Photi and S. Saejung, {\it Kadec--Klee and related properties of direct sums of Banach spaces},
		J. Nonlinear Conv. Anal. {\bf 8} (2007), no. 3, 463--469.
		
		\bibitem[DK16]{DK16} S. Draga and T. Kochanek, {\it Direct sums and summability of the Szlenk index}, J. Funct. Anal. {\bf 271} (2016), 642--671.
		
		\bibitem[DV86]{DV86} D. van Dulst and V. de Valk, {\it $(KK)$ properties, normal structure and fixed points of nonexpansive mappings in Orlicz sequence spaces},
		Canad. J. Math. {\bf 38} (1986), no. 3, 728--750.

        \bibitem[FH99]{FH99} P. Foralewski and H. Hudzik, {\it On some geometrical and topological properties of generalized Calder\'on-Lozanovski\i\u{i} sequence spaces},
        Houston J. Math. {\bf 25}, no. 3 (1999), 523-542.
         
        \bibitem[FHS10]{FHS10} P. Foralewski, H. Hudzik and A. Szymaszkiewicz, {\it Some remarks on Ces{\' a}ro--Orlicz sequence spaces}, Math. Ineq. Appl. {\bf 13} (2010), no. 2, 363--386.
		
		\bibitem[Fre01]{Fre01} D. H. Fremlin {\it Measure Theory}, Vol. 1--4, Torres Fremlin, 2001.
		
		\bibitem[GKP16]{GKP16} S. Gabriyelyan, J. K{\k a}kol and G. Plebanek, {\it The Ascoli property for function spaces and the weak topology of Banach and Fr{\' e}chet spaces},
		Studia Math. {\bf 233} (2016), 119--139.
		
		\bibitem[Gro73]{Gro73} A. Grothendieck, {\it Topological Vector Spaces}, Gordon and Breach, Science Publishers, New York 1973.
		
		\bibitem[GP05]{GP05} J. B. Guerrero and A. M. Peralta, {\it The Dunford--Pettis and the Kadec--Klee properties on tensor products of $\text{JB}^{*}$-triples},
		Math. Z. {\bf 251} (2005), no. 1, 117--130.
		
		\bibitem[Gre69]{Gre69} D. A. Gregory, {\it Some basic properties of vector sequence spaces}, J. Reine Angew. Math. {\bf 237} (1969), 26--38.
		
		\bibitem[HKM00]{HKM00} H. Hudzik, K. Kami{\' n}ska and M. Masty{\l}o, {\it Monotonicity and rotundity properties in Banach lattices}, Rocky Mountain J. Math. {\bf 30} (2000), no. 3, 933--950.

        \bibitem[HKL06]{HKL06} H. Hudzik, W. Kowalewski and G. Lewicki, {\it Approximate compactness and full rotundity in Musielak--Orlicz spaces and Lorentz--Orlicz spaces},
        Z. Anal. Anwend. {\bf 25} (2006), no. 2, 163--192.
		
		\bibitem[HL92]{HL92} H. Hudzik and R. Landes, {\it Characteristic of convexity in K{\" o}the function spaces}, Math. Ann. {\bf 294} (1992), 117--124.

        \bibitem[HP97]{HP97} H. Hudzik and D. Pallaschke, {\it On some convexity properties of Orlicz sequence spaces equipped with the Luxemburg norm}, Math. Nachr. {\bf 186} (1997), 167-185.
		
		\bibitem[Huf80]{Huf80} R. Huff, {\it Banach spaces which are nearly uniformly convex}, Rocky Mountain J. Math. {\bf 10} (1980), 743--749.
		
		\bibitem[HLR91]{HLR91} R. Hydon, R. Levy and Y. Raynaud, {\it Randomly Normed Spaces}, in: {\it Travauxen Cours} (Works in Progress), vol. 41, Springer, Berlin 1991.
		
		\bibitem[Kad59]{Kad59} M. I. Kadets,
		{\it Spaces isomorphic to a locally uniformly convex space},
		Izv. Vyssh. Uchebn. Zaved. Mat. {\bf 1959} (1959), no. 6, 51--57 ({\bf in Russian}).
		
		\bibitem[Jar81]{Jar81} H. Jarchow, {\it Locally Convex Spaces}, Mathematische Leitf{\" a}den, B. G. Teubner, Stuttgart 1981.
		
		\bibitem[Kadets]{KMW} Mikhail Iosifovich Kadets (1923--2011) memorial website, \href{https:/https://kadets.inf.ua}{\texttt{kadets.inf.ua}}.
		
		\bibitem[Kal03]{Kal03} N. J. Kalton, {\it Quasi-Banach spaces}, in: Handbook of the Geometry of Banach Spaces, Vol. 2, pp. 1099--1130, North-Holland, Amsterdam 2003.
		
		\bibitem[KA82]{KA82} L. V. Kantorovich and G. P. Akilov, {\it Functional analysis, 2nd ed.}, Pergamon Press, Oxford-Elmsford, N.Y., 1982.
		Translated from the Russian by Howard L. Silcock.
  
        \bibitem[KKT22]{KKT22} T. Kiwerski, P. Kolwicz and J. Tomaszewski, {\it Quotients, $\ell_{\infty}$ and abstract Ces{\` a}ro spaces},
        Rev. R. Acad. Cienc. Exactas F{\' i}s. Nat. Ser. A-Mat. (RACSAM) {\bf 116} (2022), no. 3, paper no. 131, 29 pp.
  		
  		\bibitem[KT24]{KT24}  T. Kiwerski and J. Tomaszewski, {\it Arithmetic, interpolation and factorization of amalgams},
  		preprint available on \href{https://arxiv.org/abs/2401.05526}{\tt arxiv.org/abs/2401.05526}, January 2024.
   
		\bibitem[KT17]{KT17}  T. Kiwerski and J. Tomaszewski, {\it Local approach to order continuity in Ces{\` a}ro function spaces}, 	J. Math. Anal. Appl. {\bf 455} (2017), no. 2, 1636--1654.
		
		\bibitem[Kol03]{Kol03} P. Kolwicz, {\it Uniform Kadec--Klee property and nearly uniform convexity in K{\" o}the--Bochner sequence spaces},
		Boll. Unione Mat. Ital. {\bf 8} (2003), no. 1, 221--235.

        \bibitem[Kol12]{Kol12} P. Kolwicz, {\it Kadec--Klee properties of Calder\'on-Lozanovski\u{\i} sequence spaces}, Collect. Math. {\bf 63} (2012), 45--58.
		
		\bibitem[Kol18]{Kol18} P. Kolwicz, {\it Kadec--Klee properties of some quasi-Banach function spaces}, Positivity {\bf 22} (2016), no. 4, 983--1013.
		
		\bibitem[KP97]{KP97} D. Krassowska and R. P{\l}uciennik, {\it A note on property $(H)$ in K{\" o}the--Bochner sequence spaces}, Math. Japon. {\bf 46} (1997), 407--412.
		
		\bibitem[KL92]{KL92} D. Kutzarova and T. Landes, {\it Nearly uniform convexity of infinite direct sums}, Indiana Univ. Math. {\bf 41} (1992), 915--926.
		
		
		\bibitem[Lab87]{Lab87} I. Labuda, {\it Submeasures and locally solid topologies on Riesz spaces}, Math. Z. {\bf 195} (1987), no. 2, 179--196.
		
		\bibitem[Lan65]{Lan65} S. Lang, {\it Algebra}, Addison-Wesley, 1965.
		
		\bibitem[Lau01]{Lau01} N. J. Laustsen, {\it Matrix multiplication operators and composition operators of operators on the direct sum of an infinite sequence of Banach spaces}, 
		Math. Proc. Camb. Phil. Soc. {\bf 131} (2001), no. 1, 165--183.
		
		\bibitem[Len88]{Len88} C. Lennard, {\it Operators and geometry of Banach spaces}, Ph.D. Thesis, Kent State University, 1988.
		
		\bibitem[Len91]{Len91} C. Lennard, {\it A new convexity property that implies a fixed point property for $L_1$}, Studia Math. {\bf 100} (1991), 95--108.
		
		\bibitem[Leo76]{Leo76} I. E. Leonard, {\it Banach sequence spaces}, J. Math. Anal. Appl. {\bf 54} (1976), 245--265.
		
		\bibitem[Lin04]{Lin04} P.-K. Lin, {\it K{\" o}the--Bochner Function Spaces}, Birkh{\" a}user, Boston 2004.
		
		\bibitem[LL85]{LL85} B.-L. Lin and P.-K. Lin, {\it Property $(H)$ in Lebesgue--Bochner function spaces}, Proc. Amer. Math. Soc. {\bf 95} (1985), no. 4, 581--584.
		
		\bibitem[LT77]{LT77} J. Lindenstrauss and L. Tzafriri, {\it Classical Banach Spaces I. Sequence Spaces}, Springer-Verlag, Berlin-New York 1977.
		
		\bibitem[LT79]{LT79} J. Lindenstrauss and L. Tzafriri, {\it Classical Banach Spaces II. Function Spaces}, Springer-Verlag, Berlin-New York 1979.
		
		\bibitem[LZ65-6]{LZ65} W. A. J. Luxemburg and A. C. Zaanen, {\it Some examples of normed K{\" o}the spaces}, Math. Ann. {\bf 162} (1965/1966), 337--350.
		
		\bibitem[Mal89]{Mal89} L. Maligranda, {\it Orlicz Spaces and Interpolation}, Sem. Mat. 5, Universidade Estadual de Campinas, Dep. Mat., Campinas 1989.
		
		\bibitem[MP22]{MP22} J. Markowicz and S. Prus, {\it Uniform convexity of general direct sums and interpolation spaces}, J. Topology and Appl. {\bf 14} (2022), no. 4, 1001--1013.
		
		\bibitem[MS81]{MS81} A. Medzitov and P. Sukochev, {\it The property $(H)$ in Orlicz spaces}, Bull. Acad. Polon. Sci. S{\' e}r. Math. {\bf 40} (1981), no. 3--4, 137--144.
		
		\bibitem[Meg98]{Meg98} R. E. Megginson, {\it An Introduction to Banach Space Theory}, Graduate Text in Mathematics, Vol. 183, 1998.

        \bibitem[Mon87]{Mon87} V. Montesinos, {\it Drop property equals reflexivity}, Studia Math. {\bf 87} (1987), 93-100.
		
		\bibitem[MN91]{MN91} P. Meyer-Nieberg, {\it Banach Lattices}, Springer-Verlag, Berlin 1991.
		
		\bibitem[Now07]{Now07} M. Nowak, {\it On some topological properties of vector-valued function spaces}, Rocky Mountain J. Math. {\bf 37} (2007), no. 3, 917--945.
		
		\bibitem[Ord66]{Ord66} E. T. Ordman, {\it Convergence almost everywhere is not topological}, Amer. Math. Monthly {\bf 73} (1966), no. 2, 182--183.

		\bibitem[Pe{\l}60]{Pel60} A. Pe{\l}czy{\' n}ski, {\it Projections in certain Banach spaces}, Studia Math. {\bf 19} (1960), 209--228.
		
		\bibitem[Pis78]{Pis78} G. Pisier, {\it Une propri{\' e}t{\' e} de stabilit{\' e} de la classe des espaces ne contenent $\ell_1$},
		C. R. Acad. Sci. Paris {\bf 284} (1974), 747--749 ({\bf in French}).
		
		\bibitem[Rad13]{Rad13} J. Radon, {\it Theorie und Anwendungen der absolut additiven Mengenfunctionen}, Sitz. Akad. Wiss. Wien {\bf 122} (1913), 1295--1438 ({\bf in German}).
		
		\bibitem[Rie28/29]{Rie28/29a} F. Riesz, {\it Sur la convergence en moyenne I, II}, Acta Sci. Math. {\bf 4} (1928/1929), 58--64 ({\bf in French}).

        \bibitem[Rol87]{Rol87} S. Rolewicz, {\it On drop property}, Studia Math. {\bf 85} (1987), 27--35.
		
		\bibitem[Ros78]{Ros78} H. P. Rosenthal, {\it Some recent discoveries in the isomorphic theory of Banach spaces}, Bull. Amer. Math. Soc. {\bf 84} (1978), 803--831.
		
		\bibitem[San23]{San23} V. Sangeetha, {\it On $k$-rotundity and $k$-uniform rotundity in direct sums of normed spaces}, Bull. Sci. Math. {\bf 189} (2023), Paper No. 103346.
		
		\bibitem[ST80]{ST80} M. A. Smith and B. Turett, {\it Rotundity in Lebesgue--Bochner function spaces}, Trans. Amer. Math. Soc. {\bf 257} (1980), no. 1, 105--118.
		
		\bibitem[{\v S}mu39]{Smu39} V. {\v S}mulian, {\it On some geometrical properties of the sphere in a space of the type $(B)$}, Mat. Sbornik N.S. {\bf 48} (1939),
		no. 1, 77--94 (\textbf{in Russian}); English translation in: Dokl. Akad. Nauk SSSR {\bf 24} (1939), no. 7, 648--652.
		
		\bibitem[Suk95]{Suk95} F. A. Sukochev, {\it On the uniform Kadec--Klee property with respect to convergence in measure}, J. Aust. Math. Soc. {\bf 59} (1995), 343--352.
		

		\bibitem[V{\' y}b56]{Vyb56} R. V{\' y}born{\' y}, {\it O slab{\' e} konvergenci v prostorech lok{\' a}ln{\v e} stejnom{\v e}rn{\v e} konvexn{\' i}ch},
		{\v C}as. p{\v e}st. mat. {\bf 81} (1956), no. 3, 352--353 (\textbf{in Czech}).
		
		\bibitem[Wiw61]{Wiw61} A. Wiweger, {\it Linear spaces with mixed topology}, Studia Math. {\bf 20} (1961), 47--68.
		
		\bibitem[Wnu99]{Wn99} W. Wnuk, {\it Banach Lattices with Order Continuous Norms}, Polish Scientific Publishers PWN, Warszawa 1999.
		
	\end{thebibliography}
\end{document}